\documentclass[reqno,11pt]{amsart}
\PassOptionsToPackage{svgnames}{xcolor}
\usepackage[object=vectorian]{pgfornament}
\usepackage{a4wide}
\usepackage{color,eucal,enumerate,mathrsfs}
\usepackage[normalem]{ulem}
\usepackage{mdwlist}
\usepackage{mathtools}
\usepackage{tikz-cd}
\usepackage{amsmath}
\usepackage{esint}
\usepackage{amssymb,epsfig,bbm}
\numberwithin{equation}{section}
\usepackage[colorlinks,citecolor=green,linkcolor=red]{hyperref}
\usepackage{cite}

\makeatletter
\@namedef{subjclassname@2020}{%
  \textup{2020} Mathematics Subject Classification}
\makeatother

\newcommand{\sectionlinetwo}[2]{
\nointerlineskip\vspace{.5\baselineskip}\hspace{\fill}
{\resizebox{0.5\linewidth}{1.2ex}
{\pgfornament[color = #1]{#2}}}
\hspace{\fill}
\par\nointerlineskip\vspace{.5\baselineskip}}

\usepackage[latin1]{inputenc}

\newcommand{\B}{\mathbb{B}}
\newcommand{\N}{\mathbb{N}}
\newcommand{\Q}{\mathbb{Q}}
\newcommand{\R}{\mathbb{R}}
\newcommand{\sfd}{{\sf d}}
\newcommand{\sfC}{{\sf C}}
\newcommand{\sfR}{{\sf R}}
\renewcommand{\d}{{\mathrm d}}
\newcommand{\la}{{\langle}}
\newcommand{\ra}{{\rangle}}
\newcommand{\eps}{\varepsilon}  
\newcommand{\nchi}{{\raise.3ex\hbox{$\chi$}}}
\newcommand{\mto}{\twoheadrightarrow}
\newcommand{\limi}{\varliminf}
\newcommand{\lims}{\varlimsup}
\newcommand{\fr}{\penalty-20\null\hfill$\blacksquare$}
\newcommand{\X}{{\rm X}}
\newcommand{\Y}{{\rm Y}}
\newcommand{\nnorm}{\boldsymbol{\sf n}}
\newcommand{\mm}{\mathfrak m}

\newtheorem{theorem}{Theorem}[section]
\newtheorem{corollary}[theorem]{Corollary}
\newtheorem{lemma}[theorem]{Lemma}
\newtheorem{proposition}[theorem]{Proposition}
\newtheorem{problem}[theorem]{Problem}

\newtheorem{definition}[theorem]{Definition}

\newtheorem{example}[theorem]{Example}
\newtheorem{remark}[theorem]{Remark}
\title{Representation theorems for normed modules}

\author{Simone Di Marino}
\address{Universit\`a di Genova (DIMA), MaLGa, Via Dodecaneso 35, 16146 Genova, Italy}
\email{simone.dimarino@unige.it}

\author{Danka Lu\v{c}i\'{c}}
\address{Universit\`{a} di Pisa, Dipartimento di Matematica,
Largo Bruno Pontecorvo 5, 56127 Pisa, Italy}
\email{danka.lucic@dm.unipi.it}

\author{Enrico Pasqualetto}
\address{Scuola Normale Superiore, Piazza dei Cavalieri 7,
56126 Pisa, Italy}
\email{enrico.pasqualetto@sns.it}
\begin{document}
\date{\today}
\begin{abstract}
In this paper we study the structure theory of normed modules,
which have been introduced by Gigli. The aim is twofold: to
extend von Neumann's theory of liftings to the framework of
normed modules, thus providing a notion of precise representative
of their elements; to prove that each separable normed module can
be represented as the space of sections of a measurable Banach bundle.
By combining our representation result with Gigli's differential structure,
we eventually show that every metric measure space (whose Sobolev space
is separable) is associated with a cotangent bundle in a canonical way.
\end{abstract}
\maketitle
\tableofcontents
\section{Introduction}
\addtocontents{toc}{\protect\setcounter{tocdepth}{1}}
In recent years, a great deal of attention has been devoted to
the study of weakly differentiable objects on abstract metric measure
spaces. In this regard, an important contribution is represented by
N.\ Gigli's book \cite{Gigli14}, where he proposed a first-order
differential calculus tailored to this framework. A key role was played
by the notion of \emph{normed \(L^0\)-module}, which has been
subsequently refined in \cite{Gigli17}; such notion is a variant
of a similar concept introduced by N.\ Weaver in \cite{Weaver01},
who was in turn inspired by J.-L.\ Sauvageot's papers
\cite{Sauvageot89,Sauvageot90}. A strictly related notion is that
of `randomly normed space', which was extensively studied in \cite{HLR91}.
\bigskip

Let \((\X,\sfd,\mm)\) be a metric measure space. We denote by
\(L^0(\mm)\) the commutative ring of real-valued Borel functions
on \(\X\), quotiented up to \(\mm\)-a.e.\ equality. Then a
\emph{normed \(L^0(\mm)\)-module} is an algebraic \(L^0(\mm)\)-module
\(\mathscr M\) equipped with a \emph{pointwise norm} operator
\(|\cdot|\colon\mathscr M\to L^0(\mm)\), which induces a complete
distance \(\sfd_{\mathscr M}\); we refer to Definition
\ref{def:normed_L0_mod} for the precise definition.

Roughly speaking, a normed \(L^0(\mm)\)-module can be thought of as
the space of measurable sections (up to \(\mm\)-a.e.\ equality) of
some notion of measurable Banach bundle. Nevertheless, this interpretation
might be not entirely correct, since it is currently unknown whether
every normed \(L^0(\mm)\)-module actually admits this sort of
representation. The main purpose of the present paper is, in fact,
to show that all those normed \(L^0(\mm)\)-modules \(\mathscr M\)
for which the distance \(\sfd_{\mathscr M}\) is separable can be
written as spaces of sections of a separable Banach bundle.
The result generalises a previous theorem, obtained by the second
and third named authors in \cite{LP18}, for normed \(L^0(\mm)\)-modules
that are `locally finitely-generated', in a suitable sense.
\bigskip

We now pass to a more detailed description of the contents of
this manuscript. We point out that, even though in the rest of
this introduction we will just focus (for simplicity) on the case
of a metric measure space, most of the results can be formulated
and proven in the more general framework of \(\sigma\)-finite measure
spaces, as we will see later on in the paper.
\subsection*{Liftings of normed modules}
A propotypical example of normed \(L^0(\mm)\)-module is the space
\(L^0(\mm)\) itself, which is generated by the function constantly equal
to \(1\). Its key feature is that it is possible to `take
precise representatives' of its elements, by just considering Borel
versions. A similar property is -- a priori, at least -- not shared
by all normed \(L^0(\mm)\)-modules, which are intrinsically defined
in the \(\mm\)-a.e.\ sense. This non-trivial issue needs to be addressed
in order to be able to provide a representation of (separable) normed
\(L^0(\mm)\)-modules as spaces of sections of measurable Banach bundles,
since the latter certainly have this sort of property. The whole
\S\ref{s:lift_norm_mod} is dedicated to achieve such a goal, as we are
now going to describe.
\bigskip

We denote by \(\mathcal L^\infty(\Sigma)\) the space of bounded,
\(\Sigma\)-measurable, real-valued functions on \(\X\), where
\(\Sigma\) stands for the completion of the Borel \(\sigma\)-algebra
\(\mathscr B(\X)\). Then there exist linear continuous mappings
\(\mathcal L\colon L^\infty(\mm)\to\mathcal L^\infty(\Sigma)\),
called \emph{liftings}, which preserve products and select
\(\mm\)-a.e.\ representatives; this is the statement of a highly
non-trivial result by von Neumann--Maharam
(cf.\ \S\ref{ss:liftings_fcs} below). A natural question arises:
\begin{equation}\label{eq:lift_quest}
\text{Can we generalise von Neumann--Maharam's theorem to normed modules?}
\end{equation}
For technical reasons -- namely, due to the fact
that von Neumann--Maharam's liftings cannot be defined on \(L^0(\mm)\),
according to \cite{SMM02} -- we need to work with normed modules over
\(L^\infty(\mm)\) and \(\mathcal L^\infty(\Sigma)\). The former have
been introduced in \cite{Gigli14}, but essentially never used
nor studied; we will investigate their properties in
\S\ref{ss:norm_Linftym_mod}. The latter will be introduced and
studied in \S\ref{ss:norm_LinftySigma_mod}.
With these tools at our disposal, we will prove (in Theorem
\ref{thm:lifting_mod}) that, given any normed \(L^\infty(\mm)\)-module
\(\mathscr M\), there exist a normed \(\mathcal L^\infty(\Sigma)\)-module
\(\bar{\mathscr M}\) and a \emph{lifting} map
\(\mathscr L\colon\mathscr M\to\bar{\mathscr M}\).
(Notice that the space \(\bar{\mathscr M}\) is not given a priori, but its
existence is part of the statement.) Given that the elements of
\(\bar{\mathscr M}\) are `everywhere defined', the above result shows
that it is possible to select precise representatives of the elements
of \(\mathscr M\), thus providing a positive answer to the question
\eqref{eq:lift_quest}; the fact of working with normed
\(L^\infty(\mm)\)-modules rather than normed \(L^0(\mm)\)-modules is
harmless, thanks to Lemma \ref{lem:C=inverse_R}. Observe that liftings
exist on \emph{every} normed \(L^\infty(\mm)\)-module, not just on the ones induced by a
separable \(L^0(\mm)\)-module.
Moreover, given any \(x\in\X\), we might consider
the \emph{fiber}
\(\bar{\mathscr M}_x\coloneqq\nchi_{\{x\}}\cdot\bar{\mathscr M}\) of \(\bar{\mathscr M}\) at \(x\),
which turns out to be a Banach space. The properties of these fibers --
that we will study in \S\ref{ss:fibers} -- play an important
role in \S\ref{s:sep_Ban_bundle}, in the proof of the representation
theorem for separable \(L^0(\mm)\)-modules.
\subsection*{Representation theorems for normed modules}
First of all, we propose in \S\ref{ss:def_sep_Ban_bundle} a notion
of separable Banach bundle: if \(\B\) is a given universal separable Banach space -- \emph{i.e.}, wherein all separable Banach spaces can be
embedded linearly and isometrically -- then we define the
\emph{separable Banach \(\B\)-bundles} over \(\X\) as those
weakly measurable (multi-valued) mappings that associate to any
\(x\in\X\) a closed linear subspace of \(\B\); see Definition
\ref{def:sep_BBbundle}.

Given a separable Banach \(\B\)-bundle \(\textbf E\),
there is a natural way to define the space \(\Gamma({\textbf E})\)
of its \(L^0(\mm)\)-sections; see Definition \ref{def:sect_sep_Bb}
and \eqref{eq:def_sect_E}. Moreover, it is straightforward to check
that the space \(\Gamma({\textbf E})\) is a separable normed
\(L^0(\mm)\)-module, cf.\ Remark \ref{rmk:bar_Gamma_mod} and
Lemma \ref{lem:Gamma(E)_separable}.
\bigskip

On the other hand, the difficult task is actually to prove that every
separable normed \(L^0(\mm)\)-module \(\mathscr M\) is isomorphic to
the space of measurable sections \(\Gamma({\textbf E})\) of some separable
Banach \(\B\)-bundle \(\textbf E\) over \(\X\). Let us briefly outline
the strategy that we will adopt:
\begin{itemize}
\item Consider the space \({\sf R}(\mathscr M)\) of all bounded
elements of \(\mathscr M\), which is a normed \(L^\infty(\mm)\)-module.
Fix a sequence \((v_n)_n\subseteq{\sf R}(\mathscr M)\) generating
\(\mathscr M\) and a lifting
\(\mathscr L\colon{\sf R}(\mathscr M)\to\bar{\mathscr M}\).
\item For any \(x\in\X\) and \(n\in\N\), we can evaluate
\(\mathscr L(v_n)\) at \(x\), thus obtaining an element
\(\mathscr L(v_n)_x\) of \(\bar{\mathscr M}_x\). The closure \(E(x)\)
of \(\big\{\mathscr L(v_n)_x\,:\,n\in\N\big\}\) is a Banach subspace
of \(\bar{\mathscr M}_x\).
\item The resulting family \(\big\{E(x)\big\}_{x\in\X}\) is a
\emph{measurable collection of separable Banach spaces}; namely,
even if the spaces \(E(x)\) live in different fibers \(\bar{\mathscr M}_x\),
they depend on \(x\) in a measurable way, in a
weak sense. This working definition is given in Definition
\ref{def:meas_coll_Banach}.
\item The key step is to embed all the spaces \(E(x)\) in the
ambient space \(\B\) in a measurable fashion, thus obtaining a separable
Banach \(\B\)-bundle \(\textbf E\). This technical result, achieved
in Proposition \ref{thm:meas_family_embed} and Corollary
\ref{cor:img_Ix_bundle}, follows from a careful analysis of
Banach--Mazur theorem (cf.\ \S\ref{ss:embeddings}),
which states that \(C([0,1])\) is a universal separable Banach space.
\item The only remaining fact to check is that
\(\Gamma({\mathbf E})\) is isomorphic to \(\mathscr M\); cf.\ Theorem
\ref{thm:representation}.
\end{itemize}
Once we have proven that each separable normed \(L^0(\mm)\)-module is
the space of sections of some separable Banach \(\B\)-bundle, it is
natural to investigate the relation between bundles and modules
at the level of categories. It turns out that the functor associating
to any bundle its module of sections is an equivalence of categories:
in analogy with the terminology used in \cite{LP18}, we refer to this
result as the \emph{Serre--Swan theorem}; see Theorem \ref{thm:Serre-Swan}.
In \S\ref{s:pullback_bundle}, we will introduce the concept of
\emph{pullback bundle} and investigate its relation with the
pullback of a normed module, a notion extensively used
in the field of nonsmooth differential geometry.
\bigskip

The duality bundle-module described above resembles -- and somehow
extends -- the theory of direct integrals of Hilbert spaces
\cite{Takesaki79}. However, it is worth pointing out that our results
provide a full correspondence between the `horizontal' notion of normed
module and the `(purely) vertical' notion of Banach bundle, in the
sense that we do not need to require the existence of a countable dense
family of sections in order to obtain the Serre--Swan theorem.

At a late stage of development of the present paper, we discovered
that a comprehensive theory of measurable Banach bundles has been
thoroughly studied by A.E.\ Gutman (see, \emph{e.g.}, the paper
\cite{Gutman93}). It would be definitely interesting -- but outside
the scopes of this manuscript -- to investigate the relations between
our approach and Gutman's one.
\bigskip

Finally, we will propose in \S\ref{s:repr_via_DL} an alternative
(but less descriptive) method to represent separable normed
\(L^0(\mm)\)-modules as spaces of sections of some bundle.
More precisely, the idea is the following: since separable
normed \(L^0(\mm)\)-modules can be approximated by finitely-generated
ones, a representation theorem can be achieved by
using the variant of Serre--Swan theorem for finitely-generated
modules that has been proven in \cite{LP18}.
See Theorem \ref{thm:Gamma_surj_obj} below.
\subsection*{A notion of cotangent bundle on metric measure spaces}
As already mentioned at the beginning of this introduction,
the interest towards the theory of normed modules was mainly
motivated by the development of a significant measure-theoretical
tensor calculus on metric measure spaces. Although in this paper
we just focus on `abstract' normed modules, let us spend a
few words about some possible applications of our results.
\bigskip

Several (essentially) equivalent concepts of Sobolev space
on \((\X,\sfd,\mm)\) were introduced and studied in
the last two decades, cf.\ \cite{Cheeger00,Shanmugalingam00,
AmbrosioGigliSavare11,Ambrosio-DiMarino14}. A common point of
all these approaches is that each Sobolev function \(f\) is
associated with a minimal object \(|Df|\), which behaves more
like the modulus of some weak differential of \(f\) than the
differential itself. Due to this reason, Gigli proposed the
theory of normed modules with the aim of providing a linear
structure underlying the Sobolev calculus; in a few words,
the idea was to define the differential \(\d f\)
rather than the modulus of the differential. In this regard,
the key object is the so-called \emph{cotangent module}
\(L^0({\rm T}^*\X)\), which can be thought of as the space of
measurable \(1\)-forms and contains the differentials
of all Sobolev functions. The module \(L^0({\rm T}^*\X)\) has a rich
and flexible functional-analytic structure, whence it plays a
central role in many works (\emph{e.g.}, \cite{GPS18,GT20}).
\bigskip

In \S\ref{s:cotg_bundle} we propose an alternative viewpoint
on the differential structure of a vast class of metric measure
spaces (\emph{i.e.}, those having separable Sobolev space, which
is a quite mild assumption). By combining our Serre--Swan
Theorem \ref{thm:Serre-Swan} with the results proven in \cite{Gigli14},
we show in Theorem \ref{thm:cotg_bundle} that \((\X,\sfd,\mm)\)
is canonically associated with a \emph{cotangent bundle} \({\rm T}^*\X\),
which does not need the language of normed modules to be formulated.
This way, at \(\mm\)-almost every point \(x\in\X\) we obtain
a weak notion of cotangent space \({\rm T}^*_x\X\), which has
a Banach space structure. However, in this manuscript we do not
investigate further the possible geometric and analytic consequences
of this more `concrete' representation of the cotangent module.
\section{Preliminaries}
\addtocontents{toc}{\protect\setcounter{tocdepth}{2}}
In this section we collect many preliminary definitions and results,
which will be needed in the sequel. The contents of \S\ref{ss:liftings_fcs},
\S\ref{ss:correspondences}, and \S\ref{ss:embeddings} are classical.
The material in \S\ref{ss:norm_L0_mod} is mostly taken from
\cite{Gigli14,Gigli17}, apart from a few technical statements.
The discussion in \S\ref{ss:norm_Linftym_mod} is new, but
strongly inspired by already known results.
\medskip

Throughout the whole paper, we will always tacitly adopt the
following convention: given any measurable space \((\X,\Sigma)\),
it holds that
\begin{equation}\label{eq:hp_singletons}
\{x\}\in\Sigma\quad\text{ for every }x\in\X.
\end{equation}
In fact, this assumption plays a role only when considering
fibers of normed \(\mathcal L^\infty(\Sigma)\)-modules
(starting from \S\ref{ss:fibers}). Notice that in the setting
of metric measure spaces (cf.\ \S\ref{s:cotg_bundle}), the
condition in \eqref{eq:hp_singletons} is always in force, since
all singletons are closed and thus Borel measurable.
\subsection{Liftings of measurable functions}\label{ss:liftings_fcs}
Let us recall the concept of `lifting of a measure space',
in the sense of \cite[Definition 341A]{Fremlin3}:
\begin{definition}[Lifting]\label{def:lifting}
Let \((\X,\Sigma,\mm)\) be a measure space. Then a map \(\ell\colon\Sigma\to\Sigma\)
is said to be a \emph{lifting} of the measure \(\mm\) provided the following
properties are satisfied:
\begin{itemize}
\item[\(\rm i)\)] \(\ell\) is a Boolean homomorphism, \emph{i.e.},
it holds that \(\ell(\emptyset)=\emptyset\), \(\ell(\X)=\X\), and
\[\begin{split}
\ell(A\Delta B)=&\ell(A)\Delta\ell(B),\\
\ell(A\cap B)=&\ell(A)\cap\ell(B),
\end{split}\quad\text{ for every }A,B\in\Sigma,
\]
\item[\(\rm ii)\)] \(\ell(N)=\emptyset\) for every \(N\in\Sigma\) such that \(\mm(N)=0\),
\item[\(\rm iii)\)] \(\mm\big(A\Delta\ell(A)\big)=0\) for every \(A\in\Sigma\).
\end{itemize}
\end{definition}

Liftings can be proven to exist in high generality.
For a proof of the following extremely deep result,
we refer to \cite[Theorem 341K]{Fremlin3}.
\begin{theorem}[von Neumann--Maharam]\label{thm:von_Neumann-Maharam}
Let \((\X,\Sigma,\mm)\) be a complete, \(\sigma\)-finite measure space
such that \(\mm(\X)>0\).
Then there exists a lifting \(\ell\colon\Sigma\to\Sigma\) of the measure \(\mm\).
\end{theorem}
\begin{remark}{\rm
We point out that Theorem \ref{thm:von_Neumann-Maharam} strongly
relies upon the Axiom of Choice. Consequently, all our results
concerning liftings of normed modules will rely upon the Axiom of
Choice as well. Nevertheless, in the study of separable normed
modules, the usage of liftings (thus, of the Axiom of Choice)
can be avoided; see Remark \ref{rmk:repr_thm_no_AC} for the details.
\fr}\end{remark}

Given any measure space \((\X,\Sigma,\mm)\), let us denote by \(\mathcal L^\infty(\Sigma)\)
the space of all bounded measurable functions
\(\bar f\colon\X\to\R\), which is a vector space and a commutative ring with respect to the
natural pointwise operations. It turns out that \(\mathcal L^\infty(\Sigma)\) is a Banach
space and a topological ring when endowed with the norm
\({\|\bar f\|}_{\mathcal L^\infty(\Sigma)}\coloneqq\sup_\X|\bar f|\). 
Consider the equivalence relation on \(\mathcal L^\infty(\Sigma)\) given by
\(\mm\)-a.e.\ equality: for any \(\bar f,\bar g\in\mathcal L^\infty(\Sigma)\), we set
\[
\bar f\sim_\mm\bar g\quad\Longleftrightarrow
\quad\bar f(x)=\bar g(x)\,\text{ for }\mm\text{-a.e.\ }x\in\X.
\]
Then we denote by \(L^\infty(\mm)\) the quotient space
\(\mathcal L^\infty(\Sigma)/\sim_\mm\). It holds that \(L^\infty(\mm)\) is a Banach
space and a topological ring when endowed with the quotient norm
\({\|f\|}_{L^\infty(\mm)}\coloneqq{\rm ess\,sup}_\X|f|\).

We denote by \([\,\cdot\,]_\mm\colon\mathcal L^\infty(\Sigma)\to L^\infty(\mm)\)
the projection map, which turns out to be a linear and continuous operator.
Let us define the family of simple functions
\(\overline{{\sf Sf}}(\Sigma)\subseteq\mathcal L^\infty(\Sigma)\) as
\[
\overline{{\sf Sf}}(\Sigma)\coloneqq\bigg\{\sum_{i=1}^n a_i\,\nchi_{A_i}\;\bigg|
\;n\in\N,\,(a_i)_{i=1}^n\subseteq\R,\,(A_i)_{i=1}^n\subseteq\Sigma
\text{ partition of }\X\bigg\}.
\]
Here, \(\nchi_A\) stands for the characteristic function of the
set \(A\subseteq\X\), namely, we set \(\nchi_A(x)\coloneqq 1\)
for every \(x\in A\) and \(\nchi_A(x)\coloneqq 0\) for every
\(x\in\X\setminus A\). It can be readily proved that
\(\overline{{\sf Sf}}(\Sigma)\) is a dense vector subspace and
subring of \(\mathcal L^\infty(\Sigma)\).

We define \({\sf Sf}(\mm)\subseteq L^\infty(\mm)\) as the image
of \(\overline{\sf Sf}(\Sigma)\) under the map \([\,\cdot\,]_\mm\), namely
\begin{equation}\label{eq:def_Sf}
{\sf Sf}(\mm)\coloneqq\bigg\{\sum_{i=1}^n a_i\,[\nchi_{A_i}]_\mm\;\bigg|
\;n\in\N,\,(a_i)_{i=1}^n\subseteq\R,\,(A_i)_{i=1}^n\subseteq\Sigma
\text{ partition of }\X\bigg\}.
\end{equation}
Therefore, \({\sf Sf}(\mm)\) is a dense vector subspace and subring
of \(L^\infty(\mm)\).
\bigskip

As we are going to prove, any lifting of a measure space can be promoted
to a `lifting of measurable functions'. This statement is taken from
\cite[Exercise 341X(e)]{Fremlin3}:
\begin{theorem}[Lifting of measurable functions]\label{thm:lifting_Linfty}
Let \((\X,\Sigma,\mm)\) be a measure space and
let \(\ell\colon\Sigma\to\Sigma\) be a lifting of the measure \(\mm\).
Then there exists a unique linear and continuous operator
\(\mathcal L\colon L^\infty(\mm)\to\mathcal L^\infty(\Sigma)\) such that
\begin{equation}\label{eq:def_ell}
\mathcal L\big([\nchi_A]_\mm\big)=\nchi_{\ell(A)}\quad\text{ for every }A\in\Sigma.
\end{equation}
Moreover, the following properties hold:
\begin{itemize}
\item[\(\rm i)\)] \(\mathcal L\) is an isometry,
\emph{i.e.}, \({\big\|\mathcal L(f)\big\|}_{\mathcal L^\infty(\Sigma)}=
{\|f\|}_{L^\infty(\mm)}\) for every \(f\in L^\infty(\mm)\).
\item[\(\rm ii)\)] \(\mathcal L\big([c\,]_\mm\big)=c\) for every constant \(c\in\R\).
\item[\(\rm iii)\)] \(\mathcal L\) is a right inverse of \([\,\cdot\,]_\mm\),
\emph{i.e.}, \(\big[\mathcal L(f)\big]_\mm=f\) for every \(f\in L^\infty(\mm)\).
\item[\(\rm iv)\)] \(\mathcal L(fg)=\mathcal L(f)\,\mathcal L(g)\)
for every \(f,g\in L^\infty(\mm)\).
\item[\(\rm v)\)] \(\big|\mathcal L(f)\big|=\mathcal L\big(|f|\big)\)
for every \(f\in L^\infty(\mm)\).
\item[\(\rm vi)\)] If \(f,g\in L^\infty(\mm)\) satisfy \(f\geq g\) in the
\(\mm\)-a.e.\ sense, then \(\mathcal L(f)\geq\mathcal L(g)\).
\end{itemize}
\end{theorem}
\begin{proof} First of all, we are forced to define the operator
\(\mathcal L\colon{\sf Sf}(\mm)\to\overline{\sf Sf}(\Sigma)\) as follows:
\begin{equation}\label{eq:lifting_L0_aux1}
\mathcal L\bigg(\sum_{i=1}^n a_i\,[\nchi_{A_i}]_\mm\bigg)
\coloneqq\sum_{i=1}^n a_i\,\nchi_{\ell(A_i)}
\quad\text{ for every }\sum_{i=1}^n a_i\,[\nchi_{A_i}]_\mm\in{\sf Sf}(\mm).
\end{equation}
We need to prove the well-posedness of \(\mathcal L\).
It thus remains to show that
\begin{equation}\label{eq:lifting_L0_aux2}
\sum_{i=1}^n a_i\,[\nchi_{A_i}]_\mm=\sum_{j=1}^m b_j\,[\nchi_{B_j}]_\mm\quad
\Longrightarrow\quad\sum_{i=1}^n a_i\,\nchi_{\ell(A_i)}=
\sum_{j=1}^m b_j\,\nchi_{\ell(B_j)}.
\end{equation}
The left-hand side of \eqref{eq:lifting_L0_aux2} is equivalent to saying that
\(\mm(A_i\cap B_j)=0\) for all \(i,j\) with \(a_i\neq b_j\); this implies that
\(\ell(A_i)\cap\ell(B_j)=\ell(A_i\cap B_j)=\emptyset\) for any such \(i,j\),
which is equivalent to the right-hand side, whence \eqref{eq:lifting_L0_aux2}
is proven. Moreover, it can be readily checked that \(\mathcal L\) is linear.
Given any simple function \(f=\sum_{i=1}^n a_i\,[\nchi_{A_i}]_\mm\in{\sf Sf}(\mm)\),
we have that \(|f|=\sum_{i=1}^n|a_i|\,[\nchi_{A_i}]_\mm\) \(\mm\)-a.e.\ and
\(\big|\mathcal L(f)\big|=\sum_{i=1}^n|a_i|\,\nchi_{\ell(A_i)}\), thus accordingly
\[
{\big\|\mathcal L(f)\big\|}_{\mathcal L^\infty(\Sigma)}
=\sup_{\X}\big|\mathcal L(f)\big|
=\underset{\substack{i=1,\ldots,n: \\ \ell(A_i)\neq \emptyset}}\max|a_i|
=\underset{\substack{i=1,\ldots,n: \\ \mm(A_i)>0}}\max|a_i|
=\underset{\X}{\rm ess\,sup}\,|f|={\|f\|}_{L^\infty(\mm)}.
\]
Then the map \(\mathcal L\) is an isometry from
\(\big({\sf Sf}(\mm),{\|\cdot\|}_{L^\infty(\mm)}\big)\)
to \(\big(\overline{\sf Sf}(\Sigma),{\|\cdot\|}_{\mathcal L^\infty(\Sigma)}\big)\),
so it can be uniquely extended to a linear isometry
\(\mathcal L\colon L^\infty(\mm)\to\mathcal L^\infty(\Sigma)\).
This proves existence, uniqueness, and item i). Item ii) immediately follows from
the \(1\)-homogeneity of \(\mathcal L\) and \eqref{eq:def_ell}.
To prove item iii), it suffices to observe that \(\big[\mathcal L(f)\big]_\mm=f\)
holds for every \(f\in{\sf Sf}(\mm)\) by construction, thus also for any
\(f\in L^\infty(\mm)\) by continuity of \([\,\cdot\,]_\mm\) and \(\mathcal L\).
Item iv) can be proved for \(f,g\in{\sf Sf}(\mm)\) by direct computation:
if \(f=\sum_{i=1}^n a_i\,[\nchi_{A_i}]_\mm\) and \(g=\sum_{j=1}^m b_j\,[\nchi_{B_j}]_\mm\),
then \(fg=\sum_{i,j}a_i\,b_j\,[\nchi_{A_i\cap B_j}]_\mm\), so that accordingly
\[
\mathcal L(fg)=\sum_{i,j}a_i\,b_j\,\nchi_{\ell(A_i\cap B_j)}=
\bigg(\sum_{i=1}^n a_i\,\nchi_{\ell(A_i)}\bigg)
\bigg(\sum_{j=1}^m b_j\,\nchi_{\ell(B_j)}\bigg)=\mathcal L(f)\,\mathcal L(g).
\]
Since \(L^\infty(\mm)\) and \(\mathcal L^\infty(\Sigma)\) are topological rings,
we deduce from the density of \({\sf Sf}(\mm)\) in \(L^\infty(\mm)\)
that \(\mathcal L(fg)=\mathcal L(f)\,\mathcal L(g)\) holds for every
\(f,g\in L^\infty(\mm)\), thus proving item iv) in full generality.
Item v) can be readily checked for \(f\in{\sf Sf}(\mm)\), whence the
general case follows from an approximation argument. Finally, let us
prove item vi). By linearity of \(\mathcal L\), it suffices to
consider \(f\in L^\infty(\mm)\) such that \(f\geq 0\) in the
\(\mm\)-a.e.\ sense. Choose any sequence \((f_n)_n\subseteq{\sf Sf}(\mm)\)
such that \(f_n\geq 0\) holds \(\mm\)-a.e.\ for each \(n\in\N\) and
\(\lim_n{\|f_n-f\|}_{L^\infty(\mm)}=0\). Given that \(\mathcal L\)
is continuous and \(\mathcal L(f_n)\geq 0\) holds for all \(n\in\N\) by construction,
we conclude that \(\mathcal L(f)\geq 0\), as
it is a uniform limit of non-negative functions. The proof
of the statement is complete.
\end{proof}
\subsection{Measurable correspondences}\label{ss:correspondences}
Aim of this subsection is to recall some basic definitions and results
concerning measurable correspondences. The whole material we are going
to discuss can be found, \emph{e.g.}, in \cite{AliprantisBorder99}.
\bigskip

Let \((\X,\Sigma)\) be a measurable space and let
\((\Y,\sfd_\Y)\) be a separable metric space.
Then any map \(\varphi\colon\X\to 2^\Y\) is said to be
a \emph{correspondence} from \(\X\) to \(\Y\) and is denoted
by \(\varphi\colon\X\mto\Y\).
\begin{definition}[Measurable correspondence]
A correspondence \(\varphi\colon\X\mto\Y\) is said to be:
\begin{itemize}
\item[\(\rm a)\)] \emph{weakly measurable}, provided
\(\big\{x\in\X\,:\,\varphi(x)\cap U\neq\emptyset\big\}\in\Sigma\)
for every \(U\subseteq\Y\) open,
\item[\(\rm b)\)] \emph{measurable}, provided
\(\big\{x\in\X\,:\,\varphi(x)\cap C\neq\emptyset\big\}\in\Sigma\)
for every \(C\subseteq\Y\) closed.
\end{itemize}
\end{definition}
Let us collect a few important properties of (weakly) measurable correspondences:
\begin{itemize}
\item[\(\rm i)\)] Every measurable correspondence from \(\X\) to \(\Y\) is
weakly measurable. Conversely, every weakly measurable correspondence
from \(\X\) to \(\Y\) with compact values is measurable.
\item[\(\rm ii)\)] Let \(\varphi\colon\X\mto\Y\) be a
single-valued correspondence, \emph{i.e.}, for every \(x\in\X\) there exists an
element \(\bar\varphi(x)\in\Y\) such that \(\varphi(x)=\big\{\bar\varphi(x)\big\}\).
Then \(\varphi\) is a measurable correspondence if and only if
\(\bar\varphi\colon\X\to\Y\) is a measurable map.
\item[\(\rm iii)\)] Let \(\varphi,\varphi'\colon\X\mto\Y\) be two measurable
correspondences with compact values. Then the intersection correspondence
\(\varphi\cap\varphi'\colon\X\mto\Y\), defined as
\((\varphi\cap\varphi')(x)\coloneqq\varphi(x)\cap\varphi'(x)\) for
every \(x\in\X\), is measurable.
\item[\(\rm iv)\)] A correspondence \(\varphi\colon\X\mto\Y\)
with non-empty values is weakly measurable if and only if
\(\X\ni x\mapsto\sfd_\Y\big(y,\varphi(x)\big)\)
is a measurable function for every \(y\in\Y\).
\item[\(\rm v)\)] Let \(f\colon\X\times\Y\to\R\) be a Carath\'{e}odory
function, \emph{i.e.},
\[\begin{split}
f(\cdot,y)\colon\X\to\R&\quad\text{ is measurable for every }y\in\Y,\\
f(x,\cdot)\colon\Y\to\R&\quad\text{ is continuous for every }x\in\X.
\end{split}\]
Define the correspondence \(\varphi\colon\X\mto\Y\) as
\(\varphi(x)\coloneqq\big\{y\in\Y\,:\,f(x,y)=0\big\}\) for every \(x\in\X\).
If the metric space \((\Y,\sfd_\Y)\) is compact, then \(\varphi\) is a
measurable correspondence.
\item[\(\rm vi)\)] Let \(\varphi\colon\X\mto\Y\) be a weakly measurable correspondence
with closed values. Then the graph of \(\varphi\) is measurable, \emph{i.e.},
it holds that
\[
\big\{(x,y)\in\X\times\Y\;\big|\;y\in\varphi(x)\big\}\in\Sigma\otimes\mathscr B(\Y).
\]
\item[\(\rm vii)\)] \textsc{Kuratowski--Ryll-Nardzewski theorem.} Suppose that
the metric space \((\Y,\sfd_\Y)\) is complete. Let \(\varphi\colon\X\mto\Y\) be
a weakly measurable correspondence with non-empty closed values. Then \(\varphi\)
admits a \emph{measurable selector} \(s\colon\X\to\Y\), \emph{i.e.}, the map \(s\)
is measurable and satisfies \(s(x)\in\varphi(x)\) for every \(x\in\X\).
\item[\(\rm viii)\)] Let \(\varphi\colon\X\mto\Y\) be a weakly measurable
correspondence. Then its \emph{closure correspondence}
\({\rm cl}_\Y(\varphi)\colon\X\mto\Y\), which is defined
as \({\rm cl}_\Y(\varphi)(x)\coloneqq{\rm cl}_\Y\big(\varphi(x)\big)\)
for every \(x\in\X\), is weakly measurable.
\end{itemize}
Furthermore, we will also need the following standard results about preimages
and compositions of measurable correspondences:
\begin{lemma}\label{lem:preimg_corr}
Let \((\X,\Sigma)\) be a measurable space. Let \((\Y,\sfd_\Y), ({\rm Z},\sfd_{\rm Z})\)
be separable metric spaces, with \((\Y,\sfd_\Y)\) compact.
Let \(\varphi\colon\X\mto{\rm Z}\) be a measurable correspondence and
let \(\psi\colon\Y\to{\rm Z}\) be a continuous map. Define the preimage correspondence
\(\psi^{-1}(\varphi)\colon\X\mto\Y\) as
\[
\psi^{-1}(\varphi)(x)\coloneqq\psi^{-1}\big(\varphi(x)\big)\subseteq\Y
\quad\text{ for every }x\in\X.
\]
Then \(\psi^{-1}(\varphi)\) is a measurable correspondence.
\end{lemma}
\begin{proof}
Let \(C\subseteq\Y\) be a closed set. In particular, \(C\) is compact,
whence \(\psi(C)\subseteq{\rm Z}\) is compact (and thus closed) by continuity
of \(\psi\). Therefore, the measurability of \(\varphi\) grants that
\[
\big\{x\in\X\;\big|\;\psi^{-1}(\varphi)(x)\cap C\neq\emptyset\big\}
=\big\{x\in\X\;\big|\;\varphi(x)\cap\psi(C)\neq\emptyset\big\}\in\Sigma.
\]
By arbitrariness of \(C\), we conclude that \(\psi^{-1}(\varphi)\) is a
measurable correspondence.
\end{proof}
\begin{lemma}\label{lem:comp_corr}
Let \((\X,\Sigma_\X)\), \((\Y,\Sigma_\Y)\) be measurable spaces and
\(({\rm Z},\sfd_{\rm Z})\) a separable metric space. Let \(u\colon\X\to\Y\)
be a measurable map and \(\varphi\colon\Y\mto{\rm Z}\) a weakly measurable
correspondence. Consider the correspondence \(\varphi\circ u\colon\X\mto{\rm Z}\),
given by
\[
(\varphi\circ u)(x)\coloneqq\varphi\big(u(x)\big)\subseteq{\rm Z}
\quad\text{ for every }x\in\X.
\]
Then \(\varphi\circ u\) is a weakly measurable correspondence.
\end{lemma}
\begin{proof}
Let \(U\subseteq{\rm Z}\) be an open set. Since \(\big\{y\in\Y\,:\,\varphi(y)
\cap U\neq\emptyset\big\}\in\Sigma_\Y\) and \(u\) is measurable, it holds that
\[
\big\{x\in\X\;\big|\;(\varphi\circ u)(x)\cap U\neq\emptyset\big\}
=u^{-1}\big(\big\{y\in\Y\;\big|\;\varphi(y)\cap U\neq\emptyset\big\}\big)\in\Sigma_\X.
\]
By arbitrariness of \(U\), we conclude that \(\varphi\circ u\)
is a weakly measurable correspondence.
\end{proof}
\subsection{Linear isometric embeddings of separable Banach spaces}\label{ss:embeddings}
In this section we collect some results about linear isometric embeddings
of Banach spaces.
\bigskip

Given any Banach space \(E\), we shall denote by \(B_E\) its closed unit ball
\(\big\{v\in E\,:\,\|v\|_E\leq 1\big\}\).
We use the notation \(E'\) to denote the continuous dual space of \(E\),
which is a Banach space.
We begin by recalling some classical definitions and results (which can be found,
\emph{e.g.}, in \cite{AliprantisBorder99}):
\begin{itemize}
\item[\(\rm i)\)] \textsc{Cantor set.} The \emph{Cantor set} is the
product \(\Delta\coloneqq\{0,2\}^\N\), where each factor \(\{0,2\}\)
is endowed with the discrete topology. The topology of \(\Delta\) is
induced by the distance
\[
\sfd_\Delta(a,b)\coloneqq\sum_{n=1}^\infty\frac{|a_n-b_n|}{3^n}
\quad\text{ for every }a=(a_n)_n,b=(b_n)_n\in\Delta.
\]
It holds that the Cantor set \(\Delta\) is compact. Moreover, \(\Delta\) is homeomorphic
to the closed subset \(C\coloneqq\big\{\sum_{n=1}^\infty a_n/3^n\,:\,a\in\Delta\big\}\)
of \([0,1]\) via the map \(\Delta\ni a\mapsto\sum_{n=1}^\infty a_n/3^n\in C\).
\item[\(\rm ii)\)] Let \(K\) be a non-empty closed subset of \(\Delta\).
Given any \(a\in\Delta\), there exists a unique element \(r(a)\in K\)
such that \(\sfd_\Delta\big(a,r(a)\big)=\sfd_\Delta(a,K)\). The resulting
map \(r\colon\Delta\to K\) is continuous and satisfies \(r|_K={\rm id}_K\).
We say that \(r\) is a \emph{retraction} of \(\Delta\) onto \(K\).
\item[\(\rm iii)\)] \textsc{Hilbert cube.} The \emph{Hilbert cube} is the product
topological space \(I^\infty\coloneqq[-1,1]^\N\). It is compact and its topology
is induced by the distance
\[
\sfd_{I^\infty}(\alpha,\beta)\coloneqq\sum_{k=1}^\infty\frac{|\alpha_k-\beta_k|}{2^k}
\quad\text{ for every }\alpha=(\alpha_k)_k,\beta=(\beta_k)_k\in I^\infty.
\]
Moreover, there exists a continuous surjective map \(\psi\colon\Delta\to I^\infty\).
\item[\(\rm iv)\)] Let \(E\) be a separable Banach space. Let \((v_k)_{k\in\N}\) be a fixed dense subset of \(B_E\). Then the map \(\iota\colon B_{E'}\to I^\infty\),
defined as \(\iota(\omega)\coloneqq\big(\omega[v_k]\big)_k\) for every
\(\omega\in B_{E'}\), is a homeomorphism with its image (when the domain
\(B_{E'}\) is endowed with the restricted weak\(^*\) topology).
\item[\(\rm v)\)] Let \(K\) be a compact Hausdorff topological space.
Then the space \(C(K)\) of real-valued continuous functions
on \(K\) is a separable Banach space if endowed with the norm
\[
\|g\|_{C(K)}\coloneqq\sup_{t\in K}\big|g(t)\big|
\quad\text{ for every }g\in C(K).
\]
\end{itemize}
We now recall the classical Banach--Mazur theorem \cite{BP75},
which states that any separable Banach space can be
embedded linearly and isometrically into \(C([0,1])\).
We also report the proof of this result, as the explicit construction
of such an embedding will be needed in \S\ref{s:sep_Ban_bundle}.
\begin{theorem}[Banach--Mazur]\label{thm:Banach-Mazur}
Let \(E\) be a separable Banach space. Then there exists a linear isometric
embedding \({\rm I}\colon E\to C([0,1])\).
\end{theorem}
\begin{proof}
Let \(\iota\colon B_{E'}\to I^\infty\) be a continuous embedding  as in item iv).
Let \(\psi\colon\Delta\to I^\infty\) be a continuous surjection as in item iii).
Note that \(\psi^{-1}\big(\iota(B_{E'})\big)\subseteq\Delta\) is closed
(thus compact) by continuity of \(\psi\) and weak\(^*\) compactness of
\(B_{E'}\) (the latter is granted by Banach--Alaoglu theorem). Then consider
a retraction \(r\colon\Delta\to\psi^{-1}\big(\iota(B_{E'})\big)\) as in item ii).
We claim that the operator \({\rm I}'\colon E\to C(\Delta)\), which is defined as
\[
{\rm I}'[v](a)\coloneqq(\iota^{-1}\circ\psi\circ r)(a)[v]
\quad\text{ for every }v\in E\text{ and }a\in\Delta,
\]
is well-posed, linear, and isometric. The fact that \({\rm I}'[v]\in C(\Delta)\)
for every \(v\in E\) can be easily checked: if a sequence \((a^i)_i\subseteq\Delta\)
converges to some element \(a\in\Delta\), then \((\iota^{-1}\circ\psi\circ r)(a^i)\)
converges to \((\iota^{-1}\circ\psi\circ r)(a)\) in the weak\(^*\) topology, and
accordingly \({\rm I}'[v](a)=\lim_i{\rm I}'[v](a^i)\). Linearity of \(\rm I'\)
immediately follows from the definition. Moreover, it holds that
\[
\big\|{\rm I}'[v]\big\|_{C(\Delta)}=\sup_{a\in\Delta}\big|{\rm I}'[v](a)\big|
=\sup_{\omega\in B_{E'}}\big|\omega[v]\big|=\|v\|_E
\quad\text{ for every }v\in E,
\]
where the last equality follows from the fact that the canonical embedding
of the space \(E\) into its bidual \(E''\) is an isometric operator.
This shows that the map \(\rm I'\) is an isometry.

Finally, denote by \(h\colon C\to\Delta\) the homeomorphism described in item i)
and write \([0,1]\setminus C\) as a disjoint union
\(\bigsqcup_{i\in\N}(t_i,s_i)\). We define the map
\(e\colon C(\Delta)\to C([0,1])\) in the following way:
given any \(g\in C(\Delta)\), we set
\[
e(g)(t)\coloneqq\left\{\begin{array}{ll}
(g\circ h)(t)\\
(g\circ h)(t_i)+\frac{t-t_i}{s_i-t_i}(g\circ h)(s_i)
\end{array}\quad\begin{array}{ll}
\text{ if }t\in C,\\
\text{ if }t\in(t_i,s_i)\text{ for some }i\in\N.
\end{array}\right.
\]
It holds that \(e\) is linear and isometric. Then the map
\({\rm I}\colon E\to C([0,1])\), given by \({\rm I}\coloneqq e\circ{\rm I}'\),
is a linear isometric embedding as well. Therefore, the statement is achieved.
\end{proof}

The statement of Theorem \ref{thm:Banach-Mazur} can be
reformulated by using the following definition:
\begin{definition}[Universal separable Banach space]
A separable Banach space \(\B\) is said to be a
\emph{universal separable Banach space (up to linear isometry)}
provided for any separable Banach space \(E\) there exists a
linear isometric embedding \({\rm I}\colon E\to\B\).
\end{definition}

Therefore, Theorem \ref{thm:Banach-Mazur} reads as follows:
{\it \(C([0,1])\) is a universal separable Banach space}.
\begin{remark}{\rm
As a byproduct of the proof of Theorem \ref{thm:Banach-Mazur},
we see that \(C(\Delta)\) is a universal separable Banach space.
\fr}\end{remark}
\subsection{Normed \texorpdfstring{\(L^0(\mm)\)}{L0(m)}-modules}
\label{ss:norm_L0_mod}
Let \((\X,\Sigma,\mm)\) be a given \(\sigma\)-finite measure space.
We shall denote by \(L^0(\mm)\) the space of all equivalence classes
(up to \(\mm\)-a.e.\ equality) of measurable
functions from \(\X\) to \(\R\),
which is a vector space and a commutative ring with respect to the
natural pointwise operations.
\bigskip

Given a probability measure \(\mm'\) on \((\X,\Sigma)\) satisfying
\(\mm\ll\mm'\ll\mm\) -- for instance, pick any sequence \((A_n)_n\subseteq\Sigma\)
such that \(0<\mm(A_n)<+\infty\) for every \(n\in\N\) and
\(\X=\bigcup_{n\in\N}A_n\), and consider
the measure \(\mm'\coloneqq\sum_{n\in\N}\frac{\mm|_{A_n}}{2^n\,\mm(A_n)}\) on \(\X\)
-- we define the complete distance \(\sfd_{L^0(\mm)}\) as
\[
\sfd_{L^0(\mm)}(f,g)\coloneqq\int|f-g|\wedge 1\,\d\mm'
\quad\text{ for every }f,g\in L^0(\mm).
\]
It turns out that \(L^0(\mm)\) is a topological vector space and a topological
ring when endowed with the distance \(\sfd_{L^0(\mm)}\). The distance \(\sfd_{L^0(\mm)}\)
depends on the chosen measure \(\mm'\), but its induced topology does not.
Observe also that the space \(L^\infty(\mm)\) is \(\sfd_{L^0(\mm)}\)-dense in \(L^0(\mm)\).
\bigskip

We recall the notion of \emph{normed \(L^0(\mm)\)-module},
which has been introduced by Gigli in \cite{Gigli14}:
\begin{definition}[Normed \(L^0(\mm)\)-module]\label{def:normed_L0_mod}
Let \((\X,\Sigma,\mm)\) be a \(\sigma\)-finite measure space. Then a
\emph{normed \(L^0(\mm)\)-module} is a couple \(\big(\mathscr M,|\cdot|\big)\)
with the following properties:
\begin{itemize}
\item[\(\rm i)\)] \(\mathscr M\) is an algebraic \(L^0(\mm)\)-module.
\item[\(\rm ii)\)] The map \(|\cdot|\colon\mathscr M\to L^0(\mm)\), which
is called a \emph{pointwise norm} on \(\mathscr M\), satisfies
\begin{equation}\label{eq:ptwse_norm}\begin{split}
|v|\geq 0&\quad\text{ for every }v\in\mathscr M\text{, with equality if and only if }v=0,\\
|v+w|\leq|v|+|w|&\quad\text{ for every }v,w\in\mathscr M,\\
|f\cdot v|=|f||v|&\quad\text{ for every }f\in L^0(\mm)\text{ and }v\in\mathscr M,
\end{split}\end{equation}
where all (in)equalities are intended in the \(\mm\)-a.e.\ sense.
\item[\(\rm iii)\)] The distance \(\sfd_{\mathscr M}\) on \(\mathscr M\) associated
with \(|\cdot|\), which is defined as
\[
\sfd_{\mathscr M}(v,w)\coloneqq\sfd_{L^0(\mm)}\big(|v-w|,0\big)
\quad\text{ for every }v,w\in\mathscr M,
\]
is complete.
\end{itemize}
\end{definition}
\begin{remark}[Locality/glueing property]{\rm
Let \((\X,\Sigma,\mm)\) be a \(\sigma\)-finite measure space.
Then any normed \(L^0(\mm)\)-module \(\mathscr M\) has
the following properties:
\begin{itemize}
\item \textsc{Locality.} If \((A_n)_{n\in\N}\subseteq\Sigma\) is
a partition of \(\X\) and \(v,w\in\mathscr M\) are two elements such that
\([\nchi_{A_n}]_\mm\cdot v=[\nchi_{A_n}]_\mm\cdot w\) holds for every \(n\in\N\),
then \(v=w\).
\item \textsc{Glueing.} If \((A_n)_{n\in\N}\subseteq\Sigma\) is a
partition of \(\X\) and \((v_n)_{n\in\N}\) is any sequence in
\(\mathscr M\), then there exists \(v\in\mathscr M\) such that
\([\nchi_{A_n}]_\mm\cdot v=[\nchi_{A_n}]_\mm\cdot v_n\) for
every \(n\in\N\).
The element \(v\) -- which is uniquely determined by the
locality property -- will be denoted by
\(\sum_{n\in\N}[\nchi_{A_n}]_\mm\cdot v_n\).
It holds \(\sfd_{\mathscr M}\big(\sum_{n=1}^N[\nchi_{A_n}]_\mm
\cdot v_n,\sum_{n\in\N}[\nchi_{A_n}]_\mm\cdot v_n\big)\to 0\) as
\(N\to\infty\).
\end{itemize}
Both the locality property and the glueing property have been
proved in \cite{Gigli14}.
\fr}\end{remark}
Let us fix some useful notation about normed \(L^0(\mm)\)-modules:
\begin{itemize}
\item[\(\rm a)\)] \textsc{Morphism.} A map \(\Phi\colon\mathscr M\to\mathscr N\) between two
normed \(L^0(\mm)\)-modules \(\mathscr M\), \(\mathscr N\) is said to be
a \emph{morphism} provided it is an \(L^0(\mm)\)-linear contraction, \emph{i.e.},
\[\begin{split}
\Phi(f\cdot v+g\cdot w)=f\cdot\Phi(v)+g\cdot\Phi(w)&
\quad\text{ for all }f,g\in L^0(\mm)\text{ and }v,w\in\mathscr M,\\
\big|\Phi(v)\big|\leq|v|\;\;\;\mm\text{-a.e.}&
\quad\text{ for all }v\in\mathscr M.
\end{split}\]
This allows us to speak about the category \({\bf NMod}(\X,\Sigma,\mm)\)
of normed \(L^0(\mm)\)-modules.
\item[\(\rm b)\)] \textsc{Dual.} The \emph{dual} of \(\mathscr M\)
is defined as the space \(\mathscr M^*\) of all \(L^0(\mm)\)-linear continuous maps
from \(\mathscr M\) to \(L^0(\mm)\). It holds that \(\mathscr M^*\) is a normed
\(L^0(\mm)\)-module if endowed with the natural operations and the pointwise norm
\begin{equation}\label{eq:dual_ptwse_norm}
|\omega|\coloneqq{\rm ess\,sup\,}\big\{\omega(v)\;\big|
\;v\in\mathscr M,\,|v|\leq 1\;\mm\text{-a.e.}\big\}\in L^0(\mm)
\quad\text{ for all }\omega\in\mathscr M^*.
\end{equation}
\item[\(\rm c)\)] \textsc{Generators.} We say that a set \(\mathcal S\subseteq\mathscr M\)
\emph{generates} \(\mathscr M\) on some set \(A\in\Sigma\) provided
the smallest (algebraic) \(L^0(\mm)\)-module containing \([\nchi_A]_\mm\cdot\mathcal S\)
is dense in \([\nchi_A]_\mm\cdot\mathscr M\).
\item[\(\rm d)\)] \textsc{Linear independence.} Some elements \(v_1,\ldots,v_n\in\mathscr M\)
are said to be \emph{independent} on a set \(A\in\Sigma\) provided
for any \(f_1,\ldots,f_n\in L^0(\mm)\) one has
\(\sum_{i=1}^n[\nchi_A]_\mm\,f_i\cdot v_i=0\) if and only if
\(f_1,\ldots,f_n=0\) holds \(\mm\)-a.e.\ on \(A\).
\item[\(\rm e)\)] \textsc{Local basis.} We say that \(v_1,\ldots,v_n\in\mathscr M\)
constitute a \emph{local basis} for \(\mathscr M\) on \(A\) provided they
are independent on \(A\) and \(\{v_1,\ldots,v_n\}\) generates \(\mathscr M\) on \(A\).
\item[\(\rm f)\)] \textsc{Local dimension.} The module \(\mathscr M\) has \emph{local dimension}
equal to \(n\in\N\) on \(A\) if it admits a local basis on \(A\) made of exactly \(n\) elements.
\item[\(\rm g)\)] \textsc{Dimensional decomposition.} It holds that the normed \(L^0(\mm)\)-module
\(\mathscr M\) admits a unique \emph{dimensional decomposition}
\(\{D_n\}_{n\in\N\cup\{\infty\}}\subseteq\Sigma\), \emph{i.e.}, \(\mathscr M\) has local
dimension \(n\) on \(D_n\) for every \(n\in\N\) and is not finitely-generated
on any measurable subset of \(D_\infty\) having
positive \(\mm\)-measure. Uniqueness here is intended up to \(\mm\)-negligible sets.
\item[\(\rm h)\)] \textsc{Proper module.} A normed \(L^0(\mm)\)-module \(\mathscr M\), whose
dimensional decomposition is denoted by \(\{D_n\}_{n\in\N\cup\{\infty\}}\), is said to be
\emph{proper} provided \(\mm(D_\infty)=0\).
\end{itemize}
We refer to \cite{Gigli14,Gigli17} for a thorough discussion
about normed \(L^0(\mm)\)-modules.
\begin{definition}[Countably-generated module]
Let \((\X,\Sigma,\mm)\) be any \(\sigma\)-finite measure space.
Let \(\mathscr M\) be a normed \(L^0(\mm)\)-module. Then we say
that \(\mathscr M\) is \emph{countably-generated} provided there
exists a countable family \(\mathcal C\subseteq\mathscr M\)
that generates \(\mathscr M\) (on \(\X\)).
\end{definition}
We call \({\bf NMod}_{\rm cg}(\X,\Sigma,\mm)\) the
category of countably-generated normed \(L^0(\mm)\)-modules.
Another class of modules we are interested in is that of
\emph{separable} normed \(L^0(\mm)\)-modules, \emph{i.e.},
those normed \(L^0(\mm)\)-modules \(\mathscr M\) for which
\((\mathscr M,\sfd_{\mathscr M})\) is a separable metric space.
We denote by \({\bf NMod}_{\rm s}(\X,\Sigma,\mm)\) the category
of separable normed \(L^0(\mm)\)-modules. In the forthcoming
discussion, we investigate the relation between countably-generated
and separable modules.
\bigskip

Let \((\X,\Sigma,\mm)\) be a \(\sigma\)-finite measure space.
Given any \(A,B\in\Sigma\), we declare that \(A\sim_\mm B\) if and
only if \(\mm(A\Delta B)=0\). This way, we obtain an equivalence
relation \(\sim_\mm\) on \(\Sigma\). Given any finite measure
\(\mm'\) on \((\X,\Sigma)\) such that \(\mm\ll\mm'\ll\mm\)
(thus in particular \(\sim_{\mm'}\) and \(\sim_\mm\) coincide),
we define the distance \(\sfd_{\mm'}\) on the quotient set
\(\Sigma/\sim_\mm\) as
\[
\sfd_{\mm'}\big([A]_{\sim_\mm},[B]_{\sim_\mm}\big)\coloneqq
\mm'(A\Delta B)\quad\text{ for every }[A]_{\sim_\mm},[B]_{\sim_\mm}
\in\Sigma/\sim_\mm.
\]
Then we say that the measure space \((\X,\Sigma,\mm)\)
is \emph{separable} provided \((\Sigma/\sim_\mm,\sfd_{\mm'})\) is
a separable metric space for some finite measure \(\mm'\) on
\((\X,\Sigma)\) such that \(\mm\ll\mm'\ll\mm\).
\begin{lemma}\label{lem:equiv_meas_sep}
Let \((\X,\Sigma,\mm)\) be a given \(\sigma\)-finite measure space.
Then \((\X,\Sigma,\mm)\) is a separable measure space if and only if
\(\big(L^0(\mm),\sfd_{L^0(\mm)}\big)\) is separable.
\end{lemma}
\begin{proof}\ \\
{\color{blue}\textsc{Necessity.}} Suppose \((\X,\Sigma,\mm)\) is
separable. Choose any probability measure \(\mm'\) on \((\X,\Sigma)\)
such that \(\mm\ll\mm'\ll\mm\) and \(\sfd_{\mm'}\) is separable. Observe
that \(L^0(\mm')\) and \(L^0(\mm)\) coincide. Fix a countable
\(\sfd_{\mm'}\)-dense subset \(\mathcal C\) of \(\Sigma/\sim_\mm\).
Now consider \(f\in L^0(\mm)\) and \(\eps>0\). As already observed,
we can find a simple function \(g\in{\sf Sf}(\mm)\)
-- see \eqref{eq:def_Sf} -- such that \(\sfd_{L^0(\mm)}(f,g)<\eps/2\).
Say that \(g=\sum_{i=1}^n a_i\,[\nchi_{A_i}]_\mm\). Without loss
of generality, we can assume that \(a_1,\ldots,a_n\in\Q\).
For any \(i=1,\ldots,n\), pick a set \(B_i\in\Sigma\) such that
\([B_i]_{\sim_\mm}\in\mathcal C\) and
\(\mm'(B_i\Delta A_i)\leq\eps/\big(2\,|a_i|\,n\big)\). Therefore,
it holds that
\[
\sfd_{L^0(\mm)}\bigg(g,\sum_{i=1}^n a_i\,[\nchi_{B_i}]_\mm\bigg)\leq
\sum_{i=1}^n|a_i|\int\big|[\nchi_{A_i}]_\mm-[\nchi_{B_i}]_\mm\big|\,\d\mm'
=\sum_{i=1}^n|a_i|\,\mm'(A_i\Delta B_i)\leq\frac{\eps}{2}.
\]
This means that the function
\(h\coloneqq\sum_{i=1}^n a_i\,[\nchi_{B_i}]_\mm\) satisfies
\(\sfd_{L^0(\mm)}(f,h)<\eps\). Given that
\[
\bigg\{\sum_{i=1}^n a_i\,[\nchi_{B_i}]_\mm\;\bigg|
\;n\in\N,\,(a_i)_{i=1}^n\subseteq\Q,\,(B_i)_{i=1}^n\subseteq\mathcal C\bigg\}
\]
is a countable family, we conclude that \(L^0(\mm)\) is separable,
as desired.\\
{\color{blue}\textsc{Sufficiency.}} Suppose \(L^0(\mm)\) is separable.
Observe that the map
\[(\Sigma/\sim_\mm,\sfd_{\mm'})\ni[A]_{\sim_\mm}
\longmapsto[\nchi_A]_\mm\in\big(L^0(\mm),\sfd_{L^0(\mm)}\big)
\]
is an isometry. Consequently,
we conclude that \((\Sigma/\sim_\mm,\sfd_{\mm'})\) is separable,
as required.
\end{proof}

Observe that, trivially, any separable normed \(L^0(\mm)\)-module \(\mathscr M\)
is countably-generated. The following result aims at determining in which cases
the converse implication is satisfied.
\begin{proposition}\label{prop:countable_gen}
Let \((\X,\Sigma,\mm)\) be any \(\sigma\)-finite measure space.
Then a countably-generated normed \(L^0(\mm)\)-module \(\mathscr M\)
is separable if and only if \((\X,\Sigma,\mm|_{\X\setminus D_0})\) is
a separable measure space, where \(\{D_n\}_{n\in\N\cup\{\infty\}}\subseteq\Sigma\)
stands for the dimensional decomposition of the module \(\mathscr M\).

In particular, if \((\X,\Sigma,\mm)\) is a separable measure space,
then any countably-generated normed \(L^0(\mm)\)-module is separable.
\end{proposition}
\begin{proof}
Suppose \((\X,\Sigma,\mm|_{\X\setminus D_0})\) is separable.
Let \(\mathcal C\) be a countable set generating
\(\mathscr M\), \emph{i.e.},
\begin{equation}\label{eq:countable_gen_aux}
\bigg\{\sum_{i=1}^n f_i\cdot v_i\;\bigg|\;
n\in\N,\,(f_i)_{i=1}^n\subseteq L^0(\mm|_{\X\setminus D_0}),\,
(v_i)_{i=1}^n\subseteq\mathcal C\bigg\}
\quad\text{ is dense in }\mathscr M.
\end{equation}
Lemma \ref{lem:equiv_meas_sep} ensures that the space
\(L^0(\mm|_{\X\setminus D_0})\)
is separable, thus we can pick a countable dense subset \(\mathcal D\)
of \(L^0(\mm|_{\X\setminus D_0})\). Since the multiplication map
\(L^0(\mm)\times\mathscr M\ni(f,v)\mapsto f\cdot v\in\mathscr M\)
is continuous, we immediately conclude from
\eqref{eq:countable_gen_aux} that the countable family
\[
\bigg\{\sum_{i=1}^n f_i\cdot v_i\;\bigg|\;
n\in\N,\,(f_i)_{i=1}^n\subseteq\mathcal D,\,
(v_i)_{i=1}^n\subseteq\mathcal C\bigg\}
\quad\text{ is dense in }\mathscr M.
\]
This proves that the metric space \((\mathscr M,\sfd_{\mathscr M})\) is separable,
as desired.

Conversely, suppose \(\mathscr M\) is separable. It can be readily checked
that there exists \(v\in\mathscr M\) such that \(|v|=1\) holds
\(\mm\)-a.e.\ on \(\X\setminus D_0\). (Notice that \([\nchi_{D_0}]_\mm\cdot v=0\)
by definition of \(D_0\).) Therefore, it holds that the module generated by
\(\{v\}\) can be identified with \(L^0(\mm|_{\X\setminus D_0})\) (considered
as a normed \(L^0(\mm)\)-module), whence the space \(L^0(\mm|_{\X\setminus D_0})\)
is separable. We can conclude that \((\X,\Sigma,\mm|_{\X\setminus D_0})\) is
separable by Lemma \ref{lem:equiv_meas_sep}, thus completing the proof.
\end{proof}

In \S\ref{s:repr_via_DL} we shall need the concept of direct limit of
normed \(L^0(\mm)\)-modules. The following result -- which is taken from
\cite[Theorem 2.1]{Pas19} -- states that direct limits always exist in the category
of normed \(L^0(\mm)\)-modules. (Actually, the statement below was proven
in the case in which \(\X\) is a Polish space, \(\Sigma\) is the Borel
\(\sigma\)-algebra of \(\X\), and \(\mm\) is a Radon measure on \(\X\);
however, the very same proof works in the more general case of
a \(\sigma\)-finite measure space.) 
\begin{theorem}[Direct limits of normed \(L^0\)-modules]\label{thm:direct_limit}
Let \((\X,\Sigma,\mm)\) be any \(\sigma\)-finite measure space.
Let \(\big(\{\mathscr M_i\}_{i\in I},\{\varphi_{ij}\}_{i\leq j}\big)\)
be a direct system of normed \(L^0(\mm)\)-modules, \emph{i.e.},
\begin{itemize}
\item[\(\rm i)\)] \((I,\leq)\) is a directed set,
\item[\(\rm ii)\)] \(\{\mathscr M_i\,:\,i\in I\}\) is a family of normed
\(L^0(\mm)\)-modules,
\item[\(\rm iii)\)] \(\{\varphi_{ij}\,:\,i,j\in I,\,i\leq j\}\) is a family of normed
\(L^0(\mm)\)-module morphisms \(\varphi_{ij}\colon\mathscr M_i\to\mathscr M_j\)
such that \(\varphi_{ii}={\rm id}_{\mathscr M_i}\) for all \(i\in I\) and
\(\varphi_{ik}=\varphi_{jk}\circ\varphi_{ij}\) for all \(i,j,k\in I\) with
\(i\leq j\leq k\).
\end{itemize}
Then there exists a unique couple
\(\big(\varinjlim\mathscr M_\star,\{\varphi_i\}_{i\in I}\big)\),
where \(\varinjlim\mathscr M_\star\) is a normed \(L^0(\mm)\)-module
and each map \(\varphi_i\colon\mathscr M_i\to\varinjlim\mathscr M_\star\)
is a normed \(L^0(\mm)\)-module morphism, such that:
\begin{itemize}
\item[\(\rm a)\)] \(\big(\varinjlim\mathscr M_\star,\{\varphi_i\}_{i\in I}\big)\)
is a target for \(\big(\{\mathscr M_i\}_{i\in I},\{\varphi_{ij}\}_{i\leq j}\big)\),
\emph{i.e.},
\[\begin{tikzcd}
\mathscr M_i \arrow{r}{\varphi_{ij}} \arrow[swap]{rd}{\varphi_i}
& \mathscr M_j \arrow{d}{\varphi_j} \\
& \varinjlim\mathscr M_\star
\end{tikzcd}\]
is a commutative diagram for every \(i,j\in I\) with \(i\leq j\).
\item[\(\rm b)\)] Given any other target \(\big(\mathscr N,\{\psi_i\}_{i\in I}\big)\)
for \(\big(\{\mathscr M_i\}_{i\in I},\{\varphi_{ij}\}_{i\leq j}\big)\), there
exists a unique normed \(L^0(\mm)\)-module morphism
\(\Phi\colon\varinjlim\mathscr M_\star\to\mathscr N\) such that
\[\begin{tikzcd}
\mathscr M_i \arrow{r}{\varphi_i} \arrow[swap]{rd}{\psi_i}
& \varinjlim\mathscr M_\star \arrow{d}{\Phi} \\
& \mathscr N
\end{tikzcd}\]
is a commutative diagram for every \(i\in I\).
\end{itemize}
\end{theorem}
A simple example of direct limit of normed \(L^0(\mm)\)-modules is
given by the ensuing result:
\begin{lemma}\label{lem:countable_DL}
Let \((\X,\Sigma,\mm)\) be a \(\sigma\)-finite measure space.
Let \(\mathscr M\) be a normed \(L^0(\mm)\)-module. Let \((\mathscr M_n)_n\) be
an increasing sequence of normed \(L^0(\mm)\)-submodules of \(\mathscr M\) such
that \(\bigcup_{n\in\N}\mathscr M_n\) is dense in \(\mathscr M\).
Call \(\iota_{nm}\colon\mathscr M_n\hookrightarrow\mathscr M_m\) the inclusion map
for every \(n,m\in\N\) with \(n\leq m\). Then
\(\big(\{\mathscr M_n\}_{n\in\N},\{\iota_{nm}\}_{n\leq m}\big)\) is a direct
system of normed \(L^0(\mm)\)-modules and
\[
\varinjlim\mathscr M_\star\cong\mathscr M.
\]
\end{lemma}
\begin{proof}
Calling \(\iota_n\colon\mathscr M_n\hookrightarrow\mathscr M\) the inclusion map
for any \(n\in\N\), we see that \(\big(\mathscr M,\{\iota_n\}_{n\in\N}\big)\)
is a target for \(\big(\{\mathscr M_n\}_{n\in\N},\{\iota_{nm}\}_{n\leq m}\big)\).
Moreover, fix any other target \(\big(\mathscr N,\{\psi_n\}_{n\in\N}\big)\).
Since the vector space \(\bigcup_{n\in\N}\iota_n(\mathscr M_n)\)
is \(\sfd_{\mathscr M}\)-dense in \(\mathscr M\) by assumption,
there clearly exists a unique linear continuous map \(\Phi\colon\mathscr M\to\mathscr N\)
such that \(\Phi(v)=\psi_n(v)\) for every \(n\in\N\) and \(v\in\mathscr M_n\).
Finally, this map can be readily proven to be a morphism of normed \(L^0(\mm)\)-modules.
\end{proof}
We conclude the subsection by reminding the key notion
of pullback of a normed \(L^0\)-module:
\begin{theorem}[Pullback of normed \(L^0\)-modules]
\label{thm:pullback_mod}
Let \((\X,\Sigma_\X,\mm_\X)\) and \((\Y,\Sigma_\Y,\mm_\Y)\) be
\(\sigma\)-finite measure spaces. Let \(\varphi\colon\X\to\Y\)
be a measurable map satisfying \(\varphi_*\mm_\X\ll\mm_\Y\).
Let \(\mathscr M\) be a normed \(L^0(\mm_\Y)\)-module.
Then there exists a unique couple \((\varphi^*\mathscr M,\varphi^*)\)
such that \(\varphi^*\mathscr M\) is a normed \(L^0(\mm_\X)\)-module
and \(\varphi^*\colon\mathscr M\to\varphi^*\mathscr M\) is a linear map
with the following properties:
\begin{itemize}
\item[\(\rm i)\)] It holds that \(|\varphi^*v|=|v|\circ\varphi\) in the
\(\mm_\X\)-a.e.\ sense for every \(v\in\mathscr M\).
\item[\(\rm ii)\)] The family \(\{\varphi^*v\,:\,v\in\mathscr M\}\)
generates \(\varphi^*\mathscr M\).
\end{itemize}
Uniqueness has to be intended up to unique isomorphism: given any other couple
\((\mathscr N,T)\) satisfying the same properties, there exists a unique
isomorphism \(\Phi\colon\varphi^*\mathscr M\to\mathscr N\) of normed
\(L^0(\mm_\X)\)-modules such that
\[\begin{tikzcd}
\mathscr M \arrow{r}{\varphi^*} \arrow[swap]{rd}{T}
& \varphi^*\mathscr M \arrow{d}{\Phi} \\
& \mathscr N
\end{tikzcd}\]
is a commutative diagram.
\end{theorem}
The pullback of a normed module has been introduced in \cite{Gigli14},
but the variant for normed \(L^0\)-modules presented above has been
considered in \cite{GR17} and \cite{Ben18}.
\subsection{Normed \texorpdfstring{\(L^\infty(\mm)\)}{Linfty(m)}-modules}
\label{ss:norm_Linftym_mod}
We recall the notion of normed \(L^\infty(\mm)\)-module, which
has been introduced in \cite{Gigli14}:
\begin{definition}[Normed \(L^\infty(\mm)\)-module]\label{def:normed_Linfty_mod}
Let \((\X,\Sigma,\mm)\) be any \(\sigma\)-finite measure space.
Then a \emph{normed \(L^\infty(\mm)\)-module} is a couple \(\big(\mathscr M,|\cdot|\big)\)
having the following properties:
\begin{itemize}
\item[\(\rm i)\)] \(\mathscr M\) is an algebraic \(L^\infty(\mm)\)-module.
\item[\(\rm ii)\)] The map \(|\cdot|\colon\mathscr M\to L^\infty(\mm)\), which is
called a \emph{pointwise norm} on \(\mathscr M\), satisfies
\begin{equation}\label{eq:ptwse_norm_infty}\begin{split}
|v|\geq 0&\quad\text{ for every }v\in\mathscr M
\text{, with equality if and only if }v=0,\\
|v+w|\leq|v|+|w|&\quad\text{ for every }v,w\in\mathscr M,\\
|f\cdot v|=|f||v|&\quad\text{ for every }f\in L^\infty(\mm)\text{ and }v\in\mathscr M,
\end{split}\end{equation}
where all (in)equalities are intended in the \(\mm\)-a.e.\ sense.
\item[\(\rm iii)\)] The space \(\mathscr M\) has the \emph{glueing property}, \emph{i.e.},
given a partition \((A_n)_{n\in\N}\subseteq\Sigma\) of \(\X\) and a sequence
\((v_n)_{n\in\N}\subseteq\mathscr M\) with \(\sup_n{\rm ess\,sup}_{A_n}|v_n|<+\infty\),
there is (a unique) \(v\in\mathscr M\) such that
\([\nchi_{A_n}]_\mm\cdot v=[\nchi_{A_n}]_\mm\cdot v_n\) for all \(n\in\N\).
The element \(v\) is denoted by \(\sum_{n\in\N}[\nchi_{A_n}]_\mm\cdot v_n\).
\item[\(\rm iv)\)] The norm \({\|\cdot\|}_{\mathscr M}\) on \(\mathscr M\)
associated with \(|\cdot|\), which is defined as
\[
{\|v\|}_{\mathscr M}\coloneqq{\big\||v|\big\|}_{L^\infty(\mm)}
\quad\text{ for every }v\in\mathscr M,
\]
is complete.
\end{itemize}
\end{definition}
We point out that in iii) we do not require
\(\lim_N\big\|\sum_{n=1}^N[\nchi_{A_n}]_\mm\cdot v_n-
\sum_{n\in\N}[\nchi_{A_n}]_\mm\cdot v_n\big\|_{\mathscr M}=0\).
\begin{remark}{\rm
The above notion of normed \(L^\infty(\mm)\)-module is equivalent to the
concept of \emph{\(L^\infty(\mm)\)-normed \(L^\infty(\mm)\)-module} introduced
in \cite{Gigli14}. Here we do not specify -- for the sake of brevity --
that the pointwise norm operator takes values into the space \(L^\infty(\mm)\),
the reason being that in the present manuscript this is the only type
of pointwise norm we are going to consider over \(L^\infty(\mm)\)-modules.
However, in \cite{Gigli14,Gigli17} also
\(L^p(\mm)\)-normed \(L^\infty(\mm)\)-modules, for any exponent \(p\in[1,\infty)\),
are studied.
\fr}\end{remark}
\begin{remark}[Locality property]\label{rmk:locality_property}{\rm
It can be readily checked that normed \(L^\infty(\mm)\)-modules
have the \emph{locality property}: if \((A_n)_{n\in\N}\subseteq\Sigma\) is a partition
of \(\X\) and \(v,w\in\mathscr M\) are two elements such that
\([\nchi_{A_n}]_\mm\cdot v=[\nchi_{A_n}]_\mm\cdot w\) for every \(n\in\N\),
then \(v=w\). This ensures that in item iii) of Definition \ref{def:normed_Linfty_mod}
the element \(\sum_{n\in\N}[\nchi_{A_n}]_\mm\cdot v_n\) is uniquely determined.
\fr}\end{remark}
\begin{remark}[Lack of glueing]\label{rmk:lack_glueing}{\rm
We point out that item iii) of Definition \ref{def:normed_Linfty_mod} is not
granted by i), ii), and iv). For instance, let us consider the Radon measure
\(\mm\coloneqq\sum_{n\in\N}\delta_n\) on \(\N\) and the space \(c_0\) of
all real-valued sequences that converge to \(0\), which is an algebraic module
over the ring \(\ell^\infty\cong L^\infty(\mm)\). We define the
pointwise norm \(|\cdot|\colon c_0\to\ell^\infty\) as
\(\big|(a_n)_n\big|\coloneqq\big(|a_n|\big)_n\) for every \((a_n)_n\in c_0\),
which clearly satisfies items ii) and iv). Nevertheless, the glueing property fails:
the elements \(e_n\coloneqq(\delta_{in})_{i\in\N}\) with \(n\in\N\) belong to \(c_0\),
but by `glueing' them we would obtain the sequence constantly equal to \(1\),
which is not an element of \(c_0\).
\fr}\end{remark}

We now aim to investigate the relation between normed \(L^\infty(\mm)\)-modules
and normed \(L^0(\mm)\)-modules. For the sake of clarity, we
denote by \(\mathscr M^\infty\) the former, by \(\mathscr M^0\) the latter.
\begin{definition}[Completion/restriction]\label{def:compl/restr}
Let \((\X,\Sigma,\mm)\) be a \(\sigma\)-finite measure space.
\begin{itemize}
\item[\(\rm i)\)] \textsc{Completion.} Let \(\mathscr M^\infty\) be a normed
\(L^\infty(\mm)\)-module. Then its \emph{completion} \(\sfC(\mathscr M^\infty)\)
is defined as the metric completion of
\(\big(\mathscr M^\infty,\sfd_{\sfC(\mathscr M^\infty)}\big)\),
where the distance \(\sfd_{\sfC(\mathscr M^\infty)}\) is given by
\(\sfd_{\sfC(\mathscr M^\infty)}(v,w)\coloneqq\sfd_{L^0(\mm)}\big(|v-w|,0\big)\)
for every \(v,w\in\mathscr M^\infty\).
\item[\(\rm ii)\)] \textsc{Restriction.} Let \(\mathscr M^0\) be a normed
\(L^0(\mm)\)-module. Then its \emph{restriction} \(\sfR(\mathscr M^0)\) is defined as
\(\sfR(\mathscr M^0)\coloneqq\big\{v\in\mathscr M^0\,:\,|v|\in L^\infty(\mm)\big\}\).
\end{itemize}
\end{definition}
It can be readily checked -- by arguing as
in \cite[Theorem/Definition 2.7]{Gigli17} -- that \(\sfC(\mathscr M^\infty)\)
inherits a normed \(L^0(\mm)\)-module structure. Moreover, \(\sfR(\mathscr M^0)\)
is a normed \(L^\infty(\mm)\)-module.
\bigskip

The following result says that `the completion map is the inverse
of the restriction map':
\begin{lemma}[`\(\sfC=\sfR^{-1}\)']\label{lem:C=inverse_R}
Let \((\X,\Sigma,\mm)\) be a \(\sigma\)-finite measure space. Then it holds that:
\begin{itemize}
\item[\(\rm i)\)] \(\sfC\big(\sfR(\mathscr M^0)\big)\cong\mathscr M^0\)
for every normed \(L^0(\mm)\)-module \(\mathscr M^0\),
\item[\(\rm ii)\)] \(\sfR\big(\sfC(\mathscr M^\infty)\big)\cong\mathscr M^\infty\)
for every normed \(L^\infty(\mm)\)-module \(\mathscr M^\infty\).
\end{itemize}
\end{lemma}
\begin{proof}\ \\
{\color{blue}i)} Let \(\mathscr M^0\) be a normed \(L^0(\mm)\)-module.
Observe that \(\sfC\big(\sfR(\mathscr M^0)\big)\) can be identified with the
\(\sfd_{\mathscr M^0}\)-closure of \(\sfR(\mathscr M^0)\) in \(\mathscr M^0\),
thus to conclude it suffices to show that \(\sfR(\mathscr M^0)\) is
\(\sfd_{\mathscr M^0}\)-dense in \(\mathscr M^0\). To this aim, fix
\(v\in\mathscr M^0\) and call \(A_n\coloneqq\big\{|v|\leq n\big\}\) for every \(n\in\N\).
Hence, it clearly holds that
\(\big([\nchi_{A_n}]_\mm\cdot v\big)_n\subseteq\sfR(\mathscr M^0)\)
and \(\lim_n\sfd_{\mathscr M^0}\big([\nchi_{A_n}]_\mm\cdot v,v)=0\),
as required.\\
{\color{blue}ii)} Let \(\mathscr M^\infty\) be a normed \(L^\infty(\mm)\)-module.
We call \(\iota\) the \(L^\infty(\mm)\)-linear isometric embedding of
\((\mathscr M^\infty,\sfd_{\sfC(\mathscr M^\infty)})\) into \(\sfC(\mathscr M^\infty)\).
Notice that \(\iota(\mathscr M^\infty)\subseteq\sfR\big(\sfC(\mathscr M^\infty)\big)\)
by definition of \(\sfR\). To conclude, it is enough to prove that actually
\(\iota(\mathscr M^\infty)=\sfR\big(\sfC(\mathscr M^\infty)\big)\).
Fix any \(w\in\sfR\big(\sfC(\mathscr M^\infty)\big)\) and \(\eps>0\).
Pick a sequence \((v_n)_n\subseteq\mathscr M^\infty\) with
\(\lim_n\sfd_{\sfC(\mathscr M^\infty)}\big(\iota(v_n),w\big)=0\).
By using Egorov theorem, we can find a partition
\((A_i)_{i\in\N}\subseteq\Sigma\) of \(\X\) and a sequence \((n_i)_{i\in\N}\subseteq\N\)
such that
\[
\underset{A_i}{\rm ess\,sup}\,\big|\iota(v_{n_i})-w\big|\leq\eps
\quad\text{ for every }i\in\N.
\]
Notice that \(\sup_i{\rm ess\,sup}_{A_i}|v_{n_i}|\leq\eps+{\rm ess\,sup}_\X|w|<+\infty\),
thus the glueing property of \(\mathscr M^\infty\) grants the existence
of \(v\coloneqq\sum_{i\in\N}[\nchi_{A_i}]_\mm\cdot v_{n_i}\).
It holds that \(\sum_{i=1}^k[\nchi_{A_i}]_\mm\cdot v_{n_i}\to v\) as \(k\to\infty\)
with respect to the distance \(\sfd_{\sfC(\mathscr M^\infty)}\)
(but not with respect to \({\|\cdot\|}_{\mathscr M^\infty}\), in general).
This ensures that \(\sum_{i=1}^k[\nchi_{A_i}]_\mm\cdot\iota(v_{n_i})=
\iota\big(\sum_{i=1}^k[\nchi_{A_i}]_\mm\cdot v_{n_i}\big)\to\iota(v)\) as \(k\to\infty\),
thus accordingly the inequality \(\big|\iota(v)-w\big|\leq\eps\) holds \(\mm\)-a.e..
We conclude that \(\iota(\mathscr M^\infty)=\sfR\big(\sfC(\mathscr M^\infty)\big)\).
\end{proof}

Let us mention that both the correspondences
\(\mathscr M^\infty\mapsto\sfC(\mathscr M^\infty)\) and
\(\mathscr M^0\mapsto\sfR(\mathscr M^0)\) can be made into
functors, which turn out to be equivalences of categories
-- one the inverse of the other. We omit the details, referring
to \cite[Appendix B]{LP18} for a similar discussion.
\section{Liftings of normed modules}\label{s:lift_norm_mod}
Aim of this section is to generalise the theory of liftings to the
setting of normed modules. In \S\ref{ss:norm_LinftySigma_mod}
we introduce and study a notion of normed module over
\(\mathcal L^\infty(\Sigma)\), whose elements are
`everywhere defined'. In \S\ref{ss:lift_mod} we prove that every
normed \(L^\infty(\mm)\)-module can be lifted to a normed
\(\mathcal L^\infty(\Sigma)\)-module. In \S\ref{ss:fibers} we
focus our attention on the functional-analytic properties of
the fibers of a normed \(\mathcal L^\infty(\Sigma)\)-module.
\subsection{Definition of normed
\texorpdfstring{\(\mathcal L^\infty(\Sigma)\)}
{mathcalLinfty(Sigma)}-module.}\label{ss:norm_LinftySigma_mod}
We propose a notion of normed module over the
commutative ring \(\mathcal L^\infty(\Sigma)\):
\begin{definition}[Normed \(\mathcal L^\infty(\Sigma)\)-module]
\label{def:normed_Lsigma_mod}
Let \((\X,\Sigma,\mm)\) be a \(\sigma\)-finite measure space.
Then a \emph{normed \(\mathcal L^\infty(\Sigma)\)-module}
is a couple \(\big(\bar{\mathscr M},|\cdot|\big)\)
that satisfies the following properties:
\begin{itemize}
\item[\(\rm i)\)] \(\bar{\mathscr M}\) is an algebraic \(\mathcal L^\infty(\Sigma)\)-module.
\item[\(\rm ii)\)] The map \(|\cdot|\colon\bar{\mathscr M}\to\mathcal L^\infty(\Sigma)\),
which is called a \emph{pointwise norm} on \(\bar{\mathscr M}\),
satisfies
\begin{equation}\label{eq:bar_ptwse_norm}\begin{split}
|\bar v|\geq 0&\quad\text{ for every }\bar v\in\bar{\mathscr M}
\text{, with equality if and only if }\bar v=0,\\
|\bar v+\bar w|\leq|\bar v|+|\bar w|&\quad
\text{ for every }\bar v,\bar w\in\bar{\mathscr M},\\
|\bar f\cdot\bar v|=|\bar f||\bar v|&\quad
\text{ for every }\bar f\in\mathcal L^\infty(\Sigma)
\text{ and }\bar v\in\bar{\mathscr M}.
\end{split}\end{equation}
\item[\(\rm iii)\)] \(\bar{\mathscr M}\) satisfies the \emph{glueing property},
\emph{i.e.}, given any partition \((A_n)_{n\in\N}\subseteq\Sigma\) of \(\X\) and a sequence
\((\bar v_n)_{n\in\N}\subseteq\bar{\mathscr M}\) such that
\(\sup_n\sup_{A_n}|\bar v_n|<+\infty\),
there is (a unique) \(\bar v\in\bar{\mathscr M}\) such that
\(\nchi_{A_n}\cdot\bar v=\nchi_{A_n}\cdot\bar v_n\) for all \(n\in\N\).
The element \(\bar v\) is denoted by \(\sum_{n\in\N}\nchi_{A_n}\cdot\bar v_n\).
\item[\(\rm iv)\)] The norm \({\|\cdot\|}_{\bar{\mathscr M}}\) on \(\bar{\mathscr M}\)
associated with \(|\cdot|\), which is defined as
\[
{\|\bar v\|}_{\bar{\mathscr M}}\coloneqq{\big\||\bar v|\big\|}_{\mathcal L^\infty(\Sigma)}
\quad\text{ for every }\bar v\in\bar{\mathscr M},
\]
is complete.
\end{itemize}
\end{definition}
\begin{remark}[Locality property]{\rm
It can be readily checked that normed \(\mathcal L^\infty(\Sigma)\)-modules
have the \emph{locality property}: if \((A_n)_{n\in\N}\subseteq\Sigma\) is a partition
of \(\X\) and \(\bar v,\bar w\in\bar{\mathscr M}\) are two elements such that
\(\nchi_{A_n}\cdot\bar v=\nchi_{A_n}\cdot\bar w\) for every \(n\in\N\),
then \(\bar v=\bar w\). This ensures that in item iii) of Definition
\ref{def:normed_Lsigma_mod} the element
\(\sum_{n\in\N}\nchi_{A_n}\cdot\bar v_n\) is uniquely determined.
\fr}\end{remark}
\begin{remark}[Lack of glueing]{\rm
The glueing property is not granted by items i), ii), iv)
of Definition \ref{def:normed_Lsigma_mod}, as shown by the same counterexample
we provided in Remark \ref{rmk:lack_glueing}.
\fr}\end{remark}

Given any normed \(\mathcal L^\infty(\Sigma)\)-module \(\bar{\mathscr M}\),
we can introduce the following equivalence relation: two elements
\(\bar v,\bar w\in\bar{\mathscr M}\) are equivalent -- shortly, \(\bar v\sim\bar w\)
-- provided \(|\bar v-\bar w|=0\) holds \(\mm\)-a.e..
Then we define the space \(\Pi_\mm(\bar{\mathscr M})\) as follows:
\begin{equation}\label{eq:def_proj_Pi_m}
\Pi_\mm(\bar{\mathscr M})\coloneqq\bar{\mathscr M}/\sim.
\end{equation}
Given any element \(\bar v\in\bar{\mathscr M}\), we will denote by
\([\bar v]_\sim\in\Pi_\mm(\bar{\mathscr M})\) its equivalence class
modulo \(\sim\). The canonical projection map \(\bar v\mapsto[\bar v]_\sim\)
will be denoted by \(\pi_\mm\colon\bar{\mathscr M}\to\Pi_\mm(\bar{\mathscr M})\).
\begin{lemma}
Let \((\X,\Sigma,\mm)\) be a \(\sigma\)-finite measure space.
Let \(\bar{\mathscr M}\) be a normed \(\mathcal L^\infty(\Sigma)\)-module.
Then \(\mathscr M\coloneqq\Pi_\mm(\bar{\mathscr M})\) is a
normed \(L^\infty(\mm)\)-module and \(\pi_\mm\colon\bar{\mathscr M}\to\mathscr M\)
is linear and continuous.
\end{lemma}
\begin{proof}
Given any \(v,w\in\mathscr M\) and \(f\in L^\infty(\mm)\) -- say
\(v=[\bar v]_\sim\), \(w=[\bar w]_\sim\), and \(f=[\bar f]_\mm\) -- we set
\[\begin{split}
v+w&\coloneqq[\bar v+\bar w]_\sim\in\mathscr M,\\
f\cdot v&\coloneqq[\bar f\cdot\bar v]_\sim\in\mathscr M,\\
|v|&\coloneqq\big[|\bar v|\big]_\mm\in L^\infty(\mm).
\end{split}\]
It can be readily checked that the above operations are well-posed,
meaning that they do not depend on the specific choice of the
representatives \(\bar v\), \(\bar w\), and \(\bar f\). Moreover, we have that
the operator \(|\cdot|\colon\mathscr M\to L^\infty(\mm)\)
satisfies \eqref{eq:ptwse_norm} as an immediate consequence
of \eqref{eq:bar_ptwse_norm}. To prove the glueing property
of \(\mathscr M\), fix a partition \((A_n)_{n\in\N}\subseteq\Sigma\) of \(\X\)
and a sequence \((v_n)_{n\in\N}\subseteq\mathscr M\) such that
\(\sup_n{\rm ess\,sup}_{A_n}|v_n|<+\infty\). Choose a representative
\(\bar v_n\) of \(v_n\) for each \(n\in\N\), so that there exists an \(\mm\)-negligible
set \(N\in\Sigma\) such that \(\sup_n\sup_{A_n\setminus N}|\bar v_n|<+\infty\).
Since \(\bar{\mathscr M}\) has the glueing property, there is
\(\bar v\in\bar{\mathscr M}\) such that
\(\nchi_{A_n\setminus N}\cdot\bar v=\nchi_{A_n\setminus N}\cdot\bar v_n\)
holds for all \(n\in\N\). Therefore,
\[
[\nchi_{A_n}]_\mm\cdot\pi_\mm(\bar v)
=\pi_\mm(\nchi_{A_n\setminus N}\cdot\bar v)
=\pi_\mm(\nchi_{A_n\setminus N}\cdot\bar v_n)
=[\nchi_{A_n}]_\mm\cdot v_n\quad\text{ holds for every }n\in\N,
\]
which shows that \(\mathscr M\) has the glueing property.
Now let us define \({\|v\|}_{\mathscr M}\coloneqq{\rm ess\,sup}_\X |v|\)
for every element \(v\in\mathscr M\). We aim to prove that
\(\big(\mathscr M,{\|\cdot\|}_{\mathscr M}\big)\) is a Banach space,
so fix a Cauchy sequence \((v_n)_n\subseteq\mathscr M\),
say \(v_n=[\bar v_n]_\sim\) for all \(n\).
Hence, there exists \(N\in\Sigma\) with \(\mm(N)=0\) such that
\(\sup_{\X\setminus N}|\bar v_n-\bar v_m|\to 0\) as \(n,m\to\infty\).
This means that the sequence \((\nchi_{\X\setminus N}\cdot\bar v_n)_{n\in\N}\)
is Cauchy in \(\bar{\mathscr M}\), thus
\({\|\nchi_{\X\setminus N}\cdot\bar v_n-\bar v\|}_{\bar{\mathscr M}}\to 0\)
for some \(\bar v\in\bar{\mathscr M}\). Then we have that
\[
{\big\|v_n-\pi_\mm(\bar v)\big\|}_{\mathscr M}
=\underset{\X}{\rm ess\,sup\,}\big|v_n-\pi_\mm(\bar v)\big|
\leq\sup_\X\big|\nchi_{\X\setminus N}\cdot\bar v_n-\bar v\big|
\longrightarrow 0\quad\text{ as }n\to\infty,
\]
as required. Finally, linearity and continuity of
\(\pi_\mm\colon\bar{\mathscr M}\to\mathscr M\) can be trivially verified.
\end{proof}
\subsection{Liftings of normed \texorpdfstring{\(L^\infty(\mm)\)}
{Linfty(m)}-modules}\label{ss:lift_mod}
The following result shows that any lifting of measurable functions
can be made into a `lifting of normed modules', much like in Theorem
\ref{thm:lifting_Linfty} we `raised' a lifting of a measure space to
a lifting of measurable functions. 
\begin{theorem}[Lifting of normed \(L^\infty(\mm)\)-modules]\label{thm:lifting_mod}
Let \((\X,\Sigma,\mm)\) be a \(\sigma\)-finite measure space.
Let \(\mathscr M\) be a normed \(L^\infty(\mm)\)-module.
Let \(\ell\) be a lifting of \(\mm\) and call
\(\mathcal L\colon L^\infty(\mm)\to\mathcal L^\infty(\Sigma)\)
its associated operator (as in Theorem \ref{thm:lifting_Linfty}).
Then there exists a unique couple \((\bar{\mathscr M},\mathscr L)\),
called the \emph{\(\mathcal L\)-lifting} of \(\mathscr M\),
where \(\bar{\mathscr M}\) is a normed \(\mathcal L^\infty(\Sigma)\)-module
and \(\mathscr L\colon\mathscr M\to\bar{\mathscr M}\) is a linear map
that satisfies the following properties:
\begin{itemize}
\item[\(\rm i)\)] It holds that \(\big|\mathscr L(v)\big|=\mathcal L\big(|v|\big)\)
for every \(v\in\mathscr M\).
\item[\(\rm ii)\)] The linear subspace \(\mathcal V\) of all elements \(\bar v\in\bar{\mathscr M}\)
of the form \(\bar v=\sum_{n\in\N}\nchi_{A_n}\cdot\mathscr L(v_n)\),
where \((A_n)_{n\in\N}\subseteq\Sigma\) is a partition of \(\X\) and
\((v_n)_{n\in\N}\subseteq\mathscr M\), is dense in \(\bar{\mathscr M}\).
\end{itemize}
Uniqueness is intended up to unique isomorphism: given any other couple
\((\bar{\mathscr N},\mathscr L')\) with the same properties, there exists
a unique \(\mathcal L^\infty(\Sigma)\)-module isomorphism
\(\Psi\colon\bar{\mathscr M}\to\bar{\mathscr N}\) preserving the pointwise norm such that
\begin{equation}\begin{tikzcd}
\mathscr M \arrow{r}{\mathscr L} \arrow[swap]{rd}{\mathscr L'}
& \bar{\mathscr M} \arrow{d}{\Psi} \\
 & \bar{\mathscr N}
\end{tikzcd}\end{equation}
is a commutative diagram.
\end{theorem}
\begin{proof}\ \\
{\color{blue}\textsc{Existence.}}
First of all, let us define the \emph{pre-module} \(\sf Pm\) as
\[ 
{\sf Pm}\coloneqq\bigg\{\big\{(A_n,v_n)\big\}_{n\in\N}\;\bigg|\;\begin{array}{ll}
(A_n)_{n\in\N}\subseteq\Sigma\text{ is a partition of }\X,\\
(v_n)_n\subseteq\mathscr M,\,
\sup\nolimits_n{\rm ess\,sup}_{A_n}|v_n|<\infty
\end{array}\bigg\}.
\]
We declare that \(\big\{(A_n,v_n)\big\}_n\sim\big\{(B_m,w_m)\big\}_m\) provided
\(\nchi_{A_n\cap B_m}\cdot\mathcal L\big(|v_n-w_m|\big)=0\) holds for every \(n,m\in\N\).
We denote by \([A_n,v_n]_n\) the equivalence class of the sequence \(\big\{(A_n,v_n)\big\}_n\)
with respect to the equivalence relation \(\sim\). We endow \({\sf Pm}/\sim\) with the
following operations:
\begin{equation}\label{eq:operations_Pm}\begin{split}
[A_n,v_n]_n+[B_m,w_m]_m&\coloneqq[A_n\cap B_m,v_n+w_m]_{n,m},\\
\bigg(\sum_{i=1}^k c_i\,\nchi_{C_i}\bigg)\cdot[A_n,v_n]_n&\coloneqq[A_n\cap C_i,c_i\, v_n]_{n,i},\\
\big|[A_n,v_n]_n\big|&\coloneqq\sum_{n\in\N}\nchi_{A_n}\,\mathcal L\big(|v_n|\big),\\
{\big\|[A_n,v_n]_n\big\|}_{\bar{\mathscr M}}&\coloneqq
\sup_\X\big|[A_n,v_n]_n\big|,
\end{split}\end{equation}
for every \([A_n,v_n]_n,[B_m,w_m]_m\in{\sf Pm}/\sim\) and
\(\sum_{i=1}^k c_i\,\nchi_{C_i}\in\overline{\sf Sf}(\Sigma)\).
Then we define \(\bar{\mathscr M}\) as the completion of the normed space
\(\big({\sf Pm}/\sim,{\|\cdot\|}_{\bar{\mathscr M}}\big)\).
It can be readily checked -- since the space
\(\overline{\sf Sf}(\Sigma)\) is dense in \(\mathcal L^\infty(\Sigma)\) -- that
the operations in \eqref{eq:operations_Pm} can be uniquely extended
to \(\bar{\mathscr M}\), which has a natural structure of
\(\mathcal L^\infty(\Sigma)\)-module endowed with a pointwise norm
\(|\cdot|\colon\bar{\mathscr M}\to\mathcal L^\infty(\Sigma)\) satisfying
items ii) and iv) of Definition \ref{def:normed_Lsigma_mod}. Define
\(\mathscr L\colon\mathscr M\to\bar{\mathscr M}\) as
\(\mathscr L(v)\coloneqq[\X,v]\in{\sf Pm}/\sim\)
for every \(v\in\mathscr M\). It is clear that \(\mathscr L\) is a linear
operator that satisfies i).
To show that \(\bar{\mathscr M}\) has the glueing property, fix
a sequence \((\bar w_n)_{n\in\N}\subseteq\bar{\mathscr M}\) and
a partition \((A_n)_{n\in\N}\subseteq\Sigma\) of \(\X\).
Given any \(n,k\in\N\), there exist a sequence
\((v^k_{n,i})_{i\in\N}\subseteq\mathscr M\) and a partition
\((A^k_{n,i})_{i\in\N}\subseteq\Sigma\) of \(A_n\) such that
\({\big\|[A^k_{n,i},v^k_{n,i}]_i-\nchi_{A_n}\cdot\bar w_n\big\|}_{\bar{\mathscr M}}
\leq 2^{-k}\). Call \(\bar z^k\coloneqq[A^k_{n,i},v^k_{n,i}]_{n,i}\in\bar{\mathscr M}\)
for every \(k\in\N\). Given that \(\sup_{A_n}|\bar z^{k+1}-\bar z^k|\leq 2^{-k+1}\)
for all \(n,k\in\N\), we have
\({\big\|\bar z^{k+1}-\bar z^k\big\|}_{\bar{\mathscr M}}\leq 2^{-k+1}\) for all \(k\in\N\).
Then \({(\bar z^k)}_{k\in\N}\) is a Cauchy sequence, whence it converges to some element
\(\bar w\in\bar{\mathscr M}\). Since for any \(n\in\N\) it holds that
\(\nchi_{A_n}\cdot\bar w=\lim_k\nchi_{A_n}\cdot\bar z^k=\nchi_{A_n}\cdot\bar w_n\),
the glueing property is proved. Finally, it only remains to show ii).
Notice that it is enough to prove that \(\mathcal V={\sf Pm}/\sim\).
To this aim, fix a partition \((A_n)_{n\in\N}\subseteq\Sigma\) of \(\X\)
and a sequence \((v_n)_n\subseteq\mathscr M\). We denote by \(\bar v\) the
element \(\sum_{n\in\N}\nchi_{A_n}\cdot\mathscr L(v_n)\in\bar{\mathscr M}\).
Then it holds that \(\bar v=[A_n,v_n]_n\): for any \(m\in\N\), we have
\[
\nchi_{A_m}\cdot\bar v=\nchi_{A_m}\cdot\mathscr L(v_m)=
\nchi_{A_m}\cdot[\X,v_m]=[A_m\cap A_n,v_n]_n=\nchi_{A_m}\cdot[A_n,v_n]_n.
\]
This proves that \(\mathcal V={\sf Pm}/\sim\) and accordingly that ii) is verified.\\
{\color{blue}\textsc{Uniqueness.}} First of all, observe that we are forced to set
\[
\Psi\Big(\sum_{n\in\N}\nchi_{A_n}\cdot\mathscr L(v_n)\Big)
\coloneqq\sum_{n\in\N}\nchi_{A_n}\cdot\mathscr L'(v_n)
\quad\text{ for every }\sum_{n\in\N}\nchi_{A_n}\cdot\mathscr L(v_n)\in\mathcal V.
\]
Its well-posedness stems from the fact that
\(\sum_{n\in\N}\nchi_{A_n}\cdot\mathscr L(v_n)\) and
\(\sum_{n\in\N}\nchi_{A_n}\cdot\mathscr L'(v_n)\) have the same pointwise norm by i).
Hence, by ii) the map \(\Psi\) can be uniquely extended to an
\(\mathcal L^\infty(\Sigma)\)-linear operator
\(\Psi\colon\bar{\mathscr M}\to\bar{\mathscr N}\) preserving
the pointwise norm. Finally, since \(\Psi(\mathcal V)\) is dense in
\(\bar{\mathscr N}\) again by ii), we conclude that \(\Psi\) is a normed
\(\mathcal L^\infty(\Sigma)\)-module isomorphism.
\end{proof}
\begin{remark}{\rm
We highlight an important byproduct of the proof of Theorem \ref{thm:lifting_mod}:
\begin{equation}\label{eq:L-lin}
\mathscr L(f\cdot v)=\mathcal L(f)\cdot\mathscr L(v)
\quad\text{ for every }v\in\mathscr M\text{ and }f\in L^\infty(\mm).
\end{equation}
By density of \({\sf Sf}(\mm)\) in \(L^\infty(\mm)\), it is enough to prove
it for \(f=\sum_{i=1}^n a_i\,[\nchi_{A_i}]_\mm\in{\sf Sf}(\mm)\). Since
\[
\nchi_{\ell(A_i)}\cdot\mathcal L\big(|a_i\,v-f\cdot v|\big)
=\mathcal L\big([\nchi_{A_i}]_\mm\big)\,\mathcal L\big(|a_i\,v-f\cdot v|\big)
=\mathcal L\big([\nchi_{A_i}]_\mm\,|a_i\,v-f\cdot v|\big)=0
\]
holds for every \(i=1,\ldots,n\), we see that
\(\big[\ell(A_i),a_i\,v\big]_i=[\X,f\cdot v]\). Therefore, we deduce
from the second line of \eqref{eq:operations_Pm} that
\[
\mathcal L(f)\cdot\mathscr L(v)=
\Big(\sum_{i=1}^n a_i\,\nchi_{\ell(A_i)}\Big)\cdot[\X,v]=
\big[\ell(A_i),a_i\,v\big]_i=[\X,f\cdot v]=\mathscr L(f\cdot v),
\]
as required.
\fr}\end{remark}

A natural question arises: given a normed \(L^\infty(\mm)\)-module \(\mathscr M\)
and calling \(\bar{\mathscr M}\) its \(\mathcal L\)-lifting, do
\(\Pi_\mm(\bar{\mathscr M})\) and \(\mathscr M\) coincide?
The following result shows that the answer is positive.
\begin{lemma}
Let \((\X,\Sigma,\mm)\) be a \(\sigma\)-finite measure space.
Let \(\mathscr M\) be a normed \(L^\infty(\mm)\)-module.
Let \(\ell\) be any lifting of \(\mm\), with associated operator
\(\mathcal L\colon L^\infty(\mm)\to\mathcal L^\infty(\Sigma)\).
Denote by \((\bar{\mathscr M},\mathscr L)\) the \(\mathcal L\)-lifting
of \(\mathscr M\). Then it holds that \(\Pi_\mm(\bar{\mathscr M})\cong\mathscr M\).
\end{lemma}
\begin{proof}
To prove the statement, it suffices to show that the map
\(T\coloneqq\pi_\mm\circ\mathscr L\colon\mathscr M\to\Pi_\mm(\bar{\mathscr M})\)
is an isomorphism of normed \(L^\infty(\mm)\)-modules.
We know that \(T\) is \(L^\infty(\mm)\)-linear: it is linear as composition
of linear operators, while for any \(f\in L^\infty(\mm)\) and \(v\in\mathscr M\) it holds that
\[
T(f\cdot v)=\pi_\mm\big(\mathscr L(f\cdot v)\big)\overset{\eqref{eq:L-lin}}=
\pi_\mm\big(\mathcal L(f)\cdot\mathscr L(v)\big)=
\big[\mathcal L(f)\cdot\mathscr L(v)\big]_\sim=
f\cdot\big[\mathscr L(v)\big]_\sim=f\cdot T(v).
\]
Furthermore, for every \(v\in\mathscr M\) we have that
\[
\big|T(v)\big|=\big|\big[\mathscr L(v)\big]_\sim\big|
=\big[\big|\mathscr L(v)\big|\big]_\mm=\big[\mathcal L\big(|v|\big)\big]_\mm
=|v|\quad\text{ holds }\mm\text{-a.e.\ on }\X,
\]
in other words, the map \(T\) preserves the pointwise norm.
In order to conclude, it suffices to prove that \(T\) is surjective.
Let \([\bar w]_\sim\in\Pi_\mm(\bar{\mathscr M})\) be fixed.
Then for any \(n\in\N\) we can pick an element
\(\bar w_n=\sum_{i\in\N}\nchi_{A^n_i}\cdot\mathscr L(v^n_i)\in\bar{\mathscr M}\)
in such a way that \(\lim_n{\|\bar w_n-\bar w\|}_{\bar{\mathscr M}}=0\).
Now let us set \(v_n\coloneqq\sum_{i\in\N}[\nchi_{A^n_i}]_\mm\cdot v^n_i\in\mathscr M\)
for all \(n\in\N\), which is well-defined as
\[
\sup_{i\in\N}\,\underset{A^n_i}{\rm ess\,sup}\,|v^n_i|
\leq\sup_{i\in\N}\,\sup_{A^n_i}\big|\mathscr L(v^n_i)\big|
={\|\bar w_n\|}_{\bar{\mathscr M}}<+\infty
\quad\text{ for every }n\in\N.
\]
Since for any \(n,m\in\N\) it holds that
\[
|v_n-v_m|=\sum_{i,j\in\N}[\nchi_{A^n_i\cap A^m_j}]_\mm\,|v^n_i-v^m_j|
=\sum_{i,j\in\N}[\nchi_{A^n_i\cap A^m_j}]_\mm\,
\Big[\big|\mathscr L(v^n_i)-\mathscr L(v^m_j)\big|\Big]_\mm
=\big[|\bar w_n-\bar w_m|\big]_\mm,
\]
we see that the sequence \((v_n)_{n\in\N}\subseteq\mathscr M\) is
Cauchy, thus it converges to some element \(v\in\mathscr M\).
In general \(\mathscr L(v)\) and \(\bar w\) may be different, but for sure one has
that \(\big[\big|\mathscr L(v)-\bar w\big|\big]_\mm=0\), indeed
\[
\Big[\big|\mathscr L(v)-\bar w\big|\Big]_\mm\leq
\Big[\big|\mathscr L(v-v_n)\big|\Big]_\mm+
\Big[\big|\mathscr L(v_n)-\bar w_n\big|\Big]_\mm+
\big[|\bar w_n-\bar w|\big]_\mm=
|v-v_n|+\big[|\bar w_n-\bar w|\big]_\mm
\]
is satisfied \(\mm\)-a.e.\ and \(|v-v_n|+\big[|\bar w_n-\bar w|\big]_\mm\to 0\)
with respect to the \(L^\infty(\mm)\)-norm. This means that \(T(v)=[\bar w]_\sim\),
whence the operator \(T\) is surjective, as desired.
\end{proof}
\subsection{Fibers of a normed
\texorpdfstring{\(\mathcal L^\infty(\Sigma)\)}
{Linfty(Sigma)}-module}\label{ss:fibers}
Let \((\X,\Sigma,\mm)\) be a \(\sigma\)-finite measure space and let
\(\bar{\mathscr M}\) be a normed \(\mathcal L^\infty(\Sigma)\)-module.
Given any \(x\in\X\), we define the submodule \(\bar{\mathscr M}_x\)
of \(\bar{\mathscr M}\) as
\[
\bar{\mathscr M}_x\coloneqq\nchi_{\{x\}}\cdot\bar{\mathscr M}.
\]
Since the ideal \((\nchi_{\{x\}})\subseteq\mathcal L^\infty(\Sigma)\) generated by
\(\nchi_{\{x\}}\) can be identified with the real field \(\R\), we deduce that
\(\bar{\mathscr M}_x\) inherits a vector space structure. Moreover, let us define
\[
{\|\bar v\|}_x\coloneqq|\bar v|(x)\quad\text{ for every }\bar v\in\bar{\mathscr M}_x.
\]
Therefore, \(\big(\bar{\mathscr M}_x,{\|\cdot\|}_x\big)\) is a Banach space.
We call it the \emph{fiber} of \(\bar{\mathscr M}\) over the point \(x\).
\begin{remark}{\rm
Given any element \(\bar v\in\bar{\mathscr M}\), we shall make use of the shorthand notation
\[
\bar v_x\coloneqq\nchi_{\{x\}}\cdot\bar v\in\bar{\mathscr M}_x
\quad\text{ for every }x\in\X.
\]
Then \(\bar v\in\bar{\mathscr M}\) can be thought of as a map assigning
to any \(x\in\X\) a vector \(\bar v_x\in\bar{\mathscr M}_x\).
\fr}\end{remark}

Fix a normed \(L^0(\mm)\)-module \(\mathscr M\) and consider the duality pairing
\begin{equation}\label{eq:duality_pairing_L0}
\la\cdot,\cdot\ra\colon\mathscr M^*\times\mathscr M\longrightarrow L^0(\mm),
\quad\la\omega,v\ra\coloneqq\omega(v)\;\;\;\text{for every }
\omega\in\mathscr M^*\text{ and }v\in\mathscr M.
\end{equation}
Choose any lifting \(\ell\) of \(\mm\). Denote by \(\mathcal L\)
the operator associated with \(\ell\) as in Theorem \ref{thm:lifting_Linfty}.
Call \((\bar{\mathscr M},\mathscr L)\) and \((\bar{\mathscr N},\mathscr L^*)\)
the \(\mathcal L\)-liftings of the normed \(L^\infty(\mm)\)-modules
\(\sfR(\mathscr M)\) and \(\sfR(\mathscr M^*)\), respectively.
Notice that the duality pairing in \eqref{eq:duality_pairing_L0}
restricts to an \(L^\infty(\mm)\)-bilinear and continuous map
\(\la\cdot,\cdot\ra\colon\sfR(\mathscr M^*)\times\sfR(\mathscr M)\to L^\infty(\mm)\).
We now want to lift it to a duality pairing between \(\bar{\mathscr M}\)
and \(\bar{\mathscr N}\). Given any two sequences \((v_n)_{n\in\N}\subseteq\sfR(\mathscr M)\),
\((\omega_m)_{m\in\N}\subseteq\sfR(\mathscr M^*)\) and any two partitions
\((A_n)_{n\in\N},(B_m)_{m\in\N}\subseteq\Sigma\) of \(\X\), let us define
\begin{equation}\label{eq:def_pairing_lift}
\bigg\langle\sum_{m\in\N}\nchi_{B_m}\cdot\mathscr L^*(\omega_m),
\sum_{n\in\N}\nchi_{A_n}\cdot\mathscr L(v_n)\bigg\rangle
\coloneqq\sum_{n,m\in\N}\nchi_{A_n\cap B_m}\,\mathcal L\big(\la\omega_m,v_n\ra\big)
\in\mathcal L^\infty(\Sigma).
\end{equation}
The well-posedness of the previous definition stems from the following inequality:
\[\begin{split}
\bigg|\sum_{n,m\in\N}\nchi_{A_n\cap B_m}\,\mathcal L\big(\la\omega_m,v_n\ra\big)\bigg|
&=\sum_{n,m\in\N}\nchi_{A_n\cap B_m}\,\mathcal L\big(\big|\la\omega_m,v_n\ra\big|\big)\\
&\leq\sum_{n,m\in\N}\nchi_{A_n\cap B_m}\,\mathcal L\big(|\omega_m||v_n|\big)\\
&=\sum_{n,m\in\N}\nchi_{A_n\cap B_m}\,\mathcal L\big(|\omega_m|\big)
\,\mathcal L\big(|v_n|\big)\\
&=\sum_{n,m\in\N}\nchi_{A_n\cap B_m}\,\big|\mathscr L^*(\omega_m)\big|
\,\big|\mathscr L(v_n)\big|\\
&=\bigg|\sum_{m\in\N}\nchi_{B_m}\cdot\mathscr L^*(\omega_m)\bigg|\,
\bigg|\sum_{n\in\N}\nchi_{A_n}\cdot\mathscr L(v_n)\bigg|.
\end{split}\]
The same inequality grants that the map in \eqref{eq:def_pairing_lift}
can be uniquely extended to a pairing
\[
\la\cdot,\cdot\ra\colon\bar{\mathscr N}\times\bar{\mathscr M}
\longrightarrow\mathcal L^\infty(\Sigma),
\]
\emph{i.e.}, to an \(\mathcal L^\infty(\Sigma)\)-bilinear and continuous map
satisfying the inequality
\begin{equation}\label{eq:pairing_lift_ineq}
\big|\la\bar\omega,\bar v\ra\big|\leq|\bar\omega||\bar v|
\quad\text{ everywhere on }\X
\end{equation}
for every \(\bar\omega\in\bar{\mathscr N}\) and \(\bar v\in\bar{\mathscr M}\).
\begin{remark}{\rm
We underline that
\begin{equation}\label{eq:pairing_mult}
\big\langle\mathscr L^*(\omega),\mathscr L(v)\big\rangle=
\mathcal L\big(\omega(v)\big)\quad\text{ for every }
v\in\sfR(\mathscr M)\text{ and }\omega\in\sfR(\mathscr M^*),
\end{equation}
by the very definition \eqref{eq:def_pairing_lift} of \(\la\cdot,\cdot\ra\colon
\bar{\mathscr N}\times\bar{\mathscr M}\to\mathcal L^\infty(\Sigma)\).
\fr}\end{remark}
\bigskip

Now fix a point \(x\in\X\). Then the pairing \(\la\cdot,\cdot\ra\colon
\bar{\mathscr N}\times\bar{\mathscr M}\to\mathcal L^\infty(\Sigma)\)
naturally induces a pairing
\(\la\cdot,\cdot\ra_x\colon\bar{\mathscr N}_x\times\bar{\mathscr M}_x\to\R\),
which is the bilinear and continuous function given by
\[
\la\bar\omega,\bar v\ra_x\coloneqq\la\bar\omega,\bar v\ra(x)\quad
\text{ for every }\bar\omega\in\bar{\mathscr N}_x\text{ and }
\bar v\in\bar{\mathscr M}_x.
\]
Observe that \(\bar{\mathscr M}_x\ni\bar v\mapsto\la\bar\omega,\bar v\ra_x\in\R\)
is a linear and continuous function for every \(\bar\omega\in\bar{\mathscr N}_x\).

Therefore, it makes sense to define the map
\({\rm R}_x\colon\bar{\mathscr N}_x\to(\bar{\mathscr M}_x)'\) as
\begin{equation}\label{eq:def_R_x}
{\rm R}_x(\bar\omega)\coloneqq\la\bar\omega,\cdot\ra_x
\quad\text{ for every }\bar\omega\in\bar{\mathscr N}_x.
\end{equation}
It is clear that \({\rm R}_x\) is a linear operator.
\begin{proposition}\label{prop:Rx_isometry}
Under the above assumptions, it holds that the map
\({\rm R}_x\colon\bar{\mathscr N}_x\to(\bar{\mathscr M}_x)'\) is an isometric embedding.
\end{proposition}
\begin{proof}
First of all, given any \(\bar\omega\in\bar{\mathscr N}_x\) it holds that
\(\big|\la\bar\omega,\bar v\ra_x\big|\leq{\|\bar\omega\|}_x{\|\bar v\|}_x\) for every
\(\bar v\in\bar{\mathscr M}_x\) as a consequence of \eqref{eq:pairing_lift_ineq},
whence accordingly
\begin{equation}\label{eq:Rx_isom_aux}
{\big\|{\rm R}_x(\bar\omega)\big\|}_{(\bar{\mathscr M}_x)'}\leq{\|\bar\omega\|}_x
\quad\text{ for every }\bar\omega\in\bar{\mathscr N}_x.
\end{equation}
In particular, the operator \({\rm R}_x\) is continuous. Then to prove the statement
it suffices to show that \({\rm R}_x\) is an isometry when restricted to any
dense subset of \(\bar{\mathscr N}_x\).

Observe that \(D\coloneqq\big\{\mathscr L^*(\omega)_x\,:\,\omega\in\sfR(\mathscr M^*)\big\}\) is a
dense subspace of \(\bar{\mathscr N}_x\) by the very definition of \(\mathcal L\)-lifting.
Fix any \(\omega\in\sfR(\mathscr M^*)\). Then it holds that
\begin{equation}\label{eq:Rx_isom_aux2}
\begin{split}
{\Big\|{\rm R}_x\big(\mathscr L^*(\omega)_x\big)\Big\|}_{(\bar{\mathscr M}_x)'}
&\overset{\phantom{\eqref{eq:pairing_mult}}}=
\sup_{\substack{\bar v\in\bar{\mathscr M}_x: \\ {\|\bar v\|}_x\leq 1}}
\big\langle\mathscr L^*(\omega)_x,\bar v\big\rangle_x
\geq\sup_{\substack{v\in\sfR(\mathscr M): \\ |v|\leq 1\ \mm\text{-a.e.}}}
\big\langle\mathscr L^*(\omega),\mathscr L(v)\big\rangle(x)\\
&\overset{\eqref{eq:pairing_mult}}=
\sup_{\substack{v\in\sfR(\mathscr M): \\ |v|\leq 1\ \mm\text{-a.e.}}}
\mathcal L\big(\omega(v)\big)(x).
\end{split}\end{equation}
Given any \(\eps>0\), there exists \(v\in\sfR(\mathscr M)\) such that \(|v|\leq 1\)
and \(\omega(v)\geq|\omega|-\eps\) are satisfied \(\mm\)-a.e.\ by \eqref{eq:dual_ptwse_norm}.
Thus \(\mathcal L\big(\omega(v)\big)\geq\mathcal L\big(|\omega|\big)-\eps\) and accordingly
\(\mathcal L\big(\omega(v)\big)(x)\geq{\big\|\mathscr L^*(\omega)_x\big\|}_x-\eps\).
This fact -- if read in conjunction with \eqref{eq:Rx_isom_aux} and
\eqref{eq:Rx_isom_aux2} -- ensures that the map \({\rm R}_x\) is an isometry
when restricted to \(D\), thus proving the statement.
\end{proof}
\begin{problem}
Are the maps \({\rm R}_x\colon\bar{\mathscr N}_x\to(\bar{\mathscr M}_x)'\)
defined in \eqref{eq:def_R_x} isomorphisms?
\end{problem}
\section{Representation of normed modules via embedding}\label{s:sep_Ban_bundle}
In this section we study separable Banach \(\B\)-bundles and
their strong connections with separable normed \(L^0(\mm)\)-modules.
Is \S\ref{ss:def_sep_Ban_bundle} we introduce our notion of separable
Banach bundle and we describe a constructive procedure to obtain one
such bundle. In \S\ref{ss:section_functor} we show that the space
of sections of a given bundle is a normed \(L^0(\mm)\)-module; also,
we extend the operation of `taking sections' to the level of categories,
thus obtaining a well-behaved section functor. In \S\ref{ss:repr_thm}
we prove that the section functor is actually an equivalence of
categories; this means, roughly speaking, that there is a full
correspondence between separable Banach bundles and separable
normed \(L^0(\mm)\)-modules.
\subsection{Separable Banach bundles}\label{ss:def_sep_Ban_bundle}
Given a Banach space \(\B\), we denote by \({\rm Gr}(\B)\) the
family of its closed linear subspaces.
\begin{definition}[Separable Banach bundle]\label{def:sep_BBbundle}
Let \((\X,\Sigma)\) be a measurable space and \(\B\) a universal
separable Banach space. Then a map \(\mathbf E\colon\X\to{\rm Gr}(\B)\)
is said to be a \emph{separable Banach \(\B\)-bundle} over
\(\X\) provided \(\mathbf E\colon\X\mto\B\) is a weakly
measurable correspondence.
\end{definition}

Given a separable Banach \(\B\)-bundle \(\mathbf E\) over \(\X\), let us define
\begin{equation}\label{eq:def:TE}
{\rm T}\mathbf E\coloneqq\bigcup_{x\in\X}\{x\}\times\mathbf E(x)\subseteq\X\times\B.
\end{equation}
Observe that \({\rm T}\mathbf E\in\Sigma\otimes\mathscr B(\B)\) by item vi)
of \S\ref{ss:correspondences}.
\bigskip

Any bundle is naturally associated with the space of its (measurable) sections:
\begin{definition}[Sections of a separable Banach bundle]
\label{def:sect_sep_Bb}
Let \((\X,\Sigma)\) be a measurable space and \(\B\) a universal separable
Banach space. Let \(\mathbf E\) be a separable Banach \(\B\)-bundle over \(\X\).
Then by \emph{section} of \(\mathbf E\) we mean a measurable selector
of \(\mathbf E\), namely, a measurable map
\(\bar s\colon\X\to\B\) such that \(\bar s(x)\in\mathbf E(x)\)
for every \(x\in\X\).
The vector space of sections of \(\mathbf E\) is denoted by \(\bar\Gamma(\mathbf E)\). Given any \(\bar s\in\bar\Gamma(\mathbf E)\), we
define the function \(|\bar s|\colon\X\to[0,+\infty)\) as
\begin{equation}\label{eq:ptwse_norm_section}
|\bar s|(x)\coloneqq\big\|\bar s(x)\big\|_\B\quad\text{ for every }x\in\X.
\end{equation}
Moreover, we define the linear subspace
\(\bar\Gamma_b(\mathbf E)\subseteq\bar\Gamma(\mathbf E)\) as
\(\bar\Gamma_b(\mathbf E)\coloneqq\big\{\bar s\in\bar\Gamma(\mathbf E)
\,:\,|\bar s|\in\mathcal L^\infty(\Sigma)\big\}\).
\end{definition}
\begin{remark}\label{rmk:bar_Gamma_mod}{\rm
It is straightforward to check that \(\bar\Gamma_b(\mathbf E)\)
is a normed \(\mathcal L^\infty(\Sigma)\)-module when endowed with
the natural pointwise operations and the pointwise norm
\(|\cdot|\colon\bar\Gamma_b(\mathbf E)\to\mathcal L^\infty(\Sigma)\)
that has been defined in \eqref{eq:ptwse_norm_section}.
\fr}\end{remark}
\begin{proposition}\label{prop:equiv_Banach_bundle}
Let \((\X,\Sigma)\) be a measurable space and \(\B\) a universal separable Banach space.
\begin{itemize}
\item[\(\rm i)\)] Let \((\bar s_n)_n\) be any sequence of measurable maps
\(\bar s_n\colon\X\to\B\). Define \(\mathbf E\colon\X\mto\B\) as
\[
\mathbf E(x)\coloneqq{\rm cl}_\B\Big({\rm span}\big\{\bar s_n(x)\;\big|\;n\in\N\big\}\Big)
\quad\text{ for every }x\in\X.
\]
Then \(\mathbf E\) is a separable Banach \(\B\)-bundle over \(\X\).
\item[\(\rm ii)\)] Let \(\mathbf E\) be a separable Banach \(\B\)-bundle
over \(\X\). Then there exists a countable \(\Q\)-linear subspace \(\mathcal C\)
of \(\bar\Gamma(\mathbf E)\) such that
\(\mathbf E(x)={\rm cl}_\B\big\{\bar s(x)\,:\,\bar s\in\mathcal C\big\}\)
for every \(x\in\X\).
\end{itemize}
\end{proposition}
\begin{proof}\ \\
{\color{blue}i)} Given any open set \(U\subseteq\B\), it holds that
\[
\big\{x\in\X\;\big|\;\mathbf E(x)\cap U\neq\emptyset\big\}
=\bigcup_{n\in\N}\,\bigcup_{q_1,\ldots,q_n\in\Q}
\bigg\{x\in\X\;\bigg|\;\sum_{i=1}^n q_i\,\bar s_i(x)\in U\bigg\}\in\Sigma.
\]
By arbitrariness of \(U\), we conclude that \(\mathbf E\colon\X\to{\rm Gr}(\B)\)
is a separable Banach \(\B\)-bundle.\\
{\color{blue}ii)} Let \((v_n)_n\subseteq\B\) be a fixed dense sequence. Given any
\(n,k\in\N\), let us define the correspondence \(\varphi_{nk}\colon\X\mto\B\) as
\[
\varphi_{nk}(x)\coloneqq\left\{\begin{array}{ll}
{\rm cl}_\B\big(\mathbf E(x)\cap B_{1/k}(v_n)\big)\\
\{0_\B\}
\end{array}\quad\begin{array}{ll}
\text{ if }\mathbf E(x)\cap B_{1/k}(v_n)\neq\emptyset,\\
\text{ otherwise.}
\end{array}\right.
\]
Notice that \(\varphi_{nk}\) is weakly measurable as a consequence of
item viii) of \S\ref{ss:correspondences}. Then by applying Kuratowski--Ryll-Nardzewski
theorem (item vii) of \S\ref{ss:correspondences}) we obtain a measurable
selector \(\bar s_{nk}\) of the correspondence \(\varphi_{nk}\).
Given any \(x\in\X\), \(v\in\mathbf E(x)\), and \(\eps>0\), we can find
\(n,k\in\N\) such that \(1/k<\eps/2\) and \(\|v-v_n\|_\B<1/k\).
Therefore, it holds that \(\big\|v-\bar s_{nk}(x)\big\|_\B<\eps\).
This shows that \(\big\{\bar s_{nk}(x)\,:\,n,k\in\N\big\}\) is a dense
subset of \(\mathbf E(x)\) for every \(x\in\X\). The claim follows by
taking as \(\mathcal C\) the \(\Q\)-linear subspace of \(\bar\Gamma(\mathbf E)\)
generated by \(\{\bar s_{nk}\,:\,n,k\in\N\}\).
\end{proof}

In the sequel, we will need the following working definition
of a measurable collection of Banach spaces. Roughly speaking,
it is an intermediate construction that will be used to cook
up the separable Banach bundle underlying a given separable
normed \(L^0(\mm)\)-module. 
\begin{definition}[Measurable collection of separable Banach spaces]
\label{def:meas_coll_Banach}
Let \((\X,\Sigma)\) be a given measurable space. Then a family
\(\big\{E(x)\big\}_{x\in\X}\) of separable Banach spaces is said to
be a \emph{measurable collection of separable Banach spaces} provided
there exist elements \(v_n(x)\in E(x)\) and \(\omega_n(x)\in E(x)'\),
with \(n\in\N\) and \(x\in\X\), such that the following properties hold:
\begin{itemize}
\item[\(\rm a)\)] \(\big(v_n(x)\big)_n\) is a dense subset of \(E(x)\)
for every \(x\in\X\),
\item[\(\rm b)\)] \(\big\|\omega_n(x)\big\|_{E(x)'}=1\) and
\(\omega_n(x)\big[v_n(x)\big]=\big\|v_n(x)\big\|_{E(x)}\)
whenever \(n\in\N\) and \(x\in\X\) are such that \(v_n(x)\neq 0_{E(x)}\),
\item[\(\rm c)\)] \(\X\ni x\mapsto\omega_n(x)\big[v_k(x)\big]\in\R\) is a
measurable function for every \(n,k\in\N\).
\end{itemize}
In particular, the function \(\X\ni x\mapsto\big\|v_n(x)\big\|_{E(x)}\in\R\)
is measurable for every \(n\in\N\).
\end{definition}

In the next two technical results we explain how to get a
separable Banach bundle out of a measurable collection of
separable Banach spaces. Here, the explicit construction in
the proof of Banach--Mazur Theorem \ref{thm:Banach-Mazur} plays a role.
\begin{theorem}[Measurable family of embeddings]
\label{thm:meas_family_embed}
Let \((\X,\Sigma)\) be a measurable space. Let \(\B\) be a universal separable
Banach space and \(\big\{E(x)\big\}_{x\in\X}\) a measurable collection
of separable Banach spaces. Choose elements
\(\big(v_n(x)\big)_{n\in\N}\subseteq E(x)\) for \(x\in\X\) as in Definition
\ref{def:meas_coll_Banach}. Then there exists a family
\(\{{\rm I}_x\}_{x\in\X}\) of linear isometric embeddings
\({\rm I}_x\colon E(x)\to\B\) such that
\begin{equation}\label{eq:meas_sect}
\X\ni x\longmapsto{\rm I}_x\big[v_n(x)\big]\in\B
\quad\text{ is a measurable map for every }n\in\N.
\end{equation}
We say that \(\{{\rm I}_x\}_{x\in\X}\) is a
\emph{measurable family of (linear isometric) embeddings}.
\end{theorem}
\begin{proof}
First of all, let us define the objects we will need throughout the proof:
\begin{itemize}
\item[\(\rm i)\)] Given any \(k\in\N\), we denote by
\(\pi_k\colon I^\infty\to[-1,1]\) the projection on the \(k^{\rm th}\)
component of the Hilbert cube \(I^\infty\), namely, we define
\(\pi_k(\alpha)\coloneqq\alpha_k\) for every
\(\alpha=(\alpha_{k'})_{k'}\in I^\infty\). Each \(\pi_k\) is
continuous, as \(I^\infty\) is endowed with the product topology
(item iii) of \S\ref{ss:embeddings}).
\item[\(\rm ii)\)] Given any \(x\in\X\), we define the map
\(\iota_x\colon B_{E(x)'}\to I^\infty\) as
\begin{equation}\label{eq:def_iota_x}
\iota_x(\omega)\coloneqq\bigg(\frac{\omega\big[v_k(x)\big]}
{\big\|v_k(x)\big\|_{E(x)}\vee 1}\bigg)_k\in I^\infty
\quad\text{ for every }\omega\in B_{E(x)'}.
\end{equation}
Then it holds that \(\iota_x\) is a homeomorphism with its image
(when the domain \(B_{E(x)'}\) is endowed with the restricted weak\(^*\) topology).
Recall item iv) of \S\ref{ss:embeddings}.
\item[\(\rm iii)\)] Let us define the correspondence \(K'\colon\X\mto I^\infty\)
as \(K'(x)\coloneqq\iota_x\big(B_{E(x)'}\big)\) for every \(x\in\X\).
Observe that \(K'\) has compact values by virtue of Banach--Alaoglu theorem.
\item[\(\rm iv)\)] Fix a continuous surjective map \(\psi\colon\Delta\to I^\infty\)
(recall item iii) of \S\ref{ss:embeddings}).
\item[\(\rm v)\)] Denote by \(K\colon\X\mto\Delta\) the preimage
correspondence \(\psi^{-1}(K')\), defined as in Lemma \ref{lem:preimg_corr}.
Since \(\psi\) is continuous and \(\Delta\) is compact, it holds that \(K\)
has compact values.
\item[\(\rm vi)\)] Given any \(x\in\X\), we define the retraction
\(r_x\colon\Delta\to K(x)\) as in item ii) of \S\ref{ss:embeddings}. Namely,
for any point \(a\in\Delta\) we have that \(r_x(a)\) is the unique
element of \(K(x)\) satisfying the identity
\(\sfd_\Delta\big(a,r_x(a)\big)=\sfd_\Delta\big(a,K(x)\big)\).
\item[\(\rm vii)\)] Given any \(x\in\X\), we define the operator
\({\rm I}'_x\colon E(x)\to C(\Delta)\) as
\begin{equation}\label{eq:def_I'_x}
{\rm I}'_x[v](a)\coloneqq(\iota_x^{-1}\circ\psi\circ r_x)(a)[v]
\quad\text{ for every }v\in E(x)\text{ and }a\in\Delta.
\end{equation}
Then each map \({\rm I}'_x\) is a linear isometric embedding
(recall the proof of Theorem \ref{thm:Banach-Mazur}).
\item[\(\rm viii)\)] Fix a linear isometric map \({\rm I}\colon C(\Delta)\to\B\).
Given any \(x\in\X\), we define the linear isometric
embedding \({\rm I}_x\colon E(x)\to\B\) as
\({\rm I}_x[v]\coloneqq({\rm I}\circ{\rm I}'_x)[v]\)
for every \(v\in E(x)\).
\end{itemize}
It remains to prove that \(\X\ni x\mapsto{\rm I}_x\big[v_k(x)\big]\in C([0,1])\)
is a measurable map for any \(k\in\N\).
Fix \(\big(\omega_n(x)\big)_{n\in\N}\subseteq B_{E(x)'}\) for \(x\in\X\)
as in Definition \ref{def:meas_coll_Banach}. We set the family of
indexes \(Q\) as
\[
Q\coloneqq\bigg\{q=(q_n)_n\in\bigoplus_\N \Q\;\bigg|\;
q_n\geq 0\text{ for every }n\in\N\text{ and }\sum_{n\in\N}q_n=1\bigg\},
\]
where \(\bigoplus_\N \Q\) stands for the set of all sequences
\(q=(q_n)_n\in\Q^\N\) such that \(q_n=0\) for all but finitely
many \(n\in\N\). Let us define
\[
\omega^q(x)\coloneqq\sum_{n\in\N}q_n\,\omega_n(x)\in B_{E(x)'}
\quad\text{ for every }q\in Q\text{ and }x\in\X.
\]
Observe that
\(\big\{\omega^q(x)\big\}_{q\in Q}\) is weak\(^*\) dense in \(B_{E(x)'}\) for
every \(x\in\X\), whence it follows that \(\big\{\iota_x\big(\omega^q(x)\big)\big\}_{q\in Q}\)
is \(\sfd_{I^\infty}\)-dense in \(K'(x)\). Moreover, given any
\(\alpha\in I^\infty\) and \(\lambda>0\), we have that
\[\begin{split}
\Big\{x\in\X\;\Big|\;\sfd_{I^\infty}\big(\alpha,K'(x)\big)<\lambda\Big\}
&\overset{\phantom{\eqref{eq:def_iota_x}}}=
\bigcup_{q\in Q}\Big\{x\in\X\;\Big|\;\sfd_{I^\infty}
\big(\alpha,\iota_x\big(\omega^q(x)\big)\big)<\lambda)\Big\}\\
&\overset{\eqref{eq:def_iota_x}}=
\bigcup_{q\in Q}\bigg\{x\in\X\;\bigg|\;\sum_{k=1}^\infty\frac{1}{2^k}
\bigg|\alpha_k-\sum_{n\in\N}\frac{q_n\,\omega_n(x)\big[v_k(x)\big]}
{\big\|v_k(x)\big\|_{E(x)}\vee 1}\bigg|
<\lambda\bigg\}
\in\Sigma,\\
\end{split}\]
as a consequence of the measurability of each function
\(\X\ni x\mapsto\omega_n(x)\big[v_k(x)\big]\). Therefore, the function
\(\X\ni x\mapsto\sfd_{I^\infty}\big(\alpha,K'(x)\big)\) is measurable
for every \(\alpha\in I^\infty\), thus accordingly \(K'\) is a weakly
measurable correspondence by item iv) of \S\ref{ss:correspondences}.
Thanks to item i) of \S\ref{ss:correspondences}, we deduce that
\(K'\) is a measurable correspondence, whence \(K\) is a measurable
correspondence as well by Lemma \ref{lem:preimg_corr}. For any \(a\in\Delta\),
let us consider the correspondence \(Z_a\colon\X\mto\Delta\) given by
\[
Z_a(x)\coloneqq\Big\{b\in\Delta\;\Big|\;
\sfd_\Delta(a,b)=\sfd_\Delta\big(a,K(x)\big)\Big\}
\quad\text{ for every }x\in\X.
\]
It holds that
\(\X\times\Delta\ni(x,b)\mapsto\sfd_\Delta(a,b)-\sfd_\Delta\big(a,K(x)\big)\in\R\)
is a Carath\'{e}odory function, since \(\Delta\ni b\mapsto\sfd_\Delta(a,b)\)
is continuous and \(\X\ni x\mapsto\sfd_\Delta\big(a,K(x)\big)\) is measurable
(the latter follows from the measurability of \(K\), by taking items i) and
iv) of \S\ref{ss:correspondences} into account). Therefore, we have that
\(Z_a\colon\X\mto\Delta\) is a measurable correspondence by item v) of
\S\ref{ss:correspondences}, whence the intersection correspondence
\(Z_a\cap K\colon\X\mto\Delta\) is measurable as well by item iii) of
\S\ref{ss:correspondences}. Since it holds that
\(Z_a(x)\cap K(x)=\big\{r_x(a)\big\}\) for every \(x\in\X\),
we deduce from item ii) of \S\ref{ss:correspondences} that
\begin{equation}\label{eq:r_x_meas}
\X\ni x\longmapsto r_x(a)\in\Delta\quad
\text{ is a measurable map for every }a\in\Delta.
\end{equation}
Let us fix \(k\in\N\), a dense subset \((a^i)_{i\in\N}\) of \(\Delta\),
and an element \(g\in C(\Delta)\). Observe that
\[
{\rm I}'_x\big[v_k(x)\big](a^i)\overset{\eqref{eq:def_I'_x}}=
(\iota_x^{-1}\circ\psi\circ r_x)(a^i)\big[v_k(x)\big]
\overset{\eqref{eq:def_iota_x}}=(\pi_k\circ\psi)\big(r_x(a^i)\big)
\quad\text{ for all }x\in\X\text{ and }i\in\N.
\]
Being \(\pi_k\circ\psi\colon\Delta\to[-1,1]\) continuous,
we deduce from \eqref{eq:r_x_meas} that
\(\X\ni x\mapsto{\rm I}'_x\big[v_k(x)\big](a^i)\in\R\) is measurable
for every \(i\in\N\). Since we have that
\[\begin{split}
\Big\|g-{\rm I}'_x\big[v_k(x)\big]\Big\|_{C(\Delta)}
&\overset{\eqref{eq:def_I'_x}}=
\sup_{i\in\N}\Big|g(a^i)-{\rm I}'_x\big[v_k(x)\big](a^i)\Big|
\quad\text{ for every }x\in\X,
\end{split}\]
we deduce that
\(\X\ni x\mapsto\big\|g-{\rm I}'_x\big[v_k(x)\big]\big\|_{C(\Delta)}\in\R\)
is measurable for every element \(g\in C(\Delta)\). Therefore, it holds that
\(\X\ni x\mapsto{\rm I}'_x\big[v_k(x)\big]\in C(\Delta)\)
is a measurable map for all \(k\in\N\). Recalling that
\({\rm I}_x={\rm I}\circ{\rm I}'_x\) for every \(x\in\X\)
and that \(\rm I\) is continuous, we can finally conclude that
\(\X\ni x\mapsto{\rm I}_x\big[v_k(x)\big]\in\B\)
is a measurable map for all \(k\in\N\), as required.
\end{proof}
\begin{corollary}\label{cor:img_Ix_bundle}
Let \((\X,\Sigma)\) be a measurable space and \(\B\)
a universal separable Banach space. Let \(\big\{E(x)\big\}_{x\in\X}\)
be a measurable collection of separable Banach spaces. Consider the
associated measurable family \(\{{\rm I}_x\}_{x\in\X}\) of linear
isometric embeddings \({\rm I}_x\colon E(x)\to\B\) as in Theorem
\ref{thm:meas_family_embed}. Then the map
\(\X\ni x\mapsto\mathbf E(x)\coloneqq
{\rm I}_x\big(E(x)\big)\in{\rm Gr}(\B)\)
is a separable Banach \(\B\)-bundle over \(\X\).
\end{corollary}
\begin{proof}
Choose elements \(\big(v_n(x)\big)_{n\in\N}\subseteq E(x)\)
for \(x\in\X\) as in Definition \ref{def:meas_coll_Banach}. Let us define
\[
\bar s_n(x)\coloneqq{\rm I}_x\big[v_n(x)\big]
\quad\text{ for every }n\in\N\text{ and }x\in\X.
\]
Theorem \ref{thm:meas_family_embed} guarantees that each map \(\bar s_n\colon\X\to\B\)
is measurable, thus accordingly the correspondence \(\mathbf E\colon\X\mto\B\),
which is given by
\[
\mathbf E(x)={\rm I}_x\big(E(x)\big)=
{\rm cl}_\B\Big({\rm span}\big\{\bar s_n(x)\;\big|\;n\in\N\big\}\Big)
\quad\text{ for every }x\in\X,
\]
is a separable Banach \(\B\)-bundle over \(\X\) by item i) of Proposition \ref{prop:equiv_Banach_bundle}.
\end{proof}
\subsection{The section functor}\label{ss:section_functor}
Let \((\X,\Sigma,\mm)\) be a \(\sigma\)-finite measure space
and \(\B\) a universal separable Banach space. Let \(\mathbf E\) be
a separable Banach \(\B\)-bundle over \(\X\). Then we define
\begin{equation}\label{eq:def_sect_E}
\Gamma_b(\mathbf E)\coloneqq\Pi_\mm\big(\bar\Gamma_b(\mathbf E)\big),
\quad\Gamma(\mathbf E)\coloneqq\sfC\big(\Gamma_b(\mathbf E)\big),
\end{equation}
where the operations \(\Pi_\mm\) and \(\sfC\) are defined as
in \eqref{eq:def_proj_Pi_m} and Definition \ref{def:compl/restr}, respectively.
\begin{remark}\label{rmk:other_def_Gamma(E)}{\rm
Observe that \(\Gamma(\mathbf E)\) can be identified with the quotient
space \(\bar\Gamma(\mathbf E)/\sim\), where the equivalence relation
\(\sim\) on \(\bar\Gamma(\mathbf E)\) is defined in the following way:
given any \(\bar s,\bar t\in\bar\Gamma(\mathbf E)\), we declare that
\(\bar s\sim\bar t\) provided
\(\mm\big(\big\{x\in\X\,:\,\bar s(x)\neq\bar t(x)\big\}\big)=0\).
\fr}\end{remark}
\begin{lemma}\label{lem:Gamma(E)_separable}
Let \((\X,\Sigma,\mm)\) be a \(\sigma\)-finite measure space
and \(\B\) a universal separable Banach space. Let \(\mathbf E\) be
a separable Banach \(\B\)-bundle over \(\X\). Then the normed \(L^0(\mm)\)-module
\(\Gamma(\mathbf E)\) is countably-generated. In particular, if \((\X,\Sigma,\mm)\)
is separable, then \(\Gamma(\mathbf E)\) is separable.
\end{lemma}
\begin{proof}
Thanks to item ii) of Proposition \ref{prop:equiv_Banach_bundle},
we can find a sequence \((\bar s_n)_n\subseteq\bar\Gamma(\mathbf E)\)
such that the identity
\(\mathbf E(x)={\rm cl}_\B\big\{\bar s_n(x)\,:\,n\in\N\big\}\)
is satisfied for every \(x\in\X\). Given any \(n,k\in\N\), we define
\(B_{nk}\coloneqq\big\{x\in\X\,:\,|\bar s_n|(x)\leq k\big\}\)
and \(s_{nk}\coloneqq\pi_\mm(\nchi_{B_{nk}}\cdot\bar s_n)\),
where \(\pi_\mm\colon\bar\Gamma_b(\mathbf E)\to\Gamma_b(\mathbf E)\)
stands for the canonical projection map. We claim that
\((s_{nk})_{n,k}\) generates \(\Gamma(\mathbf E)\).
To prove it, fix \(t\in\Gamma(\mathbf E)\) and \(\eps>0\).
Then there exist \(k\in\N\) and \(A\in\Sigma\) with
\(|t|<k-\eps/2\) \(\mm\)-a.e.\ on \(A\) and
\(\sfd_{\Gamma(\mathbf E)}(t_0,t)<\eps/2\), where we
set \(t_0\coloneqq[\nchi_A]_\mm\cdot t\in\Gamma_b(\mathbf E)\).
Choose any element \(\bar t_0\in\bar\Gamma_b(\mathbf E)\) such that
\(t_0=\pi_\mm(\bar t_0)\). Pick a partition \((A_n)_n\subseteq\Sigma\)
of \(A\) such that \(|\bar s_n-\bar t_0|\leq\eps/2\) on \(A_n\)
for every \(n\in\N\). Therefore, we have that
\(\big|\sum_{n\in\N}[\nchi_{A_n}]_\mm\cdot s_{nk}-t_0\big|\leq\eps/2\)
holds \(\mm\)-a.e.\ on \(\X\). This implies that
\(\sfd_{\Gamma(\mathbf E)}\big(\sum_{n\in\N}[\nchi_{A_n}]_\mm
\cdot s_{nk},t\big)<\eps\), thus proving the first claim.
The second one is then an immediate consequence of Proposition \ref{prop:countable_gen}.
\end{proof}
\begin{definition}[Morphism of separable Banach bundles]
Let \((\X,\Sigma,\mm)\) be a \(\sigma\)-finite measure space
and \(\B\) a universal separable Banach space. Let \(\mathbf E,\mathbf F\)
be two separable Banach \(\B\)-bundles over \(\X\). Then a \emph{pre-morphism}
\(\bar\varphi\) from \(\mathbf E\) to \(\mathbf F\) is a measurable map
\(\bar\varphi\colon{\rm T}\mathbf E\to{\rm T}\mathbf F\) such that
\(\bar\varphi\big(\{x\}\times\mathbf E(x)\big)\subseteq\{x\}\times\mathbf F(x)
\cong\mathbf F(x)\) for every \(x\in\X\) and
\[
\bar\varphi(x,\cdot)\colon\mathbf E(x)\to\mathbf F(x)
\quad\text{ is a linear contraction for every }x\in\X.
\]
We declare two pre-morphisms \(\bar\varphi_1,\bar\varphi_2\) from
\(\mathbf E\) to \(\mathbf F\) to be \emph{equivalent} if there
exists a set \(N\in\Sigma\) with \(\mm(N)=0\) such that
\(\bar\varphi_1(x,\cdot)=\bar\varphi_2(x,\cdot)\) for every \(x\in\X\setminus N\).
This defines an equivalence relation, whose equivalence classes are
called \emph{morphisms} and usually denoted by \(\varphi\colon\mathbf E\to\mathbf F\).
\end{definition}

We denote by \(\mathbf{SBB}_\B(\X,\Sigma,\mm)\) the category having
the separable Banach \(\B\)-bundles over the space \(\X\) as objects
and the morphisms of separable Banach \(\B\)-bundles as arrows.
\bigskip

Let us consider two separable Banach \(\B\)-bundles
\(\mathbf E,\mathbf F\) over \(\X\) and a morphism
\(\varphi\colon\mathbf E\to\mathbf F\). Fix any pre-morphism
\(\bar\varphi\colon{\rm T}\mathbf E\to{\rm T}\mathbf F\) that is a
representative of \(\varphi\). Then we define the morphism of normed
\(L^0(\mm)\)-modules
\(\Gamma(\varphi)\colon\Gamma(\mathbf E)\to\Gamma(\mathbf F)\) as
follows: given any \(s\in\Gamma(\mathbf E)\), we define
\(\Gamma(\varphi)(s)\) as the equivalence class (under the
relation \(\sim\) introduced in Remark 
\ref{rmk:other_def_Gamma(E)}) of
\[
\X\ni x\longmapsto\bar\varphi\big(x,\bar s(x)\big)\in\mathbf F(x),
\]
where \(\bar s\in\bar\Gamma(\mathbf E)\) is any representative
of \(s\). It can be readily checked that this way we obtain
a covariant functor \(\Gamma\colon
\mathbf{SBB}_\B(\X,\Sigma,\mm)\to\mathbf{NMod}_{\rm cg}(\X,\Sigma,\mm)\),
which we call the \emph{section functor}. (For brevity, in our
notation the dependence of \(\Gamma\) on the space \(\B\) is omitted.)
\begin{lemma}[\(\Gamma\) is full]\label{lem:Gamma_full}
Let \((\X,\Sigma,\mm)\) be a \(\sigma\)-finite
measure space and \(\B\) a universal separable Banach space.
Let \(\mathbf E,\mathbf F\) be two separable Banach \(\B\)-bundles
and \(\Phi\colon\Gamma(\mathbf E)\to\Gamma(\mathbf F)\) a morphism
of normed \(L^0(\mm)\)-modules. Then there exists a morphism
\(\varphi\colon\mathbf E\to\mathbf F\) of separable Banach
\(\B\)-bundles such that \(\Gamma(\varphi)=\Phi\).
\end{lemma}
\begin{proof}
Thanks to item ii) of Proposition \ref{prop:equiv_Banach_bundle},
there is a countable \(\Q\)-linear subspace \(\mathcal C\) of
\(\bar\Gamma(\mathbf E)\) such that the \(\Q\)-linear space
\(\mathcal C(x)\coloneqq\big\{\bar s(x)\,:\,\bar s\in\mathcal C\big\}\)
is dense in \(\mathbf E(x)\) for every \(x\in\X\). Given any
\(\bar s\in\mathcal C\), choose a representative
\(\bar\Phi(\bar s)\in\bar\Gamma(\mathbf F)\) of
\(\Phi\big([\bar s]_\sim\big)\in\Gamma(\mathbf F)\).
Then there exists \(N\in\Sigma\) with \(\mm(N)=0\) such that
the following properties hold:
\begin{equation}\label{eq:Gamma_full_aux}\begin{split}
\bar\Phi(\bar s+\bar t\,)(x)&=\bar\Phi(\bar s)(x)+\bar\Phi(\bar t\,)(x),\\
\bar\Phi(q\,\bar s)(x)&=q\,\bar\Phi(\bar s)(x),\qquad\qquad\quad
\text{ for every }x\in\X\setminus N,\text{ }\bar s,\bar t\in\mathcal C,
\text{ and }q\in\Q.\\
\big\|\bar\Phi(\bar s)(x)\big\|_{\mathbf F(x)}&\leq
\big\|\bar s(x)\big\|_{\mathbf E(x)},
\end{split}\end{equation}
Given any \(x\in\X\), we define the map
\(\bar\varphi_x\colon\mathcal C(x)\to\mathbf F(x)\) as
\[
\bar\varphi_x\big(\bar s(x)\big)\coloneqq\left\{\begin{array}{ll}
\bar\Phi(\bar s)(x)\\
0_{\mathbf F(x)}
\end{array}\quad\begin{array}{ll}
\text{ if }x\in\X\setminus N,\\
\text{ if }x\in N.
\end{array}\right.
\]
The properties in \eqref{eq:Gamma_full_aux} grant that
each map \(\bar\varphi_x\) is a \(\Q\)-linear contraction,
thus it can be uniquely extended to an \(\R\)-linear contraction
\(\bar\varphi_x\colon\mathbf E(x)\to\mathbf F(x)\). Then we
define \(\bar\varphi\colon{\rm T}\mathbf E\to{\rm T}\mathbf F\)
as \(\bar\varphi(x,v)\coloneqq\big(x,\bar\varphi_x(v)\big)\)
for every \(x\in\X\) and \(v\in\mathbf E(x)\). In order to prove that
\(\bar\varphi\) is a pre-morphism, it is sufficient to check its
measurability. To this aim, we just have to show that \(\bar\varphi^{-1}
\big(A\times\bar B_r(w)\big)\in\Sigma\otimes\mathscr B(\B)\) for any
\(A\in\Sigma\), \(w\in\B\), and \(r>0\). This follows from the identity
\[
\bar\varphi^{-1}\big(A\times\bar B_r(w)\big)=
S\cup\Big\{(x,v)\in\big((A\setminus N)\times\B\big)\cap{\rm T}\mathbf E
\;\Big|\;\bar\varphi_x(v)\in\bar B_r(w)\Big\}=
S\cup\bigcup_{n\in\N}\bigcap_{\bar s\in\mathcal C}A_{n,\bar s},
\]
where we set \(S\coloneqq\big((A\cap N)\times\B\big)\cap{\rm T}\mathbf E\in\Sigma\)
if \(\|w\|_\B\leq r\), while \(S\coloneqq\emptyset\) if \(\|w\|_\B>r\), and
\[
A_{n,\bar s}\coloneqq\bigg\{(x,v)\in\big((A\setminus N)\times\B\big)
\cap{\rm T}\mathbf E\;\bigg|\;\big\|v-\bar s(x)\big\|_\B<\frac{1}{k},
\;\big\|w-\bar\Phi(\bar s)(x)\big\|_\B<r+\frac{1}{k}\bigg\}\in\Sigma.
\]
Therefore, \(\bar\varphi\) is a pre-morphism from \(\mathbf E\)
to \(\mathbf F\). We denote by \(\varphi\colon\mathbf E\to\mathbf F\)
its equivalence class. Observe that
\(\Gamma(\varphi)\big([\bar s]_\sim\big)=\Phi\big([\bar s]_\sim\big)\)
for every \(\bar s\in\mathcal C\) by construction. Finally,
by arguing exactly as in the proof of Lemma \ref{lem:Gamma(E)_separable},
we deduce that \(\big\{[\bar s]_\sim\,:\,\bar s\in\mathcal C\big\}\)
generates \(\Gamma(\mathbf E)\), whence we can conclude that
\(\Gamma(\varphi)=\Phi\). Consequently, the statement is achieved.
\end{proof}
\begin{lemma}[\(\Gamma\) is faithful]\label{lem:Gamma_faithful}
Let \((\X,\Sigma,\mm)\) be a \(\sigma\)-finite
measure space and \(\B\) a universal separable Banach space.
Let \(\mathbf E,\mathbf F\) be separable Banach \(\B\)-bundles.
Let \(\varphi,\psi\colon\mathbf E\to\mathbf F\) be two
morphisms of separable Banach \(\B\)-bundles such that
\(\varphi\neq\psi\). Then \(\Gamma(\varphi)\neq\Gamma(\psi)\).
\end{lemma}
\begin{proof}
Choose representatives \(\bar\varphi,\bar\psi\) of
\(\varphi,\psi\), respectively. Pick a set \(P'\in\Sigma\)
such that \(\mm(P')>0\) and \(\bar\varphi(x,\cdot)\neq\bar\psi(x,\cdot)\)
for every \(x\in P'\). By item ii) of Proposition
\ref{prop:equiv_Banach_bundle}, there exists
\((\bar s_n)_n\subseteq\bar\Gamma(\mathbf E)\) such that
\(\big(\bar s_n(x)\big)_n\) is dense in \(\mathbf E(x)\)
for all \(x\in\X\). Therefore, there exist \(n\in\N\) and
\(P\in\Sigma\) such that \(P\subseteq P'\), \(\mm(P)>0\), and
\(\bar\varphi\big(x,\bar s_n(x)\big)\neq\bar\psi\big(x,\bar s_n(x)\big)\)
for every \(x\in P\). This ensures that
\(\Gamma(\varphi)\big([\bar s_n]_\sim\big)\neq
\Gamma(\psi)\big([\bar s_n]_\sim\big)\) and thus
accordingly \(\Gamma(\varphi)\neq\Gamma(\psi)\), as required.
\end{proof}
\subsection{Representation theorem}\label{ss:repr_thm}
We are finally in a position -- by combining the whole machinery
developed so far -- to prove that every separable
normed \(L^0(\mm)\)-module is the space of sections of a
separable Banach bundle.
\begin{theorem}[Representation theorem]\label{thm:representation}
Let \((\X,\Sigma,\mm)\) be a complete, \(\sigma\)-finite measure space.
Let \(\mathscr M\) be a countably-generated normed \(L^0(\mm)\)-module.
Let \(\B\) be a universal separable Banach space. Then there exists a
separable Banach \(\B\)-bundle \(\mathbf E\) over \(\X\) such that
\(\Gamma(\mathbf E)\cong\mathscr M\).
\end{theorem}
\begin{proof}\ \\
{\color{blue}\textsc{Step 1.}} First, fix a countable \(\Q\)-linear
subspace \((v_n)_n\) of \(\sfR(\mathscr M)\) that generates \(\mathscr M\).
Choose a sequence \((\omega_n)_n\subseteq\sfR(\mathscr M^*)\) such
that the identities \(|\omega_n|=1\) and \(\omega_n(v_n)=|v_n|\) hold
\(\mm\)-a.e.\ for every \(n\in\N\). Let \(\ell\) be any lifting
of \(\mm\), whose existence is granted by Theorem \ref{thm:von_Neumann-Maharam}.
Consider the operator \(\mathcal L\colon L^\infty(\mm)\to\mathcal L^\infty(\Sigma)\)
associated with \(\ell\) as in Theorem \ref{thm:lifting_Linfty}.
Denote by \((\bar{\mathscr M},\mathscr L)\) and \((\bar{\mathscr N},\mathscr L^*)\)
the \(\mathcal L\)-liftings of \(\sfR(\mathscr M)\) and \(\sfR(\mathscr M^*)\),
respectively; recall Theorem \ref{thm:lifting_mod}. Therefore,
\begin{equation}\label{eq:repr_thm_1}\begin{split}
\big|\mathscr L^*(\omega_n)\big|(x)&=1,\\
\big\langle\mathscr L^*(\omega_n),\mathscr L(v_n)\big\rangle(x)&=
\big|\mathscr L(v_n)\big|(x),\end{split}
\quad\text{ for every }n\in\N\text{ and }x\in\X.
\end{equation}
(Recall the discussion about the duality pairing \(\la\cdot,\cdot\ra\)
between \(\bar{\mathscr M}\) and \(\bar{\mathscr N}\) in \S\ref{ss:fibers}.)

Given any point \(x\in\X\), we define the separable Banach
subspace \(E(x)\) of \(\bar{\mathscr M}_x\) as
\begin{equation}\label{eq:def_E(x)}
E(x)\coloneqq{\rm cl}_{\bar{\mathscr M}_x}\big\{\mathscr L(v_n)_x\;\big|\;n\in\N\big\}.
\end{equation}
Consider the isometric embedding
\({\rm R}_x\colon\bar{\mathscr N}_x\to(\bar{\mathscr M}_x)'\),
which has been introduced in \eqref{eq:def_R_x} and studied in Proposition
\ref{prop:Rx_isometry}. Let us define
\[\begin{split}
\bar v_n(x)&\coloneqq\mathscr L(v_n)_x\in E(x),\\
\bar\omega_n(x)&\coloneqq{\rm R}_x\big(\mathscr L^*(\omega_n)_x\big)|_{E(x)}
\in E(x)',\end{split}
\quad\text{ for every }n\in\N\text{ and }x\in\X.
\]
It follows from the second line in \eqref{eq:repr_thm_1} that
\begin{equation}\label{eq:repr_thm_1bis}
\bar\omega_n(x)\big[\bar v_n(x)\big]=\big\|\bar v_n(x)\big\|_{E(x)}
\quad\text{ for every }n\in\N\text{ and }x\in\X.
\end{equation}
Moreover, observe that for any \(n\in\N\) and \(x\in\X\) one has that
\[
\big\|\bar\omega_n(x)\big\|_{E(x)'}\leq
\Big\|{\rm R}_x\big(\mathscr L^*(\omega_n)_x\big)\Big\|_{(\bar{\mathscr M}_x)'}
=\big\|\mathscr L^*(\omega_n)_x\big\|_{\bar{\mathscr N}_x}
\overset{\eqref{eq:repr_thm_1}}=1.
\]
Hence, if \(n\in\N\) and \(x\in\X\) satisfy \(\bar v_n(x)\neq 0_{E(x)}\),
then \eqref{eq:repr_thm_1bis} forces \(\big\|\bar\omega_n(x)\big\|_{E(x)'}=1\).
Finally, given any \(n,k\in\N\), we have that the function
\(\X\ni x\mapsto\bar\omega_n(x)\big[\bar v_k(x)\big]=\big\langle\mathscr L^*(\omega_n),
\mathscr L(v_k)\big\rangle(x)\) is measurable. All in all, we have proven that
\(\big\{E(x)\big\}_{x\in\X}\) is a measurable collection of separable Banach spaces
(in the sense of Definition \ref{def:meas_coll_Banach}) when equipped with
\((\bar v_n)_n,(\bar\omega_n)_n\). Therefore, let us consider a measurable
family \(\{{\rm I}_x\}_{x\in\X}\) of linear isometric embeddings
\({\rm I}_x\colon E(x)\to\B\), whose existence is granted by Theorem
\ref{thm:meas_family_embed}. We thus denote by \(\mathbf E\colon\X\to{\rm Gr}(\B)\)
the map \(\X\ni x\mapsto{\rm I}_x\big(E(x)\big)\), which is a separable Banach
\(\B\)-bundle over \(\X\) thanks to Corollary \ref{cor:img_Ix_bundle}.\\
{\color{blue}\textsc{Step 2.}} Let \(v\in\sfR(\mathscr M)\) be fixed. We claim that
\begin{equation}\label{eq:repr_thm_2}
\mathscr L(v)_x\in E(x)\quad\text{ for }\mm\text{-a.e.\ }x\in\X.
\end{equation}
Indeed, we can find a sequence \((u_k)_k\) -- where
\(u_k=\sum_{i=1}^{m_k}f^k_i\cdot u^k_i\) for some
\((f^k_i)_{i=1}^{m_k}\subseteq L^\infty(\mm)\) and
\((u^k_i)_{i=1}^{m_k}\subseteq\{v_n\}_n\) -- such that
\(\lim_k\sfd_{\mathscr M}(u_k,v)=0\). Then (up to taking a not relabelled
subsequence) we have that \(\big|\mathscr L(u_k)-\mathscr L(v)\big|(x)\to 0\)
for \(\mm\)-a.e.\ point \(x\in\X\), or equivalently that
\(\lim_k{\big\|\mathscr L(u_k)_x-\mathscr L(v)_x\big\|}_x=0\) for \(\mm\)-a.e.\ \(x\in\X\). Since
\[
\mathscr L(u_k)_x=\sum_{i=1}^{m_k}\mathcal L(f^k_i)(x)\,
\mathscr L(u^k_i)_x\in E(x)\quad\text{ for every }k\in\N\text{ and }x\in\X,
\]
we obtain \eqref{eq:repr_thm_2}. 
Now let us define the map \(\bar{\mathcal I}(v)\colon\X\to\B\) as
\begin{equation}\label{eq:bar_I(v)}
\bar{\mathcal I}(v)(x)\coloneqq\left\{\begin{array}{ll}
{\rm I}_x\big[\mathscr L(v)_x\big]\\
0_\B
\end{array}\quad\begin{array}{ll}
\text{ if }\mathscr L(v)_x\in E(x),\\
\text{ otherwise.}
\end{array}\right.
\end{equation}
Choose any set \(N\in\Sigma\) such that \(\mm(N)=0\) and
\(\mathscr L(v)_x=\lim_k\mathscr L(u_k)_x\) for every \(x\in\X\setminus N\).
Hence, we have that
\[
\bar{\mathcal I}(v)(x)=\lim_{k\to\infty}{\rm I}_x\big[\mathscr L(u_k)_x\big]
=\lim_{k\to\infty}\sum_{i=1}^{m_k}\mathcal L(f^k_i)(x)\,{\rm I}_x\big[\mathscr L(u^k_i)_x\big]
\quad\text{ for every }x\in\X\setminus N.
\]
By recalling Theorem \ref{thm:meas_family_embed}
and the fact that the measure space \((\X,\Sigma,\mm)\) is complete, we deduce
that \(\bar{\mathcal I}(v)\) is a measurable map from \(\X\) to \(\B\).
In other words, it holds that \(\bar{\mathcal I}(v)\in\bar\Gamma(\mathbf E)\).
Then let us denote by \(\mathcal I\colon\sfR(\mathscr M)\to\Gamma(\mathbf E)\)
the map given by \(\mathcal I(v)\coloneqq\big[\bar{\mathcal I}(v)\big]_\sim\)
for every \(v\in\sfR(\mathscr M)\).\\
{\color{blue}\textsc{Step 3.}} We aim to prove that
\(\mathcal I\) maps \(\sfR(\mathscr M)\) to \(\Gamma_b(\mathbf E)\)
and that \(\mathcal I\colon\sfR(\mathscr M)\to\Gamma_b(\mathbf E)\)
is an isomorphism of normed \(L^\infty(\mm)\)-modules. This is sufficient to
conclude that the spaces \(\mathscr M\) and \(\Gamma(\mathbf E)\) are isomorphic
as normed \(L^0(\mm)\)-modules by item i) of Lemma \ref{lem:C=inverse_R}.
We first check the \(L^\infty(\mm)\)-linearity of \(\mathcal I\):
if \(v,w\in\sfR(\mathscr M)\) and \(f,g\in L^\infty(\mm)\),
then for \(\mm\)-a.e.\ \(x\in\X\) we have
\[\begin{split}
\bar{\mathcal I}(f\cdot v+g\cdot w)(x)
&={\rm I}_x\big[\mathscr L(f\cdot v+g\cdot w)_x\big]
={\rm I}_x\big[\mathcal L(f)(x)\cdot\mathscr L(v)_x
+\mathcal L(g)(x)\cdot\mathscr L(w)_x\big]\\
&=\mathcal L(f)(x)\cdot\bar{\mathcal I}(v)(x)
+\mathcal L(g)(x)\cdot\bar{\mathcal I}(w)(x),
\end{split}\]
thus accordingly \(\mathcal I(f\cdot v+g\cdot w)=f\cdot\mathcal I(v)+g\cdot\mathcal I(w)\).
Moreover, given any \(v\in\sfR(\mathscr M)\) one has
\[
{\big\|\bar{\mathcal I}(v)(x)\big\|}_\B=
{\big\|\mathscr L(v)_x\big\|}_x
=\big|\mathscr L(v)\big|(x)=\mathcal L\big(|v|\big)(x)
\quad\text{ for }\mm\text{-a.e.\ }x\in\X,
\]
whence \(\big|\mathcal I(v)\big|=|v|\) holds in the
\(\mm\)-a.e.\ sense. This ensures that \(\mathcal I\)
maps \(\sfR(\mathscr M)\) to \(\Gamma_b(\mathbf E)\) and that
\(\mathcal I\colon\sfR(\mathscr M)\to\Gamma_b(\mathbf E)\)
is a morphism of normed \(L^\infty(\mm)\)-modules preserving the
pointwise norm. Finally, to prove that the map \(\mathcal I\) is
surjective, it is enough to show that its image is dense in
\(\Gamma_b(\mathbf E)\). Fix \(s\in\Gamma_b(\mathbf E)\) and \(\eps>0\).
Choose any representative \(\bar s\in\bar\Gamma_b(\mathbf E)\) of \(s\).
Given that the sequence \(\big(\bar{\mathcal I}(v_n)(x)\big)_n\) is dense
in \(\mathbf E(x)\) for every \(x\in\X\) by \eqref{eq:def_E(x)} and 
\eqref{eq:bar_I(v)}, we can find a partition \((A_n)_n\subseteq\Sigma\)
of \(\X\) such that
\(\big\|\bar s(x)-\bar{\mathcal I}(v_n)(x)\big\|\leq\eps\)
for every \(n\in\N\) and \(x\in A_n\). This implies that the inequality
\(\big|\bar s-\sum_{n\in\N}\nchi_{A_n}\cdot\bar{\mathcal I}(v_n)\big|\leq\eps\)
holds everywhere on \(\X\).

Now let us define
\(v\coloneqq\sum_{n\in\N}[\nchi_{A_n}]_\mm\cdot v_n\in\sfR(\mathscr M)\).
Clearly \(\big|s-\mathcal I(v)\big|\leq\eps\) holds \(\mm\)-a.e.\ by
construction, so that
\(\big\|s-\mathcal I(v)\big\|_{\Gamma_b(\mathbf E)}\leq\eps\).
This yields surjectivity of \(\mathcal I\), which is consequently
an isomorphism of normed \(L^\infty(\mm)\)-modules. Hence,
the statement is finally achieved.
\end{proof}
\begin{remark}[Proof of the representation theorem without the Axiom of Choice]\label{rmk:repr_thm_no_AC}
{\rm
In the above proof of Theorem \ref{thm:representation}, we made use of
the theory of liftings of normed modules that we developed in
\S\ref{s:lift_norm_mod}. Nevertheless, it is possible to provide 
an alternative proof which does not rely upon the Axiom of Choice
(and, thus, without appealing to von Neumann's theory of lifting).
Some weaker form of the Axiom of Choice is needed anyway, \emph{e.g.},
in the proof of Banach--Mazur Theorem \ref{thm:Banach-Mazur}, where
Banach--Alaoglu theorem is used. We now sketch the argument of the
alternative proof of Theorem \ref{thm:representation},
leaving its verification to the reader.

Fix a countable \(\Q\)-linear subspace \((v_n)_n\) of \(\mathscr M\) that
generates \(\mathscr M\). Choose \((\omega_n)_n\subseteq\mathscr M^*\)
such that \(|\omega_n|=1\) and \(\omega_n(v_n)=|v_n|\) hold \(\mm\)-a.e.\ for
every \(n\in\N\). Given any \(n,k\in\N\), pick a measurable representative
\(\overline{\omega_n(v_k)}\) of \(\omega_n(v_k)\). Then we can
find a \(\mm\)-null set \(N\subseteq\X\) such that
\[\begin{split}
\overline{\omega_n(v_k+v_{k'})}(x)&=\overline{\omega_n(v_k)}(x)
+\overline{\omega_n(v_{k'})}(x),\\
\overline{\omega_n(q v_k)}(x)&=q\,\overline{\omega_n(v_k)}(x),\\
\overline{\omega_n(v_k)}(x)&\leq\overline{\omega_k(v_k)}(x)
\end{split}\]
for every \(n,k,k'\in\N\), \(q\in\Q\), and \(x\in\X\setminus N\).
Observe that for any \(x\in\X\setminus N\) the family
\[
V_x\coloneqq\Big\{\big(\,\overline{\omega_n(v_k)}(x)\big)_{n\in\N}\;
\Big|\;k\in\N\Big\}\subseteq\ell^\infty
\]
is a \(\Q\)-linear subpace of \(\ell^\infty\). Hence, calling
\(E(x)\coloneqq\{0_{\ell^\infty}\}\) for all \(x\in N\) and
\(E(x)\coloneqq{\rm cl}_{\ell^\infty}(V_x)\) for all \(x\in\X\setminus N\),
we have that \(\big\{E(x)\big\}_{x\in\X}\) is a family of separable Banach
subspaces of \(\ell^\infty\). Moreover, for any \(x\in\X\) and \(n,k\in\N\),
we define \(\tilde v_k(x)\in E(x)\) and \(\tilde\omega_n(x)\in E(x)'\) as
follows: trivially, \(\tilde v_k(x)\coloneqq 0_{E(x)}\) and
\(\tilde\omega_n(x)\coloneqq 0_{E(x)'}\) if \(x\in N\); if \(x\notin N\),
then we set
\[
\tilde v_k(x)\coloneqq\big(\,\overline{\omega_{n'}(v_k)}(x)\big)_{n'\in\N}
\in E(x),
\]
while we denote by \(\tilde\omega_n(x)\colon E(x)\to\R\) the unique
linear and continuous operator satisfying
\(\tilde\omega_n(x)\big[\tilde v_{k'}(x)\big]=\overline{\omega_n(v_{k'})}(x)\)
for every \(k'\in\N\). Then it holds that \(\big\{E(x)\big\}_{x\in\X}\) is
a measurable collection of separable Banach spaces -- together with
\(\tilde v_k(x)\) and \(\tilde\omega_n(x)\). Finally, one can also prove
that the associated separable Banach \(\B\)-bundle \(\bf E\) satisfies
\(\Gamma({\bf E})\cong\mathscr M\), as desired.
\fr}\end{remark}

In analogy with \cite{LP18}, we have a Serre--Swan
theorem for separable normed \(L^0(\mm)\)-modules:
\begin{theorem}[Serre--Swan theorem]\label{thm:Serre-Swan}
Let \((\X,\Sigma,\mm)\) be a complete, \(\sigma\)-finite
measure space. Let \(\B\) be a universal separable Banach
space. Then the section functor
\[
\Gamma\colon\mathbf{SBB}_\B(\X,\Sigma,\mm)
\longrightarrow\mathbf{NMod}_{\rm cg}(\X,\Sigma,\mm)
\]
is an equivalence of categories. In particular, if \((\X,\Sigma,\mm)\)
is a separable measure space, then
\[
\Gamma\colon\mathbf{SBB}_\B(\X,\Sigma,\mm)
\longrightarrow\mathbf{NMod}_{\rm s}(\X,\Sigma,\mm)
\]
is an equivalence of categories.
\end{theorem}
\begin{proof}
The first claim follows from Lemma \ref{lem:Gamma_full}, Lemma
\ref{lem:Gamma_faithful}, and Theorem \ref{thm:representation}.
The second claim follows from the first one by taking
Proposition \ref{prop:countable_gen} into account.
\end{proof}
\begin{problem}
Does there exist some notion of measurable Banach bundle that is
sufficient to describe also the duals of separable normed
\(L^0(\mm)\)-modules? In this regard, it is well-known that
all duals of separable Banach spaces can be embedded
linearly and isometrically into \(\ell^\infty\), thus
the space \(\ell^\infty\) would be a good candidate for the
`ambient space' of the fibers of the bundle.
\end{problem}

Finally, we now extend Banach--Mazur Theorem \ref{thm:Banach-Mazur}
to the setting of separable normed \(L^0(\mm)\)-modules; \emph{i.e.},
as we are going to see, we prove the existence of universal such modules.
\begin{definition}[Universal separable normed \(L^0\)-module]
Let \((\X,\Sigma,\mm)\) be a \(\sigma\)-finite
measure space. Let \(\mathscr M\) be a separable normed
\(L^0(\mm)\)-module. Then we say that \(\mathscr M\) is
a \emph{universal separable normed \(L^0(\mm)\)-module}
if for any separable normed \(L^0(\mm)\)-module
\(\mathscr N\) there exists a normed \(L^0(\mm)\)-module morphism
\(\mathcal I\colon\mathscr N\to\mathscr M\) that preserves
the pointwise norm.
\end{definition}

Given a measurable space \((\X,\Sigma)\) and a universal
separable Banach space \(\B\), we shall denote by \(\Gamma(\B)\)
the space of sections of the separable Banach
\(\B\)-bundle \(\X\ni x\mapsto\B\in{\rm Gr}(\B)\).
\begin{theorem}[Existence of universal modules]
Let \((\X,\Sigma,\mm)\) be a complete, \(\sigma\)-finite, separable
measure space. Let \(\B\) be a universal separable Banach space.
Then \(\Gamma(\B)\) is a universal separable normed \(L^0(\mm)\)-module.
\end{theorem}
\begin{proof}
Let \(\mathscr M\) be any given separable normed \(L^0(\mm)\)-module.
Theorem \ref{thm:Serre-Swan} grants the existence of a separable Banach
\(\B\)-bundle \(\mathbf E\) over \(\X\) such that \(\Gamma(\mathbf E)\cong\mathscr M\).
With a slight abuse of notation, we use the symbol \(\B\) to denote the separable Banach
\(\B\)-bundle \(\X\ni x\mapsto\B\). Consider the pre-morphism
\(\bar\varphi\) from \(\mathbf E\) to \(\B\) defined as follows: given any
\(x\in\X\), we declare that \(\bar\varphi(x,\cdot)\colon\mathbf E(x)\to\B\)
is the inclusion map. Call \(\varphi\colon\mathbf E\to\B\) the equivalence
class of \(\bar\varphi\). Therefore, it holds that
\(\Gamma(\varphi)\colon\Gamma(\mathbf E)\to\Gamma(\B)\)
is a morphism of normed \(L^0(\mm)\)-modules that preserves the pointwise norm.
The statement is achieved.
\end{proof}
\appendix
\section{Representation of normed modules via direct limits}
\label{s:repr_via_DL}
In this section we provide an alternative description of separable
normed \(L^0(\mm)\)-modules, which builds upon the representation
results for proper modules that have been proven in \cite{LP18}.

Roughly speaking, the strategy we will adopt is the following:
any separable normed \(L^0(\mm)\)-module \(\mathscr M\) can be
obtained as direct limit of finite-dimensional normed
\(L^0(\mm)\)-modules \((\mathscr M_k)_k\); each module
\(\mathscr M_k\) is the space of sections of some finite-dimensional
Banach bundle \(E_k\), thus by `patching together' the bundles \(E_k\)
we obtain some notion of separable Banach bundle \(E\), whose space of
sections can be eventually identified with \(\mathscr M\). Even though
this approach is much more `implicit' than the one proposed in
\S\ref{s:sep_Ban_bundle}, it has the advantage of clarifying how
to approximate separable normed \(L^0(\mm)\)-modules by proper ones.

A word on notation: for simplicity, we will use some terminology that
has been already used in \S\ref{s:sep_Ban_bundle} (such as
`separable Banach bundle" and so on), but with a different meaning.
Since this section is independent of \S\ref{s:sep_Ban_bundle},
we believe that this will not cause any ambiguity.
\bigskip

Let \((\X,\Sigma,\mm)\) be a fixed \(\sigma\)-finite measure space.
For the sake of simplicity, we shall write
\[
\mm_A\coloneqq\mm|_A\quad\text{ for every }A\in\Sigma
\text{ such that }\mm(A)>0.
\]
As observed in \cite[Section 2]{GPS18}, any normed
\(L^0(\mm_A)\)-module \(\mathscr N\) can be canonically viewed
as a normed \(L^0(\mm)\)-module; we shall denote it by
\({\rm Ext}_A(\mathscr N)\) and call it the \emph{extension} of
\(\mathscr N\).

On the other hand, a normed \(L^0(\mm)\)-module \(\mathscr M\)
can be `localised' on \(A\) as follows: we define \(\mathscr M|_A\)
as the pullback of \(\mathscr M\) under the identity map from
\((\X,\Sigma,\mm_A)\) and \((\X,\Sigma,\mm)\)
(cf.\ \cite[Section 1.6]{Gigli14}), so that \(\mathscr M|_A\)
is a normed \(L^0(\mm_A)\)-module. It is clear that
\({\rm Ext}_A(\mathscr M|_A)\) is isomorphic to the normed
\(L^0(\mm)\)-submodule \([\nchi_A]_\mm\cdot\mathscr M=
\big\{[\nchi_A]_\mm\cdot v\,:\,v\in\mathscr M\big\}\) of \(\mathscr M\).
\subsection{Separable Banach bundles}
We propose an alternative notion of separable Banach bundle over \(\X\),
which extends the one that has been introduced in \cite{LP18}.
The language we adopt here is slightly different from that of \cite{LP18},
but it can be readily checked that the two resulting theories are fully
consistent.
\bigskip

We fix the notation \(\bar\N\coloneqq\N\cup\{\infty\}\).
Given any \(n\in\bar\N\), we define the vector space \(\mathbb V_n\) as
\[
\mathbb V_n\coloneqq\left\{\begin{array}{ll}
\R^n\\
c_{00}
\end{array}\quad\begin{array}{ll}
\text{ if }n<\infty,\\
\text{ if }n=\infty,
\end{array}\right.
\]
where \(c_{00}\) stands for the space of all sequences in \(\R\) having only
finitely many non-zero terms. Calling \((e_n)_{n\in\N}\) the canonical basis
of \(\mathbb V_\infty\) (\emph{i.e.}, \(e_n\coloneqq(\delta_{in})_{i\in\N}\)
for all \(n\in\N\)), we shall always implicitly identify \(\mathbb V_n\)
with the subspace of \(\mathbb V_\infty\) spanned by \(e_1,\ldots,e_n\).

The topology we shall consider on the space \(\mathbb V_\infty\) is the one
induced by the \(\ell^\infty\)-norm. It is straightforward to check
that a set \(S\subseteq c_{00}\) belongs to the Borel \(\sigma\)-algebra
associated to such topology if and only if \(S\cap\mathbb V_n\) is a
Borel subset of \((\R^n,\sfd_{\rm Eucl})\) for every \(n\in\N\).
\begin{definition}[Banach bundle of dimension $n$]\label{def:BB_dim_n}
Let \(n\in\bar\N\) be given. Then we say that a couple \(E=(A,\nnorm)\) is a
\emph{Banach bundle of dimension \(n\)} over a given set \(A\in\Sigma\)
provided the function \(\nnorm\colon A\times\mathbb V_n\to[0,+\infty)\)
is measurable and satisfies the following property:
\[
\nnorm(x,\cdot)\text{ is a norm on }\mathbb V_n
\quad\text{ for every }x\in A.
\]
\end{definition}

Notice that $(A,\nnorm)\) is a Banach bundle of dimension \(n\in\bar\N\)
if and only if \((A,\nnorm|_{A\times\mathbb V_k})\) is a Banach
bundle of dimension \(k\) for every \(k\in\N\) satisfying \(k\leq n\).
\begin{remark}{\rm
We observe that if \(n=\infty\), then the norms
\(\boldsymbol{\sf n}(x,\cdot)\) on \(\mathbb V_\infty\) cannot be complete,
as the vector space \(c_{00}\) does not support any complete norm;
cf.\ for instance \cite{AliprantisBorder99}.
\fr}\end{remark}

Let us consider a Banach bundle \(E=(A,\nnorm)\) of dimension \(n\in\N\).
Then the space \(\Gamma_A(E)\) of \emph{sections} of \(E\) is defined as the
family of all measurable maps \(s\colon A\to\R^n\), considered up to
\(\mm_A\)-a.e.\ equality. As shown in \cite{LP18}, it turns out that
\(\Gamma_A(E)\) is a normed \(L^0(\mm_A)\)-module when endowed with the
natural pointwise operations and the following pointwise norm:
\[
|s|(x)\coloneqq\boldsymbol{\sf n}\big(x,s(x)\big)
\quad\text{ for }\mm_A\text{-a.e.\ }x\in A
\]
for every \(s\in\Gamma_A(E)\). More precisely, \(\Gamma_A(E)\) is a free
\(L^0(\mm_A)\)-module of rank \(n\).
\bigskip

We now define the space of sections of a Banach bundle \(E=(A,\nnorm)\)
of any dimension \(n\in\bar\N\), possibly \(n=\infty\).
Call \(I\) the set of all \(k\in\N\) with \(k\leq n\).
We set \(E_k\coloneqq(A,\nnorm|_{A\times\mathbb V_k})\)
for all \(k\in I\). Given any \(j,k\in I\) with \(j\leq k\),
we have a canonical inclusion
map \(\iota_{jk}\colon\Gamma_A(E_j)\hookrightarrow\Gamma_A(E_k)\).
Namely, \(\iota_{jk}\) is the map sending (the equivalence class of)
any section \(s=(s_1,\ldots,s_j)\) of \(E_j\) to (the equivalence class of)
the section \((s_1,\ldots,s_j,0\ldots,0)\) of \(E_k\).
It can be readily checked that
\(\big(\big\{\Gamma_A(E_k)\big\}_{k\in I},\{\iota_{jk}\}_{j\leq k}\big)\)
is a direct system in the category of normed \(L^0(\mm_A)\)-modules.
Then we define
\[
\Gamma_A(E)\coloneqq\varinjlim\Gamma_A(E_\star),
\]
whose existence is granted by Theorem \ref{thm:direct_limit}.
Notice that such definition of \(\Gamma_A(E)\) is consistent with
the previous one when \(E\) is a Banach bundle of finite dimension;
cf.\ \cite[Lemma 2.10]{Pas19}.
\begin{definition}[Banach bundle]
We say that \(E=\big\{(A_n,E_n)\big\}_{n\in\overline\N}\) is a
\emph{separable Banach bundle} over \(\X\) provided
\(\{A_n\}_{n\in\overline\N}\subseteq\Sigma\) is a partition of \(\X\)
and each \(E_n=(A_n,\nnorm_n)\) is a Banach bundle of dimension \(n\).
Moreover, we say that \(E\) is \emph{proper} provided \(\mm(A_\infty)=0\).
\end{definition}

Let \(E=\big\{(A_n,E_n)\big\}_{n\in\overline\N}\)
be a separable Banach bundle over \(\X\).
Then we define the space of its sections as
\[
\Gamma(E)\coloneqq\prod_{n\in\overline\N}{\rm Ext}_{A_n}
\big(\Gamma_{A_n}(E_n)\big).
\]
The direct product \(\Gamma(E)\) inherits an (algebraic)
\(L^0(\mm)\)-module structure. Moreover, the fact that
the sets \(A_n\) are pairwise disjoint grants that
the following definition is meaningful:
\begin{equation}\label{eq:ptwse_norm_Gamma(E)}
|s|\coloneqq\sum_{n\in\overline\N}[\nchi_{A_n}]_\mm\,|s_n|
\quad\text{ for every }s=\{s_n\}_{n\in\overline\N}\in\Gamma(E).
\end{equation}
It is straightforward to verify that \eqref{eq:ptwse_norm_Gamma(E)}
actually defines a pointwise norm on \(\Gamma(E)\), whose associated
distance \(\sfd_{\Gamma(E)}\) is complete. Therefore, \(\Gamma(E)\)
is a normed \(L^0(\mm)\)-module.
\subsection{Representation theorem}
The purpose of this subsection is to show that any separable normed
\(L^0(\mm)\)-module \(\mathscr M\) is isomorphic to the space of sections
\(\Gamma(E)\) of some separable Banach bundle \(E\) over \(\X\).
\bigskip

In the sequel, we will need the following consequence of
\cite[Theorem 3]{LP18}, which we rephrase in the current language.
Actually, the result was obtained for modules on metric measure
spaces, but the same proof can be repeated verbatim in the
case of \(\sigma\)-finite measure spaces.
\begin{theorem}[Representation theorem for proper modules]
\label{thm:SS_pr}
Let \((\X,\Sigma,\mm)\) be a \(\sigma\)-finite measure space.
Let \(\mathscr M\) be a proper normed \(L^0(\mm)\)-module.
Then there exists a proper Banach bundle \(E\) over \(\X\)
such that \(\mathscr M\) is isomorphic to \(\Gamma(E)\) as
a normed \(L^0(\mm)\)-module.
\end{theorem}
\begin{remark}\label{rmk:build_approx}{\rm
Let \(\mathscr M\) be a separable normed \(L^0(\mm)\)-module that is not
finitely-generated on any measurable subset of \(\X\) having positive
\(\mm\)-measure. We claim that there exists an increasing sequence
\((\mathscr N_n)_{n\in\N}\) of normed \(L^0(\mm)\)-submodules of
\(\mathscr M\) with the following properties:
\begin{itemize}
\item[\(\rm i)\)] Each \(\mathscr N_n\) has local dimension \(n\) on \(\X\).
\item[\(\rm ii)\)] The set \(\bigcup_{n\in\N}\mathscr N_n\) is dense
in \(\mathscr M\).
\end{itemize}
We construct the desired modules \(\mathscr N_n\) is a recursive way.
Fix any dense subset \((v_n)_{n\in\N}\) of \(\mathscr M\).
We aim to build a sequence \((\mathscr N_n)_{n\in\N}\) of normed
\(L^0(\mm)\)-modules satisfying i) and such that each \(\mathscr N_n\)
contains the elements \(v_1,\ldots,v_n\). This would clearly imply ii).
First, we define \(\mathscr N_1\) as the normed \(L^0(\mm)\)-module
generated by the element
\[
\sum_{k\in\N}\big[\nchi_{\{|v_k|>0\}\setminus\bigcup_{j<k}\{|v_j|>0\}}
\big]_\mm\cdot v_k\in\mathscr M.
\]
Then \(\mathscr N_1\) has dimension \(1\) (as \(\mathscr M\) is
not finitely-generated on any measurable set) and \(v_1\in\mathscr N_1\).

Now suppose to have already defined \(\mathscr N_n\) for some \(n\in\N\).
We want to define \(\mathscr N_{n+1}\). Fix a local basis
\(w_1,\ldots,w_n\) of \(\mathscr N_n\). For any \(k\geq n+1\) we call
\(B'_k\) the set where \(w_1,\ldots,w_n,v_k\) are independent and
\(B_k\coloneqq B'_k\setminus\bigcup_{j=n+1}^{k-1}B'_j\); then we
define \(\mathscr N_{n+1}\) as the normed \(L^0(\mm)\)-module generated
by \(\mathscr N_n\cup\{w_{n+1}\}\), where we put
\(w_{n+1}\coloneqq\sum_{k\geq n+1}[\nchi_{B_k}]_\mm\cdot v_k\).
Hence, \(\mathscr N_{n+1}\) has local dimension equal to \(n+1\) on \(\X\)
(as \(\mathscr M\) is not finitely-generated on any measurable set)
and contains the elements \(v_1,\ldots,v_{n+1}\) by construction.
\fr}\end{remark}

By building on top of Theorem \ref{thm:SS_pr}, we can eventually
prove the following result:
\begin{theorem}[Representation theorem]\label{thm:Gamma_surj_obj}
Let \((\X,\Sigma,\mm)\) be any \(\sigma\)-finite measure space.
Let \(\mathscr M\) be a separable normed \(L^0(\mm)\)-module.
Then there exists a separable Banach bundle \(E\) over the space \(\X\)
such that \(\Gamma(E)\cong\mathscr M\).
\end{theorem}
\begin{proof}
Let us call \(\{A_n\}_{n\in\overline\N}\) the dimensional decomposition
of the module \(\mathscr M\). Consider the normed
\(L^0(\mm_{A_\infty})\)-module
\(\mathscr N\coloneqq\mathscr M|_{A_\infty}\). As shown in Remark
\ref{rmk:build_approx}, one can build an increasing sequence
\((\mathscr N_k)_{k\in\N}\) of normed \(L^0(\mm_{A_\infty})\)-submodules
of \(\mathscr N\) such that each \(\mathscr N_k\) has local dimension
equal to \(k\) on the set \(A_\infty\) and
\(\bigcup_{k\in\N}\mathscr N_k\) is dense in \(\mathscr N\).
Let us pick any sequence \((v_k)_{k\in\N}\subseteq\mathscr N\) such that
\begin{equation}\label{eq:Gamma_surj_obj_aux}
v_1,\ldots,v_k\;\text{ is a local basis for }\mathscr N_k
\text{ on }A_\infty\text{ for every }k\in\N.
\end{equation}
Given \(k\in\N\), we can find (by Theorem \ref{thm:SS_pr}) a Banach bundle
\(F'_k=(A_\infty,\boldsymbol{\sf n}'_k)\) of dimension \(k\) such that
\(\Gamma_{A_\infty}(F'_k)\cong\mathscr N_k\). Under such isomorphism,
the elements \(v_1,\ldots,v_k\in\mathscr N_k\) correspond to some sections
\(s^k_1,\ldots,s^k_k\in\Gamma(F'_k)\), respectively.
Pick representatives \(\bar s^k_1,\ldots,\bar s^k_k\) of them.
Then we know from \eqref{eq:Gamma_surj_obj_aux} that there is a
measurable set \(N\subseteq A_\infty\) with \(\mm(N)=0\) such that
\[
\bar s^k_1(x),\ldots,\bar s^k_k(x)\quad\text{ is a basis of }\R^k
\text{ for every }k\in\N\text{ and }x\in A_\infty\setminus N.
\]
Therefore, for any \(k\in\N\) we define a new Banach bundle
\(F_k=(A_\infty,\boldsymbol{\sf n}_k)\) of dimension \(k\) as
\[
\boldsymbol{\sf n}_k\big(x,(\lambda_1,\ldots,\lambda_k)\big)
\coloneqq\left\{\begin{array}{ll}
\boldsymbol{\sf n}'_k\big(x,\lambda_1\,
\bar s^k_1(x)+\ldots+\lambda_k\,\bar s^k_k(x)\big)\\
(\lambda_1^2+\ldots+\lambda_k^2)^{1/2}
\end{array}\quad\begin{array}{ll}
\text{ if }x\in A_\infty\setminus N,\\
\text{ if }x\in N.
\end{array}\right.
\]
It follows that \(\Gamma_{A_\infty}(F_k)\cong\mathscr N_k\) and
\(\boldsymbol{\sf n}_{k+1}|_{A_\infty\times\mathbb V_k}=
\boldsymbol{\sf n}_k\) for every \(k\in\N\). Hence, we can consider
the Banach bundle \(E_\infty=(A_\infty,\nnorm_\infty)\) of dimension
\(\infty\), where
\(\nnorm_\infty\colon A_\infty\times\mathbb V_\infty\to[0,+\infty)\)
is defined as the unique function satisfying
\(\nnorm_\infty|_{A_\infty\times\mathbb V_k}=\nnorm_k\) for all \(k\in\N\).
As granted by Lemma \ref{lem:countable_DL}, the direct limit of
\(\Gamma_{A_\infty}(F_k)\cong\mathscr N_k\) is isomorphic to \(\mathscr N\),
thus \(\Gamma_{A_\infty}(E_\infty)\cong\mathscr M|_{A_\infty}\).
To conclude, notice that \(\mathscr M|_{\X\setminus A_\infty}\)
is a proper normed \(L^0(\mm_{\X\setminus A_\infty})\)-module,
whence accordingly
\[
\mathscr M|_{\X\setminus A_\infty}\cong
\Gamma_{\X\setminus A_\infty}(E')\quad\text{ for some
proper Banach bundle }E'=\big\{(A_n,E_n)\big\}_{n\in\N}
\]
as a consequence of Theorem \ref{thm:SS_pr}.
Hence, the separable Banach bundle
\(E\coloneqq\big\{(A_n,E_n)\big\}_{n\in\overline\N}\)
satisfies \(\Gamma(E)\cong\mathscr M\), as required.
\end{proof}
\begin{remark}[`Serre--Swan theorem']{\rm
Given a \(\sigma\)-finite measure space \((\X,\Sigma,\mm)\)
and two separable Banach bundles \(E=\big\{(A_n,E_n)\big\}_{n\in\bar\N}\)
and \(F=\big\{(B_m,F_m)\big\}_{m\in\bar\N}\) over \(\X\), we can define a
\emph{pre-morphism} between \(E\) and \(F\) as a family
\(\varphi=\{\varphi_{nm}\}_{n,m\in\bar\N}\) of measurable maps
\[
\varphi_{nm}\colon(A_n\cap B_m)\times\mathbb V_n\longrightarrow
(A_n\cap B_m)\times\mathbb V_m
\]
such that \(\varphi_{nm}(x,\mathbb V_n)\subseteq\{x\}\times\mathbb V_m\) and
\[
\varphi_{nm}(x,\cdot)\colon\big(\mathbb V_n,\nnorm^E_n(x,\cdot)\big)
\longrightarrow\big(\mathbb V_m,\nnorm^F_m(x,\cdot)\big)
\quad\text{ is a linear contraction}
\]
for every \(n,m\in\bar\N\) and \(x\in A_n\cap B_m\),
where we call \(E_n=(A_n,\nnorm^E_n)\) and \(F_m=(B_m,\nnorm^F_m)\).

We declare two pre-morphisms \(\{\varphi_{nm}\}_{n,m\in\bar\N}\) and
\(\{\psi_{nm}\big\}_{n,m\in\bar\N}\) between \(E\) and \(F\) to be
equivalent provided there exists a set \(N\in\Sigma\) with \(\mm(N)=0\)
such that
\[
\varphi_{nm}(x,\cdot)=\psi_{nm}(x,\cdot)\quad\text{ for every }
x\in\X\setminus N\text{ and }n,m\in\bar\N.
\]
Therefore, it makes sense to consider the category of separable Banach
bundles over \(\X\) having the equivalence classes of pre-morphisms as
arrows. We point out that -- in analogy with what done in
\S\ref{ss:section_functor} -- it is possible to promote the correspondence
\(E\mapsto\Gamma(E)\) to a functor (called the \emph{section functor})
from the category of separable Banach bundles over \(\X\) to that of
separable normed \(L^0(\mm)\)-modules. Furthermore,
the section functor can be shown to be an equivalence of categories
(the so-called `Serre--Swan theorem'). We omit the details.
\fr}\end{remark}
\section{Cotangent bundle on metric measure spaces}\label{s:cotg_bundle}
As mentioned in the introduction, in the study of the differential structure
of metric measure spaces a key role is played by the so-called cotangent
module \(L^0({\rm T}^*\X)\), which has been introduced by Gigli in
\cite{Gigli14}. We now propose an alternative axiomatisation, at least in
the case in which \(W^{1,2}(\X)\) is separable. In the approach we are going
to present, we directly introduce a notion of \emph{cotangent bundle}
\({\rm T}^*\X\) over \((\X,\sfd,\mm)\), which does not require the theory
of normed modules to be formulated (cf.\ Remark \ref{rmk:no_mod}).
Even though we just consider \(p=2\) for simplicity, a similar
construction could be performed for any \(p\in(1,\infty)\).
\bigskip

By \emph{metric measure space} \((\X,\sfd,\mm)\) we mean a
complete, separable metric space \((\X,\sfd)\), endowed
with a non-negative Radon measure \(\mm\). Calling \(\Sigma\)
the completion of the Borel \(\sigma\)-algebra \(\mathscr B(\X)\)
and \(\bar\mm\) the completion of the measure \(\mm\), it holds
that \((\X,\Sigma,\bar\mm)\) is a complete, \(\sigma\)-finite measure
space, thus in particular the results of \S\ref{ss:repr_thm} can
be applied. Observe that normed \(L^0(\mm)\)-modules and normed
\(L^0(\bar\mm)\)-modules can be identified in a canonical way.

We denote by \({\rm LIP}(\X)\) the space of all real-valued Lipschitz
functions on \(\X\). Given any function \(f\in{\rm LIP}(\X)\), its
\emph{slope} \({\rm lip}(f)\colon\X\to[0,+\infty)\) is defined as
\[
{\rm lip}(f)(x)\coloneqq\lims_{y\to x}\frac{\big|f(x)-f(y)\big|}{\sfd(x,y)}
\quad\text{ if }x\in\X\text{ is an accumulation point,}
\]
and \({\rm lip}(f)(x)\coloneqq 0\) otherwise. We define the
\emph{Cheeger energy} \({\rm Ch}\colon L^2(\mm)\to[0,+\infty]\) on \(\X\) as
\[
{\rm Ch}(f)\coloneqq\inf\bigg\{\limi_{n\to\infty}\int{\rm lip}^2(f_n)
\,\d\mm\;\bigg|\;(f_n)_n\subseteq{\rm LIP}(\X),\,[f_n]_\mm\in L^2(\mm),
\,[f_n]_\mm\to f\text{ in }L^2(\mm)\bigg\}
\]
for every \(f\in L^2(\mm)\). Then (following \cite{Cheeger00})
we define the \emph{Sobolev space} on \((\X,\sfd,\mm)\) as
\[
W^{1,2}(\X)\coloneqq\Big\{f\in L^2(\mm)\;\Big|\;{\rm Ch}(f)<+\infty\Big\}.
\]
Given any \(f\in W^{1,2}(\X)\), there exists a distinguished
function \(|Df|\in L^2(\mm)\) -- called the \emph{minimal relaxed slope}
of \(f\) -- which provides the integral representation
\({\rm Ch}(f)=\int|Df|^2\,\d\mm\) of the Cheeger energy.
It turns out that \(W^{1,2}(\X)\) is a Banach space if endowed
with the norm
\[
\|f\|_{W^{1,2}(\X)}\coloneqq\bigg(\int|f|^2\,\d\mm+{\rm Ch}(f)\bigg)^{1/2}
\quad\text{ for every }f\in W^{1,2}(\X).
\]
In the case in which the measure \(\mm\) is boundedly finite,
some equivalent notions of Sobolev spaces have been introduced
in \cite{Shanmugalingam00,AmbrosioGigliSavare11,Ambrosio-DiMarino14}.
It is worth pointing out that in most cases the Sobolev space
\(W^{1,2}(\X)\) is separable. Indeed, as proven in \cite{ACM14},
the separability is granted by the reflexivity of the Sobolev space,
which is in turn known to hold on a vast class of metric measure
spaces; for instance, whenever the underlying metric space
\((\X,\sfd)\) is doubling (cf.\ \cite{ACM14}) or carries a `nice' geometric
structure (such as Euclidean spaces, Carnot groups, Finsler manifolds,
subRiemannian manifolds, and locally \({\sf CAT}(\kappa)\) spaces;
cf.\ the introduction of \cite{LDLP19}). To the best of our knowledge,
the only known examples of a non-separable Sobolev space are described
in \cite{ACM14}.
\bigskip

With the terminology introduced above at our disposal, we can
now state and prove our existence and uniqueness result about the
cotangent bundle of a metric measure space:
\begin{theorem}[Cotangent bundle]\label{thm:cotg_bundle}
Let \((\X,\sfd,\mm)\) be a metric measure space such that the Sobolev
space \(W^{1,2}(\X)\) is separable. Let \(\B\) be a universal separable
Banach space. Then there exists a unique couple \(({\rm T}^*\X,\d)\),
where \({\rm T}^*\X\) is a separable Banach \(\B\)-bundle over \(\X\)
(in the sense of Definition \ref{def:sep_BBbundle}) called the
\emph{cotangent bundle} of \((\X,\sfd,\mm)\) and
\(\d\colon W^{1,2}(\X)\to\Gamma({\rm T}^*\X)\) is a linear operator
called the \emph{differential}, such that the following properties
are satisfied:
\begin{itemize}
\item[\(\rm i)\)] It holds that \(|\d f|=|Df|\) in the \(\mm\)-a.e.\ sense
for every \(f\in W^{1,2}(\X)\).
\item[\(\rm ii)\)] Given any dense sequence \((f_n)_n\) in \(W^{1,2}(\X)\),
it holds that
\[
\big\{\d f_n(x)\;\big|\;n\in\N\big\}\;\text{ is dense in }
{\rm T}^*_x\X\coloneqq{\rm T}^*\X(x)\text{ for }\mm\text{-a.e.\ }x\in\X.
\]
\end{itemize}
Uniqueness is intended up to unique isomorphism: given any other couple
\((\mathbf E,\d')\) with the same properties, there exists a unique
isomorphism \(\varphi\colon{\rm T}^*\X\to\mathbf E\) such that
\[\begin{tikzcd}
W^{1,2}(\X) \arrow{r}{\d} \arrow[swap]{rd}{\d'}
& \Gamma({\rm T}^*\X) \arrow{d}{\Gamma(\varphi)} \\
& \Gamma(\mathbf E)
\end{tikzcd}\]
is a commutative diagram.
\end{theorem}
\begin{proof}
First of all, recall the notion of cotangent module introduced
in \cite[Definition 2.2.1]{Gigli14} (cf.\ \cite[Proposition 4.1.8]{GP20}
for the formulation we will present, via normed
\(L^0(\mm)\)-modules): there exists a unique couple
\(\big(L^0({\rm T}^*\X),\d\big)\), where \(L^0({\rm T}^*\X)\) is a separable
normed \(L^0(\mm)\)-module and \(\d\colon W^{1,2}(\X)\to L^0({\rm T}^*\X)\)
is a linear operator, such that:
\begin{itemize}
\item[\(\rm i')\)] \(|\d f|=|Df|\) holds \(\mm\)-a.e.\ for every
\(f\in W^{1,2}(\X)\).
\item[\(\rm ii')\)] Given any \((f_n)_n\) dense in
\(W^{1,2}(\X)\), the family \(\{\d f_n\,:\,n\in\N\}\) generates
\(L^0({\rm T}^*\X)\).
\end{itemize}
(The separability of \(L^0({\rm T}^*\X)\) is granted, \emph{e.g.},
by \cite[Lemma 3.1.17]{GP20}.) Uniqueness is intended in the following
sense: given another couple \((\mathscr M,\d')\) having the same
properties, there exists a unique isomorphism
\(\Phi\colon L^0({\rm T}^*\X)\to\mathscr M\) of normed \(L^0(\mm)\)-modules
such that \(\Phi\circ\d=\d'\).

By applying Lemma \ref{lem:Gamma_full}, we find a
separable Banach \(\B\)-bundle \({\rm T}^*\X\) over the space \(\X\)
such that \(\Gamma({\rm T}^*\X)\cong L^0({\rm T}^*\X)\). It is clear
that the map \(\d\colon W^{1,2}(\X)\to\Gamma({\rm T}^*\X)\) satisfies
the item \(\rm i)\) as a consequence of \(\rm i')\). Moreover, the
validity of the item \(\rm ii)\) can be deduced from \(\rm ii')\) by
suitably adapting the arguments
in the proof of Proposition \ref{prop:equiv_Banach_bundle}.
Finally, the uniqueness of the couple \(({\rm T}^*\X,\d)\) up to unique
isomorphism can be obtained by combining the analogous property of
\(\big(L^0({\rm T}^*\X),\d\big)\) with the fact (Theorem
\ref{thm:Serre-Swan}) that the section functor is an equivalence
of categories. Therefore, the statement is achieved.
\end{proof}
\begin{remark}\label{rmk:no_mod}{\rm
The theory of normed modules is not really used in Theorem
\ref{thm:cotg_bundle}, in the sense that the spaces \(\Gamma({\rm T}^*\X)\)
and \(\Gamma({\textbf E})\) can be just considered as vector spaces,
without looking at their module structure. The same observation
applies to the morphism \(\Gamma(\varphi)\) as well.
\fr}\end{remark}
\section{The pullback bundle}\label{s:pullback_bundle}
In the field of nonsmooth differential geometry, a prominent role is played
by the notion of pullback of a normed module, which we recalled in Theorem
\ref{thm:pullback_mod}. The primary purpose of this section is to show that
separable Banach bundles come with a natural notion of pullback (see Definition
\ref{def:pullback_bundle}) which is consistent with that of normed
modules (see Theorem \ref{thm:pullback_sect_comm}). For technical reasons
(namely, because we will need the Disintegration Theorem \ref{thm:disint}),
we will mostly work with metric measure spaces (instead of general \(\sigma\)-finite measure
spaces). Furthermore, we study the projection operator
\({\rm Pr}_\varphi\) associated with a normed \(L^0\)-module \(\mathscr M\)
(see Theorem \ref{thm:Pr_phi_mod}), which is a distinguished left inverse of
the pullback map \(\varphi^*\colon\mathscr M\to\varphi^*\mathscr M\). In the setting
of separable normed modules, we also provide a fiberwise description of the
operator \({\rm Pr}_\varphi\); see Proposition \ref{prop:char_Pr_phi_mod} for the details.
\subsection{Pullback and section functors commute}
Let us begin with the definition of pullback of a separable Banach bundle.
\begin{definition}[Pullback bundle]\label{def:pullback_bundle}
Let \((\X,\Sigma_\X)\), \((\Y,\Sigma_\Y)\) be measurable spaces and
\(\B\) a universal separable Banach space. Let \(\varphi\colon\X\to\Y\) be a
measurable map. Let \(\mathbf E\) be a separable Banach \(\B\)-bundle
over \(\Y\). Then we define the map
\(\varphi^*\mathbf E\colon\X\to{\rm Gr}(\B)\) as
\[
\varphi^*\mathbf E(x)\coloneqq\mathbf E\big(\varphi(x)\big)
\quad\text{ for every }x\in\X.
\]
It follows from Lemma \ref{lem:comp_corr} that \(\varphi^*\mathbf E\) is a
separable Banach \(\B\)-bundle over \(\X\). We call it the \emph{pullback bundle}
of \(\mathbf E\) under the map \(\varphi\).
\end{definition}
Let \((\X,\sfd_\X)\), \((\Y,\sfd_\Y)\) be separable metric spaces.
By \(\mathscr P(\X)\) we mean the family of all Borel probability measures
\(\mu\) on \(\X\), \emph{i.e.}, with \(\mu(\X)=1\). Then a family
\(\{\mu^y\}_{y\in\Y}\subseteq\mathscr P(\X)\) is said to be a
\emph{Borel family of measures} provided for any Borel function
\(f\colon\X\to[0,+\infty]\) it holds that
\[
\Y\ni y\longmapsto\int f\,\d\mu^y\in[0,+\infty]
\quad\text{ is Borel measurable.}
\]
We need the following classical result, whose proof can be found,
\emph{e.g.}, in \cite[Theorem 5.3.1]{AmbrosioGigliSavare08}.
\begin{theorem}[Disintegration theorem]\label{thm:disint}
Let \((\X,\sfd_\X,\mm_\X)\), \((\Y,\sfd_\Y,\mm_\Y)\) be metric measure spaces,
with \(\mm_\X\), \(\mm_\Y\) finite. Let \(\varphi\colon\X\to\Y\) be a Borel map
satisfying \(\varphi_*\mm_\X=\mm_\Y\). Then there exists a \(\mm_\Y\)-a.e.\ uniquely
determined Borel family of measures \(\{\mm_\X^y\}_{y\in\Y}\subseteq\mathscr P(\X)\)
such that
\begin{subequations}\begin{align}\label{eq:disint1}
\mm_\X^y\big(\X\setminus\varphi^{-1}(y)\big)=0&\quad\text{ for }\mm_\Y\text{-a.e.\ }
y\in\Y,\\
\label{eq:disint2}\int f\,\d\mm_\X=\int\bigg(\int f\,\d\mm_\X^y\bigg)\d\mm_\Y(y)&
\quad\text{ for every }f\colon\X\to[0,+\infty]\text{ Borel.}
\end{align}\end{subequations}
We abbreviate the conditions \eqref{eq:disint1} and \eqref{eq:disint2}
to the single expression \(\mm_\X=\int\mm_\X^y\,\d\mm_\Y(y)\).
\end{theorem}
We are now ready to prove that the sections of the pullback bundle can be identified
with the elements of the pullback of the space of sections of the bundle itself.
\begin{theorem}[Pullback and section functors commute]\label{thm:pullback_sect_comm}
Let \((\X,\sfd_\X,\mm_\X)\), \((\Y,\sfd_\Y,\mm_\Y)\) be metric measure spaces,
with \(\mm_\X\), \(\mm_\Y\) finite. Let \(\B\) be a universal separable Banach
space and let \(\varphi\colon\X\to\Y\) be a Borel map satisfying
\(\varphi_*\mm_\X=\mm_\Y\). Let \(\mathbf E\) be a separable Banach \(\B\)-bundle
over \(\Y\). Then it holds that
\[
\varphi^*\Gamma(\mathbf E)\cong\Gamma(\varphi^*\mathbf E),
\]
the pullback map \(\varphi^*\colon\Gamma(\mathbf E)\to\Gamma(\varphi^*\mathbf E)\)
being given by \((\varphi^*s)(x)\coloneqq s\big(\varphi(x)\big)\) for
\(\mm_\X\)-a.e.\ \(x\in\X\).
\end{theorem}
\begin{proof}
First of all, observe that \(|\varphi^*s|(x)=\big\|(s\circ\varphi)(x)\big\|_\B
=\big(|s|\circ\varphi\big)(x)\) holds for every \(s\in\Gamma(\mathbf E)\) and
\(\mm_\X\)-a.e.\ \(x\in\X\). Consequently, in order to prove the statement,
it suffices to show that the family \(\big\{\varphi^*s\,:\,s\in\Gamma(\mathbf E)\big\}\)
generates \(\Gamma(\varphi^*\mathbf E)\) on \(\X\). To this aim, let
\(t\in\Gamma(\varphi^*\mathbf E)\) and \(\varepsilon>0\) be fixed. Thanks
to the separability of \(\B\), we can find a Borel partition \((A_n)_{n\in\N}\)
of \(\X\) such that for every \(n\in\N\) it holds that
\(\big\|t(x)-t(x')\big\|_\B\leq\varepsilon\) for \(\mm_\X\)-a.e.\ \(x,x'\in A_n\).
By using Theorem \ref{thm:disint}, we can disintegrate the measure
\(\mm_\X\) along \(\varphi\) as \(\int\mm_\X^y\,\d\mm_\Y(y)\).
Then let us define
\begin{equation}\label{eq:pullback_sect_comm_aux1}
\tilde t\coloneqq\sum_{n\in\N}[\nchi_{A_n}]_{\mm_\X}\cdot\varphi^*s_n
\in\Gamma(\varphi^*\mathbf E),
\end{equation}
where for every \(n\in\N\) the section \(s_n\in\Gamma(\mathbf E)\) is given by
\begin{equation}\label{eq:pullback_sect_comm_aux2}
s_n(y)\coloneqq\fint_{A_n}t(x)\,\d\mm_\X^y(x)
\in\mathbf E(y)\quad\text{ for }\mm_\Y\text{-a.e.\ }y\in\Y,
\end{equation}
with the convention that \(\fint_{A_n}t\,\d\mm_\X^y\coloneqq 0_{\mathbf E(y)}\)
if \(\mm_\X^y(A_n)=0\). Some verifications are in order: in view of \eqref{eq:disint1},
we know that for \(\mm_\Y\)-a.e.\ \(y\in\Y\) it holds \(t(x)\in\mathbf E(y)\)
for \(\mm_\X^y\)-a.e.\ \(x\in\X\). Being the map
\(t\) Borel (thus strongly Borel, as \(\B\) is separable) and bounded on the set
\(A_n\), we have that the Bochner integral \(\fint_{A_n}t\,\d\mm_\X^y\in\mathbf E(y)\)
exists for \(\mm_\Y\)-a.e.\ \(y\in\Y\), whence the well-posedness of
\eqref{eq:pullback_sect_comm_aux1} follows. Moreover, given any element
\(\omega\in\B'\), it holds that
\[
\omega\big[s_n(y)\big]=\nchi_{\{y'\in\Y\,:\,\mm_\X^{y'}(A_n)>0\}}(y)\,
\frac{\int\omega[t(x)]\,\d\mm_\X^y(x)}{\int\nchi_{A_n}\,\d\mm_\X^y}
\quad\text{ for }\mm_\Y\text{-a.e.\ }y\in\Y.
\]
Given that \(\{\mm_\X^y\}_{y\in\Y}\) is a Borel family of measures, we deduce
that \(\Y\ni y\mapsto\omega\big[s_n(y)\big]\in\R\) is a Borel function, thus accordingly
the map \(s_n\colon\Y\to\B\) is weakly Borel. Being \(\B\) separable, we conclude that
\(s_n\) is Borel. Therefore, we have that \((s_n)_{n\in\N}\subseteq\Gamma(\mathbf E)\),
so that the definition in \eqref{eq:pullback_sect_comm_aux1} is meaningful.
Given any \(n\in\N\) and \(\mm_\X\)-a.e.\ \(x\in A_n\), we may estimate
\[\begin{split}
\big\|\tilde t(x)-t(x)\big\|_\B&=\big\|t(x)-(\varphi^*s_n)(x)\big\|_\B
=\bigg\|t(x)-\fint_{A_n}t(x')\,\d\mm_\X^{\varphi(x)}(x')\bigg\|_\B\\
&=\bigg\|\fint_{A_n}\big(t(x)-t(x')\big)\,\d\mm_\X^{\varphi(x)}(x')\bigg\|_\B
\leq\fint_{A_n}\big\|t(x)-t(x')\big\|_\B\,\d\mm_\X^{\varphi(x)}(x')
\leq\varepsilon,
\end{split}\]
whence \(\sfd_{\Gamma(\varphi^*\mathbf E)}(\tilde t,t)\leq\varepsilon\).
This shows that \(\big\{\varphi^*s\,:\,s\in\Gamma(\mathbf E)\big\}\) generates
\(\Gamma(\varphi^*\mathbf E)\) on \(\X\).
\end{proof}
\begin{remark}{\rm
Some of the assumptions of Theorem \ref{thm:pullback_sect_comm} might be dropped:
\begin{itemize}
\item[\(\rm i)\)] The result holds whenever \(\mm_\X\) and \(\mm_\Y\)
are \(\sigma\)-finite. Indeed, by arguing as we did at the beginning of
\S\ref{ss:norm_L0_mod}, one can find \(\mm_\X'\in\mathscr P(\X)\) such
that \(\mm_\X\ll\mm_\X'\ll\mm_\X\). Observe that \(\mm_\Y\ll\varphi_*\mm_\X'\ll\mm_\Y\).
Given that \(L^0(\mm_\X')=L^0(\mm_\X)\) and \(L^0(\varphi_*\mm_\X')=L^0(\mm_\Y)\),
the claim immediately follows from Theorem \ref{thm:pullback_sect_comm}.
\item[\(\rm ii)\)] In addition, the assumption \(\varphi_*\mm_\X=\mm_\Y\) can be relaxed
to \(\varphi_*\mm_\X\ll\mm_\Y\) (requiring that the measure \(\varphi_*\mm_\X\) is
\(\sigma\)-finite when \(\mm_\X(\X)=+\infty\)).
To show it, consider the following equivalence relation \(\sim\) on
a given normed \(L^0(\mm_\Y)\)-module \(\mathscr M\): for any \(v,w\in\mathscr M\),
we declare that \(v\sim w\) if and only if \(|v-w|=0\) holds \(\varphi_*\mm_\X\)-a.e.\ on
\(\Y\). Then the quotient \(\mathscr M_{\varphi_*\mm_\X}\coloneqq\mathscr M/\sim\)
inherits a natural structure of normed \(L^0(\varphi_*\mm_\X)\)-module. Moreover, it is
easy to prove that \(\varphi^*\mathscr M_{\varphi_*\mm_\X}\cong\varphi^*\mathscr M\).
Hence, given a separable Banach \(\B\)-bundle \(\mathbf E\) over \(\Y\),
we can deduce from Theorem \ref{thm:pullback_sect_comm} that \(\varphi^*\Gamma(\mathbf E)
\cong\varphi^*\Gamma(\mathbf E)_{\varphi_*\mm_\X}\cong\Gamma(\varphi^*\mathbf E)\).
\end{itemize}
We omit the details and will not insist further on these observations.
\fr}\end{remark}
\subsection{The projection operator \texorpdfstring{\({\rm Pr}_\varphi\)}{Prphi}}
Let us begin by recalling a few basic notions and results about the projection
operator \({\rm Pr}_\varphi\) for bounded functions. Cf.\ \cite{GP16-2} for a
related discussion.
\begin{definition}[The operator \({\rm Pr}_\varphi\) for functions]
Let \((\X,\Sigma_\X,\mm_\X)\), \((\Y,\Sigma_\Y,\mm_\Y)\) be \(\sigma\)-finite
measure spaces. Let \(\varphi\colon\X\to\Y\) be a measurable map satisfying
\(\varphi_*\mm_\X=\mm_\Y\). Then we define the \emph{projection} operator
\({\rm Pr}_\varphi\colon L^\infty(\mm_\X)\to L^\infty(\mm_\Y)\) as the
linear, \(1\)-Lipschitz mapping
\[
{\rm Pr}_\varphi(f)\coloneqq\frac{\d\varphi_*(f^+\mm_\X)}{\d\mm_\Y}
-\frac{\d\varphi_*(f^-\mm_\X)}{\d\mm_\Y}\in L^\infty(\mm_\Y)
\quad\text{ for every }f\in L^\infty(\mm_\X),
\]
where \(\frac{\d\varphi_*(f^\pm\mm_\X)}{\d\mm_\Y}\) stands for the Radon--Nikod\'{y}m
derivative of \(\varphi_*(f^\pm\mm_\X)\) with respect to \(\mm_\Y\), while
\(f^+\coloneqq f\vee 0\) and \(f^-\coloneqq-f\wedge 0\) are the positive part and
the negative part of \(f\), respectively.
\end{definition}
Let us briefly comment on the well-posedness of \({\rm Pr}_\varphi\): given that
\[
\varphi_*(f^\pm\mm_\X)\leq\varphi_*\big(\|f\|_{L^\infty(\mm_\X)}\mm_\X\big)
=\|f\|_{L^\infty(\mm_\X)}\,\varphi_*\mm_\X=\|f\|_{L^\infty(\mm_\X)}\mm_\Y,
\]
we know that \(\varphi_*(f^\pm\mm_\X)\) is \(\sigma\)-finite and absolutely
continuous with respect to \(\mm_\Y\), so that the Radon--Nikod\'{y}m derivative
\(\frac{\d\varphi_*(f^\pm\mm_\X)}{\d\mm_\Y}\) exists. Moreover, the same estimate
shows that
\[\begin{split}
\big|{\rm Pr}_\varphi(f)\big|&=\bigg|\frac{\d\varphi_*(f^+\mm_\X)}{\d\mm_\Y}
-\frac{\d\varphi_*(f^-\mm_\X)}{\d\mm_\Y}\bigg|\leq
\frac{\d\varphi_*(f^+\mm_\X)}{\d\mm_\Y}+\frac{\d\varphi_*(f^-\mm_\X)}{\d\mm_\Y}\\
&=\frac{\d\varphi_*\big((f^+ +f^-)\,\mm_\X\big)}{\d\mm_\Y}=
\frac{\d\varphi_*(|f|\,\mm_\X)}{\d\mm_\Y}\leq\|f\|_{L^\infty(\mm_\X)}
\end{split}\]
holds \(\mm_\Y\)-a.e.\ on \(\Y\) for any given function \(f\in L^\infty(\mm_\X)\), thus
\(\big\|{\rm Pr}_\varphi(f)\big\|_{L^\infty(\mm_\Y)}\leq\|f\|_{L^\infty(\mm_\X)}\).
Finally, the linearity of \({\rm Pr}_\varphi\) can be directly checked from
its definition.
\begin{remark}{\rm
It is straightforward to check the validity of the following properties:
\begin{subequations}\begin{align}
\label{eq:extra_prop_Pr_phi1}
{\rm Pr}_\varphi(c)=c&\quad\text{ for every }c\in\R,\\
\label{eq:extra_prop_Pr_phi2}
{\rm Pr}_\varphi(f)\leq{\rm Pr}_\varphi(g)&\quad\text{ for every }
f,g\in L^\infty(\mm_\X)\text{ with }f\leq g,\\
\label{eq:extra_prop_Pr_phi3}
{\rm Pr}_\varphi(g\circ\varphi\,f)=g\,{\rm Pr}_\varphi(f)&
\quad\text{ for every }f\in L^\infty(\mm_\X)\text{ and }g\in L^\infty(\mm_\Y),
\end{align}\end{subequations}
where all equalities and inequalities are intended in the almost everywhere sense.
\fr}\end{remark}
\begin{example}{\rm
In general, the projection operator
\({\rm Pr}_\varphi\colon L^\infty(\mm_\X)\to L^\infty(\mm_\Y)\) cannot
be extended to a continuous map from \(L^0(\mm_\X)\) to \(L^0(\mm_\Y)\),
as shown by the following example.

Let us consider \(\X\coloneqq\N\) and \(\Y\coloneqq\{0\}\), endowed with the
Borel measures \(\mm_\X\coloneqq\sum_{i\in\N}2^{-i}\,\delta_i\)
and \(\mm_\Y\coloneqq\delta_0\), respectively. The unique map
\(\varphi\colon\X\to\Y\) (sending all elements to \(0\)) is Borel
and satisfies \(\varphi_*\mm_\X=\mm_\Y\). We argue by contradiction:
suppose there exists a continuous extension \(T\colon L^0(\mm_\X)\to L^0(\mm_\Y)\)
of \({\rm Pr}_\varphi\). Define \(f_n\coloneqq\sum_{i=1}^n 2^i\,\nchi_i
\in L^\infty(\mm_\X)\) for every \(n\in\N\) and
\(f\coloneqq\sum_{i\in\N}2^i\,\nchi_i\in L^0(\mm_\X)\). Notice that
\(f_n\to f\) in \(L^0(\mm_\X)\) as \(n\to\infty\) and \(L^0(\mm_\Y)\cong\R\).
Given that \({\rm Pr}_\varphi(f_n)=n\) for every \(n\in\N\), we deduce that
\[
T(f)=\lim_{n\to\infty}T(f_n)=\lim_{n\to\infty}{\rm Pr}_\varphi(f_n)=+\infty,
\]
which leads to a contradiction. Therefore, \({\rm Pr}_\varphi\) cannot be
extended to such a map \(T\).
\fr}\end{example}
In the metric measure space setting, the operator
\({\rm Pr}_\varphi\colon L^\infty(\mm_\X)\to L^\infty(\mm_\Y)\) can be
alternatively described (thanks to the disintegration theorem) in the following way:
\begin{proposition}[Characterisation of \({\rm Pr}_\varphi\) for functions]
\label{prop:char_Pr_phi_fcts}
Let \((\X,\sfd_\X,\mm_\X)\), \((\Y,\sfd_\Y,\mm_\Y)\) be metric measure spaces,
with \(\mm_\X\), \(\mm_\Y\) finite. Let \(\varphi\colon\X\to\Y\) be a Borel map
with \(\varphi_*\mm_\X=\mm_\Y\). Then for every function \(f\in L^\infty(\mm_\X)\)
it holds that
\begin{equation}\label{eq:char_Pr_phi_fcts_cl}
{\rm Pr}_\varphi(f)(y)=\int f\,\d\mm_\X^y\quad\text{ for }\mm_\Y\text{-a.e.\ }y\in\Y,
\end{equation}
where \(\mm_\X=\int\mm_\X^y\,\d\mm_\Y(y)\).
\end{proposition}
\begin{proof}
Given that \(\mm_\X\) is finite, it holds that
\({\rm Pr}_\varphi(f)=\frac{\d\varphi_*(f\mm_\X)}{\d\mm_\Y}\). Therefore, in order to
prove the statement it is sufficient to show that
\begin{equation}\label{eq:char_Pr_phi_fcts}
\varphi_*(f\mm_\X)(A)=\int_A\bigg(\int f\,\d\mm_\X^y\bigg)\d\mm_\Y(y)
\quad\text{ for every }A\subseteq\Y\text{ Borel.}
\end{equation}
Since for \(\mm_\Y\)-a.e.\ \(y\in\Y\) it holds \(\varphi(x)=y\) for
\(\mm_\X^y\)-a.e.\ \(x\in\X\) by \eqref{eq:disint1}, we have that
\[\begin{split}
\varphi_*(f\mm_\X)(A)&=\int\nchi_A\circ\varphi\,f\,\d\mm_\X\overset{\eqref{eq:disint2}}=
\int\bigg(\int\nchi_A\big(\varphi(x)\big)f(x)\,\d\mm_\X^y(x)\bigg)\d\mm_\Y(y)\\
&=\int_A\bigg(\int f\,\d\mm_\X^y\bigg)\d\mm_\Y(y),
\end{split}\]
thus proving \eqref{eq:char_Pr_phi_fcts}. Consequently, the proof is complete.
\end{proof}
We are now in a position to generalise the object \({\rm Pr}_\varphi\) to the framework
of normed modules. This construction has been first obtained in \cite{Gigli14} and later
studied in \cite{GP16-2}. Here we work with normed \(L^\infty\)-modules equipped with
a pointwise norm taking values in \(L^\infty\), thus we need to provide a slightly
different proof, but the ideas are essentially borrowed from \cite{Gigli14,GP16-2}.
\begin{theorem}[The operator \({\rm Pr}_\varphi\) for modules]\label{thm:Pr_phi_mod}
Let \((\X,\Sigma_\X,\mm_\X)\), \((\Y,\Sigma_\Y,\mm_\Y)\) be \(\sigma\)-finite
measure spaces. Let \(\varphi\colon\X\to\Y\) be a measurable map satisfying
\(\varphi_*\mm_\X=\mm_\Y\). Let \(\mathscr M^\infty\) be a normed
\(L^\infty(\mm_\Y)\)-module. Denote \(\mathscr M^0\coloneqq{\sf C}(\mathscr M^\infty)\).
Then there exists a unique linear and continuous operator
\({\rm Pr}_\varphi\colon{\sf R}(\varphi^*\mathscr M^0)\to\mathscr M^\infty\) --
called the \emph{projection operator} -- such that
\begin{equation}\label{eq:Pr_phi_mod1}
{\rm Pr}_\varphi(f\cdot\varphi^*v)={\rm Pr}_\varphi(f)\cdot v
\quad\text{ for every }f\in L^\infty(\mm_\X)\text{ and }v\in\mathscr M^\infty.
\end{equation}
Moreover, it holds that
\begin{equation}\label{eq:Pr_phi_mod2}
\big|{\rm Pr}_\varphi(w)\big|\leq{\rm Pr}_\varphi\big(|w|\big)\;\;\;\mm_\Y\text{-a.e.}
\quad\text{ for every }w\in{\sf R}(\varphi^*\mathscr M^0).
\end{equation}
\end{theorem}
\begin{proof}
Denote by \(\mathcal V\) the family of all those elements
\(w\in{\sf R}(\varphi^*\mathscr M^0)\) that can be written as
\(w=\sum_{n\in\N}[\nchi_{A_n}]_{\mm_\X}\cdot\varphi^*v_n\), for a partition
\((A_n)_{n\in\N}\subseteq\Sigma_\X\) of \(\X\) and a sequence
\((v_n)_{n\in\N}\subseteq\mathscr M^\infty\). Note that \(\mathcal V\) is a dense
linear subspace of \({\sf R}(\varphi^*\mathscr M^0)\). We are forced to set
\({\rm Pr}_\varphi\colon\mathcal V\to\mathscr M^\infty\) as
\[
{\rm Pr}_\varphi(w)\coloneqq\sum_{n\in\N}{\rm Pr}_\varphi\big([\nchi_{A_n}]_{\mm_\X}\big)
\cdot v_n\in\mathscr M^\infty\quad\text{ for every }
w=\sum_{n\in\N}[\nchi_{A_n}]_{\mm_\X}\cdot\varphi^*v_n\in\mathcal V,
\]
where the sum is intended with respect to the distance \(\sfd_{\mathscr M^0}\).
Let us check that such sum is actually well-defined: given any \(k\in\N\), we have the
\(\mm_\Y\)-a.e.\ inequality
\[\begin{split}
\sum_{n=1}^k\Big|{\rm Pr}_\varphi\big([\nchi_{A_n}]_{\mm_\X}\big)\cdot v_n\Big|
&=\sum_{n=1}^k{\rm Pr}_\varphi\big([\nchi_{A_n}]_{\mm_\X}\big)|v_n|
\overset{\eqref{eq:extra_prop_Pr_phi3}}=
\sum_{n=1}^k{\rm Pr}_\varphi\big([\nchi_{A_n}]_{\mm_\X}|v_n|\circ\varphi\big)\\
&={\rm Pr}_\varphi\bigg(\Big|\sum_{n=1}^k[\nchi_{A_n}]_{\mm_\X}\cdot
\varphi^*v_n\Big|\bigg)\overset{\eqref{eq:extra_prop_Pr_phi2}}\leq
{\rm Pr}_\varphi\big(|w|\big),
\end{split}\]
whence it follows that \(\sum_{n\in\N}\big|{\rm Pr}_\varphi
\big([\nchi_{A_n}]_{\mm_\X}\big)\cdot v_n\big|\leq{\rm Pr}_\varphi\big(|w|\big)\)
holds \(\mm_\Y\)-a.e.\ on \(\Y\). This grants that the sum
\(\sum_{n\in\N}{\rm Pr}_\varphi\big([\nchi_{A_n}]_{\mm_\X}\big)\cdot v_n\) exists in
\(\mathscr M^0\) and defines an element \({\rm Pr}_\varphi(w)\in\mathscr M^\infty\).
Moreover, the same estimates show that \eqref{eq:Pr_phi_mod2} holds for every
\(w\in\mathcal V\), thus \({\rm Pr}_\varphi\) can be uniquely extended to a linear and
continuous map \({\rm Pr}_\varphi\colon{\sf R}(\varphi^*\mathscr M^0)\to\mathscr M^\infty\),
which also satisfies \eqref{eq:Pr_phi_mod2}. By construction, the resulting map
\({\rm Pr}_\varphi\) is the unique linear and continuous operator satisfying
\eqref{eq:Pr_phi_mod1} for functions \(f\in L^\infty(\mm_\X)\) of the form
\(f=[\nchi_A]_{\mm_\X}\), where \(A\in\Sigma_\X\). Finally,
since simple functions are dense in \(L^\infty(\mm_\X)\),
one can easily deduce that \eqref{eq:Pr_phi_mod1} is verified.
\end{proof}
\begin{remark}{\rm
Observe that \(L^\infty(\mm_\X)\) and \(L^0(\mm_\X)\) have a natural structure
of normed \(L^\infty(\mm_\X)\)-module and normed \(L^0(\mm_\X)\)-module, respectively,
and \({\sf C}\big(L^\infty(\mm_\X)\big)=L^0(\mm_\X)\); the same holds for
\(L^\infty(\mm_\Y)\) and \(L^0(\mm_\Y)\). Moreover, the pullback map
\(\varphi^*\colon L^0(\mm_\Y)\to L^0(\mm_\X)\) is given by
\(\varphi^*g=g\circ\varphi\) for every \(g\in L^0(\mm_\Y)\).
Therefore, by using \eqref{eq:extra_prop_Pr_phi3} we deduce that
the operator \({\rm Pr}_\varphi\colon L^\infty(\mm_\X)\to L^\infty(\mm_\Y)\)
actually coincides with the projection operator between normed modules,
thus accordingly no ambiguity may arise.
\fr}\end{remark}
We conclude by showing that the projection operator \({\rm Pr}_\varphi\)
(associated with a separable normed module) can be characterised
in a fiberwise way, thus generalising Proposition \ref{prop:char_Pr_phi_fcts}.
\begin{proposition}[Characterisation of \({\rm Pr}_\varphi\) for modules]
\label{prop:char_Pr_phi_mod}
Let \((\X,\sfd_\X,\mm_\X)\), \((\Y,\sfd_\Y,\mm_\Y)\) be metric measure spaces,
with \(\mm_\X\), \(\mm_\Y\) finite. Let \(\varphi\colon\X\to\Y\) be a Borel map
with \(\varphi_*\mm_\X=\mm_\Y\). Let \(\B\) be a universal separable Banach space.
Let \(\mathbf E\) be a separable Banach \(\B\)-bundle over \(\Y\). Then the projection
operator \({\rm Pr}_\varphi\colon\Gamma_b(\varphi^*\mathbf E)\to\Gamma_b(\mathbf E)\)
is given by
\[
{\rm Pr}_\varphi(t)(y)=\int t(x)\,\d\mm_\X^y(x)\in\mathbf E(y)\quad\text{ for every }
t\in\Gamma_b(\varphi^*\mathbf E)\text{ and }\mm_\Y\text{-a.e.\ }y\in\Y,
\]
where \(\mm_\X=\int\mm_\X^y\,\d\mm_\Y(y)\).
\end{proposition}
\begin{proof}
First of all, let us observe that Theorem \ref{thm:pullback_sect_comm} gives
\(\Gamma_b(\varphi^*\mathbf E)={\sf R}\big(\Gamma(\varphi^*\mathbf E)\big)
\cong{\sf R}\big(\varphi^*\Gamma(\mathbf E)\big)\), thus the projection operator
\({\rm Pr}_\varphi\) associated with \(\Gamma(\mathbf E)\) can be seen as a
mapping from \(\Gamma_b(\varphi^*\mathbf E)\) to \(\Gamma_b(\mathbf E)\).
Given any \(t\in\Gamma_b(\varphi^*\mathbf E)\), we define
\(\Phi(t)(y)\coloneqq\int t\,\d\mm_\X^y\in\mathbf E(y)\) for \(\mm_\Y\)-a.e.\ \(y\in\Y\).
By arguing as in the proof of Theorem \ref{thm:pullback_sect_comm}, one can see
that \(\Phi(t)\in\Gamma_b(\mathbf E)\). The resulting map
\(\Phi\colon\Gamma_b(\varphi^*\mathbf E)\to\Gamma_b(\mathbf E)\) is linear and
\(1\)-Lipschitz. For any \(f\in L^\infty(\mm_\X)\) and \(s\in\Gamma_b(\mathbf E)\) we have
\[
\Phi(f\cdot\varphi^*s)(y)=\int f(x)\,s\big(\varphi(x)\big)\,\d\mm_\X^y(x)
=\bigg(\int f\,\d\mm_\X^y\bigg)s(y)\overset{\eqref{eq:char_Pr_phi_fcts_cl}}=
\big({\rm Pr}_\varphi(f)\cdot s\big)(y)
\]
for \(\mm_\Y\)-a.e.\ \(y\in\Y\), thus accordingly \(\Phi={\rm Pr}_\varphi\)
by Theorem \ref{thm:Pr_phi_mod}, which yields the statement.
\end{proof}
\sectionlinetwo{Black}{88} 
\bigskip

\noindent{\bf Acknowledgements.}
The authors would like to thank Nicola Gigli for having suggested the study
of liftings in the setting of normed modules. They would also like to
thank Tapio Rajala for his useful comments on \S\ref{s:sep_Ban_bundle},
as well as Milica Lu\v{c}i\'{c}, Ivana Vojnovi\'{c}, and Daniele Semola
for the careful reading of a preliminary version of the manuscript.
The second and third named authors acknowledge the support by the Academy
of Finland, project number 314789. The second named author also
acknowledges the support by the project 2017TEXA3H ``Gradient flows,
Optimal Transport and Metric Measure Structures'', funded by the Italian
Ministry of Research and University. The third named author also
acknowledges the support by the Balzan project led by Luigi Ambrosio.
\def\cprime{$'$} \def\cprime{$'$}


\begin{thebibliography}{10}

\bibitem{AliprantisBorder99}
{\sc C.~Aliprantis and K.~Border}, {\em Infinite {D}imensional {A}nalysis: {A}
  {H}itchhiker's {G}uide}, Studies in Economic Theory, Springer, 1999.

\bibitem{ACM14}
{\sc L.~Ambrosio, M.~Colombo, and S.~Di~Marino}, {\em Sobolev spaces in metric
  measure spaces: reflexivity and lower semicontinuity of slope},  (2015),
  pp.~1--58.
\newblock Variational methods for evolving objects.

\bibitem{Ambrosio-DiMarino14}
{\sc L.~Ambrosio and S.~Di~Marino}, {\em Equivalent definitions of {$BV$} space
  and of total variation on metric measure spaces}, J. Funct. Anal., 266
  (2014), pp.~4150--4188.

\bibitem{AmbrosioGigliSavare08}
{\sc L.~Ambrosio, N.~Gigli, and G.~Savar{\'e}}, {\em Gradient flows in metric
  spaces and in the space of probability measures}, Lectures in Mathematics ETH
  Z\"urich, Birkh\"auser Verlag, Basel, second~ed., 2008.

\bibitem{AmbrosioGigliSavare11}
\leavevmode\vrule height 2pt depth -1.6pt width 23pt, {\em Calculus and heat
  flow in metric measure spaces and applications to spaces with {R}icci bounds
  from below}, Invent. Math., 195 (2014), pp.~289--391.

\bibitem{Ben18}
{\sc L.~Benatti}, {\em A review of differential calculus on metric measure
  spaces via ${L}^0$-normed modules}, 2018.
\newblock Master's thesis, University of Trieste.

\bibitem{BP75}
{\sc C.~Bessaga and A.~Pe{\l}czy\'{n}ski}, {\em Selected topics in
  infinite-dimensional topology}, Polska Akademia Nauk. Instytut Matematyczny.
  Monografie matematyczne, 1975.

\bibitem{Cheeger00}
{\sc J.~Cheeger}, {\em Differentiability of {L}ipschitz functions on metric
  measure spaces}, Geom. Funct. Anal., 9 (1999), pp.~428--517.

\bibitem{Fremlin3}
{\sc D.~Fremlin}, {\em Measure {T}heory: {M}easure algebras. Volume 3}, Measure
  Theory, Torres Fremlin, 2011.

\bibitem{Gigli14}
{\sc N.~Gigli}, {\em Nonsmooth differential geometry - an approach tailored for
  spaces with {R}icci curvature bounded from below}.
\newblock Mem. Amer. Math. Soc., 251 (2018), pp.~161.

\bibitem{Gigli17}
\leavevmode\vrule height 2pt depth -1.6pt width 23pt, {\em {L}ecture notes on
  differential calculus on $\sf {R}{C}{D}$ spaces},  (2018).
\newblock Publ. RIMS Kyoto Univ. 54.

\bibitem{GP16-2}
{\sc N.~Gigli and E.~Pasqualetto}, {\em Behaviour of the reference measure on
  {${\rm RCD}$} spaces under charts}.
\newblock To appear on Communications in Analysis and Geometry, arXiv:1607.05188.

\bibitem{GP20}
\leavevmode\vrule height 2pt depth -1.6pt width 23pt, {\em Lectures on
  Nonsmooth Differential Geometry}, SISSA Springer Series 2, 2020.

\bibitem{GPS18}
{\sc N.~Gigli, E.~Pasqualetto, and E.~Soultanis}, {\em Differential of metric
  valued {S}obolev maps}, Journal of Functional Analysis, 278 (2020),
  p.~108403.

\bibitem{GR17}
{\sc N.~Gigli and C.~Rigoni}, {\em Recognizing the flat torus among
  ${R}{C}{D}^*(0,{N})$ spaces via the study of the first cohomology group},
  Calculus of Variations and Partial Differential Equations, 57(5) (2017).

\bibitem{GT20}
{\sc N.~Gigli and A.~Tyulenev}, {\em Korevaar-{S}choen's energy on strongly
  rectifiable spaces},  (2020).
\newblock Preprint, arXiv:2002.07440.

\bibitem{Gutman93}
{\sc A.~Gutman}, {\em Banach bundles in the theory of lattice-normed spaces.
  {I}{I}. {M}easurable {B}anach bundles}, Siberian Adv. Math., 3 (1993),
  pp.~8--40.

\bibitem{HLR91}
{\sc R.~Haydon, M.~Levy, and Y.~Raynaud}, {\em Randomly normed spaces}, vol.~41
  of Travaux en Cours [Works in Progress], Hermann, Paris, 1991.

\bibitem{LDLP19}
{\sc E.~Le~Donne, D.~Lu\v{c}i\'{c}, and E.~Pasqualetto}, {\em Universal
  infinitesimal {H}ilbertianity of sub-{R}iemannian manifolds},  (2019).
\newblock Preprint, arXiv:1910.05962.

\bibitem{LP18}
{\sc D.~Lu\v{c}i\'{c} and E.~Pasqualetto}, {\em The {S}erre-{S}wan theorem for
  normed modules}, Rendiconti del Circolo Matematico di Palermo Series 2, 68
  (2019), pp.~385--404.

\bibitem{Pas19}
{\sc E.~Pasqualetto}, {\em Direct and inverse limits of normed modules},
  (2019).
\newblock Preprint, arXiv:1902.04126.

\bibitem{Sauvageot89}
{\sc J.-L. Sauvageot}, {\em Tangent bimodule and locality for dissipative
  operators on {$C^*$}-algebras}, in Quantum probability and applications, {IV}
  ({R}ome, 1987), vol.~1396 of Lecture Notes in Math., Springer, Berlin, 1989,
  pp.~322--338.

\bibitem{Sauvageot90}
{\sc J.-L. Sauvageot}, {\em Quantum {D}irichlet forms, differential calculus
  and semigroups}, in Quantum probability and applications, {V} ({H}eidelberg,
  1988), vol.~1442 of Lecture Notes in Math., Springer, Berlin, 1990,
  pp.~334--346.

\bibitem{Shanmugalingam00}
{\sc N.~Shanmugalingam}, {\em Newtonian spaces: an extension of {S}obolev
  spaces to metric measure spaces}, Rev. Mat. Iberoamericana, 16 (2000),
  pp.~243--279.

\bibitem{SMM02}
{\sc W.~Strauss, N.~Macheras, and K.~Musia\l}, {\em Liftings}, Handbook on
  Measure Theory (E. Pap, ed.), Elsevier, Amsterdam, 2002.

\bibitem{Takesaki79}
{\sc M.~Takesaki}, {\em Theory of operator algebras. {I}}, Springer-Verlag, New
  York-Heidelberg, 1979.

\bibitem{Weaver01}
{\sc N.~Weaver}, {\em Lipschitz algebras and derivations. {II}. {E}xterior
  differentiation}, J. Funct. Anal., 178 (2000), pp.~64--112.

\end{thebibliography}
\end{document}